\documentclass[a4paper,reqno,11pt]{amsart}
\usepackage[a4paper,top=1in, bottom=1in, left=1in, right=1in]{geometry}
\usepackage{enumerate}
\usepackage{amsmath,amsfonts,amssymb,amsthm}
\usepackage[amsstyle]{cases}
\usepackage{mathrsfs}
\usepackage{stmaryrd}
\usepackage{graphicx}
\usepackage[dvipsnames]{xcolor}  
\usepackage[linkcolor=blue,citecolor=cyan,colorlinks]{hyperref}
\usepackage{cite}
\usepackage{pgfplots}
\pgfplotsset{compat=1.18}
\usepackage{tikz-3dplot}
\usetikzlibrary{shapes.geometric, calc, arrows.meta}

\allowdisplaybreaks

\newtheorem{theorem}{Theorem}[section]

\newtheorem{proposition}[theorem]{Proposition}
\newtheorem{lemma}[theorem]{Lemma}
\theoremstyle{definition}
\newtheorem{definition}[theorem]{Definition}
\theoremstyle{remark}
\newtheorem*{remark}{Remark}
\numberwithin{equation}{section}

\newcommand{\divop}{\operatorname{div}}
\newcommand{\curl}{\operatorname{curl}}
\newcommand{\peta}{\partial^\eta}
\newcommand{\symeta}{\mathcal{S}^\eta}
\renewcommand\leq{\leqslant}
\renewcommand\geq{\geqslant}
\renewcommand\le{\leqslant}
\renewcommand\ge{\geqslant}
\newcommand{\sym}{\mathcal{S}}

\begin{document}
		
		\title[Finite-Time Splash in Neo-Hookean Elastodynamics]{Finite-Time Splash in Free Boundary Problem of 3D Neo-Hookean Elastodynamics}
			\author{Wei Zhang}
\author{Chengchun Hao}	
\author{Jie Fu}

		\address{Wei Zhang\newline School of Mathematics and Statistics,Hefei University, Hefei, Anhui, 230000, China}
		\email{zhangwei16@mails.ucas.edu.cn}

\address{Chengchun Hao\newline	State Key Laboratory of Mathematical Sciences, Academy of Mathematics and Systems Science, Chinese Academy of Sciences, Beijing 100190, China\newline
	\and
	School of Mathematical Sciences,
	University of Chinese Academy of Sciences,
	Beijing 100049, China}
\email{hcc@amss.ac.cn}
\address{Jie Fu\newline School of Mathematical Sciences, Xiamen University, xiamen, Fujian 361005, China}
\email{fujie@amss.ac.cn}			
		
\begin{abstract}
This paper establishes finite-time splash singularity formation for 3D viscous incompressible neo-Hookean elastodynamics with free boundaries. The system features mixed stress-kinematic conditions where viscous-elastic stresses balance pressure forces at the evolving interface --- a configuration generating complex boundary integrals that distinguish it from Navier-Stokes or MHD systems. To address this challenge, we employ a Lagrangian framework inspired by Coutand and Shkoller (2019), developing specialized coordinate charts and constructing a sequence of shrinking initial domains with cylindrical necks connecting hemispherical regions to bases. Divergence-free initial velocity and deformation tensor fields are designed to satisfy exact mechanical compatibility. Uniform a priori estimates across the domain sequence demonstrate that interface evolution preserves local smoothness while developing finite-time self-intersection. Energy conservation provides foundational stability, while higher-order energy functionals yield scaling-invariant regularity control. The analysis proves inevitable splash singularity formation within explicitly bounded time, maintaining spatial smoothness near the singular point up to the intersection time.
\end{abstract}

					

\maketitle

\section{Introduction}

We consider the free boundary problem for incompressible viscoelastic neo-Hookean materials:
\begin{equation}\label{1.1}
	\begin{cases}
		u_{t}+u\cdot\nabla u-\Delta u+\nabla p=\divop (FF^{\top}) \qquad  &
		\text{in}\ \Omega(t),\\
		F_{t}+u\cdot\nabla F=\nabla u F   \qquad  &\text{in} \ \Omega(t), \\
		\divop  u=0   \qquad &\text{in} \ \Omega(t),\\
		\divop  F^\top=0  \qquad &  \text{in}\ \Omega(t).
	\end{cases}
\end{equation}
The motion of the material occupying a time-evolving domain $\Omega(t)\subset\mathbb{R}^{3}$ over $t\in[0,T]$ is governed by the velocity field $u=(u_{1},u_{2},u_{3})$ and pressure $p$, with the incompressibility constraint \eqref{1.1}. The deformation tensor $F=(F_{ij})$ and its transpose $F^{\top}=(F_{ji})$ are fundamental to the constitutive description, where the Cauchy-Green tensor $FF^{\top}$ characterizes the neo-Hookean material response to isochoric deformations. Normalization with unit physical constants preserves essential nonlinear deformation-momentum couplings while simplifying the equations.

The doubled symmetric deformation tensor $\mathcal{S}(u)$, defined component-wise as
\[
(\mathcal{S}(u))_{ij}=\partial_i u_j+\partial_j u_i,
\]
quantifies strain-rate dissipation in viscoelastic continua. The boundary $\partial\Omega(t)$ evolves with normal velocity $V(\partial\Omega(t)) = u \cdot N$, where $N(t,x)$ is the outward unit normal, subject to the following conditions (cf. \cite{DMS}):
\begin{equation}\label{1.2}
	\begin{cases}
		\left((\sym(u)-p\mathrm{I})+(FF^{\top}-\mathrm{I})\right)N=0 \quad \text{on }  \partial\Omega(t),  \\
		V(\partial\Omega(t))=u\cdot N,  \\
		\Omega(t)|_{t=0}=\Omega_{0}, \quad (u,F)|_{t=0}=(u_{0},F_{0}),
	\end{cases}
\end{equation}
where $\mathrm{I}$ denotes the $3 \times 3$ identity matrix. The stress equilibrium balances viscous-elastic stresses against pressure forces at the free boundary, while the kinematic condition ensures geometric consistency. The initial data  $(\Omega_{0}, u_{0}, F_{0})$ satisfy the compatibility conditions $\det F_0=1$, $\divop  u_0=0$  and $\divop F_0^{\top} = 0$ to maintain finite-strain kinematics.  The coupled operator $(\mathcal{S}(u)-p\mathrm{I})+(FF^{\top}-\mathrm{I})$ embodies the neo-Hookean constitutive law.

The splash singularity phenomenon has been studied in several key contexts.  Castro, C\'ordoba, Fefferman, Gancedo, and G\'omez first established its existence for the 2D Navier-Stokes equations \cite{CCFGG2}. Coutand and Shkoller later extended this result to special domains in both 2D and 3D \cite{CS5}.

For elastodynamics, the tangentiality condition $F^{\top} \cdot N = 0$ at free boundaries emerges naturally from  the constraint $\operatorname{div} F^{\top} = 0$. Hao and Wang derived \textit{a priori} estimates for this classical boundary configuration ($p=0$ and $F^{\top} \cdot N = 0$) in the incompressible case \cite{HW}. Subsequent developments established local well-posedness for both incompressible \cite{ZY} and compressible \cite{ZJ} elastodynamics, extended to mixed-type stability conditions \cite{GW}, and yielded blow-up criteria for incompressible elastodynamics \cite{FHYZ}.

In this work, system \eqref{1.1} is complemented by the boundary conditions \eqref{1.2}, representing static force equilibrium at the interface. Significant prior results include: Gu and Lei's proof of local well-posedness for incompressible elastodynamics with surface tension ($\sigma > 0$) \cite{GL}, and Di Iorio, Marcati, and Spirito's analysis of the non-tension case ($\sigma = 0$) \cite{DMS}, both studying the boundary condition:
\begin{align*}
	\left[-p\mathrm{I} + (FF^\top - \mathrm{I}) +\nu(\nabla v+\nabla v^{\top})\right]N = \sigma H N.
\end{align*}

A crucial mathematical connection emerges when $F$ degenerates to a rank-$1$ matrix with single non-zero column $H \in \mathbb{R}^3$. Here, $FF^{\top}$ reduces to the rank-$1$ operator $H \otimes H$, causing \eqref{1.1}-\eqref{1.2} to formally coincide with incompressible magnetohydrodynamics (MHD) where $H$ represents the magnetic field. This correspondence is significant because splash singularities are well-established for such MHD systems --- in 2D \cite{HY} and recently in 3D \cite{HLZ}. However, unlike Navier-Stokes in \cite{CS5} or MHD in \cite{HLZ}, the mixed condition \eqref{1.2} for elastodynamics generates intricate boundary integrals and restricts choices of initial data, which demands sharper tools to handle deformation tensor evolution and coupled stress conditions.
Let $D_t:=\partial_t+u\cdot\nabla$ denote the material derivative. Define the spatial gradient operator $\nabla = (\partial_{1} , \partial_{2}, \partial_{3})$ such that for scalar fields $f$ and vector fields $g = (g_1, g_2, g_3)$:
\begin{align*}
	\nabla f := (\partial_1 f, \partial_2 f, \partial_3 f), \quad
	\nabla^{\top} g := (\nabla^{\top} g_1, \nabla^{\top} g_2, \nabla^{\top} g_3).
\end{align*}
We employ the Einstein-type summation convention with Latin indices ($i,j,k,\ldots$) ranging over $\{1,2,3\}$ for bulk quantities, and Greek indices ($\alpha,\beta,\gamma,\ldots$) over $\{1,2\}$ for surface operations. This indexing facilitates unified treatment of bulk-surface coupling.

Key differential operators are defined as:
\begin{align*}
	(\nabla u F)_{ij} &:= (\nabla u)_{ik} F_{kj}, \\
	\divop u &:= \partial_i u_i, \\
	(\divop F^\top)_i &:= \partial_j F_{ji}, \\
	(\divop (FF^\top))_i &:= \partial_j (F_{ik} F_{jk}).
\end{align*}
The system \eqref{1.1}-\eqref{1.2} can now be expressed as:
\begin{equation}\label{1.3}
	\begin{cases}
		D_{t}u_{i}-\partial_k\partial_k u_i+\partial_{i}p=\partial_{j}F_{ik}F_{jk} \qquad &\text{in}\  \Omega(t), \\
		D_{t}F_{ij}=\partial_{k}u_{i}F_{kj} \qquad &\text{in}\  \Omega(t), \\
		\partial_{k}u_{k}=0, \quad \partial_{j}F_{ji}=0   \qquad & \text{in}\  \Omega(t), \\ ((\partial_j u_i+\partial_i u_j)-p\delta_{ij}+F_{ik}F_{jk}-\delta_{ij})N_{j}=0 \qquad &  \text{on} \ \partial\Omega(t), \\
		V(\partial\Omega(t))=u_{i}N_{i},
	\end{cases}
\end{equation}
where $\delta_{ij}$ is the Kronecker delta.

Now, we recall the precise definition of a ‌splash singularity in the context of fluid interface dynamics (cf. \cite{CS3}):

\begin{definition}[Splash Singularity]
	For a time-evolving fluid interface $\partial\Omega(t)$ that remains locally smooth in space and time, a ‌splash singularity occurs at time $T<\infty$ if $\partial\Omega(T)$ develops a self-intersection at a point while retaining local smoothness near that point.	
\end{definition}

The central objective of this work is to demonstrate that smooth initial data exist for system \eqref{1.1} which leads to finite-time splash singularity formation. Our main result is:

\begin{theorem}[Finite-Time Splash Formation] \label{thm1.1}
	There exists an initial data set $(\Omega_{0}, u_{0}, F_{0})$ for the incompressible neo-Hookean elastodynamics system \eqref{1.1}-\eqref{1.2} satisfying:
	\begin{enumerate}[(1)]
		\item $\Omega_{0} \subset \mathbb{R}^{3}$ is a bounded $C^\infty$-domain.
		\item The initial velocity $u_0$ and deformation tensor $F_0$ are smooth and satisfy $\divop u_0 = 0$ and  $\divop F_0^\top = 0$ with compatibility conditions: $\left[\sym(u_0) N_0\right] \times N_0 = 0$, $			(F_0^{\top}-I)\cdot N_0 = 0$ on $\partial \Omega_{0}$, where $N_0$ denotes the outward unit normal.
	\end{enumerate}
	Then for the corresponding solution $(\Omega(t), u, F)$, there exists $T^{*} > 0$ such that
	$\partial\Omega(T^*)$  undergoes self-intersection at  $x_0 \in \mathbb{R}^3$  (splash point), while maintaining  $C^\infty$-regularity locally near  $x_0$.
	
\end{theorem}

The paper is structured as follows: Section \ref{sec3} first constructs the geometric framework for singularity formation, developing specialized initial domains with cylindrical necks connecting hemispherical regions to bases and designing compatible divergence-free velocity and deformation tensor fields. Building upon this geometric foundation, Section \ref{sec2} introduces the Lagrangian coordinate framework with custom charts that reformulate the governing equations in material coordinates, establishing key commutator identities and transformation rules for differential operators. Section \ref{sec.conserved} then establishes the conserved physical energy that provides fundamental stability throughout the evolution. This enables the crucial a priori estimates in Section \ref{sec4}, where boundary regularity control, deformation tensor bounds, and higher-order energy functionals are developed with uniform constants across the domain sequence. Section \ref{sec6} establishes quantitative continuity for second-order tangential derivatives. Finally, Section \ref{sec5} synthesizes these components to demonstrate finite-time splash singularity formation through quantitative trajectory analysis, proving interface self-intersection occurs at an explicitly bounded time while preserving local smoothness until the singular moment.

\section{Construction of the Initial Data}\label{sec3}

\subsection{Initial Domain}
\begin{definition}[Domain $\Omega$]\label{defA.1}
	Let $\Omega \subset \mathbb{R}^3$ be a smooth bounded domain with boundary $\partial\Omega$, composed of three distinct open regions as shown in Figure \ref{fig:domain}:
	\begin{enumerate}
		\item $\omega$: Vertical cylinder of radius 1 and height $h$:
		\[
		\omega = \{(x_1,x_2,x_3) : x_1^2 + x_2^2 < 1,\ 2 < x_3 < 2+h\}.
		\]
		\item $\omega_+$: Lower hemisphere of radius 1 centered at $(0,0,2)$:
		\[
		\omega_+ = \{(x_1,x_2,x_3) : x_1^2 + x_2^2 + (x_3-2)^2 < 1,\ x_3 < 2\}
		\]
		with south pole at $X_+ = (0,0,1)$.
		\item $\omega_-$: Base region connecting to $\omega_+$ at $X_+$, extending downward with $\partial\omega_- \cap \partial\Omega \subset \{x_3=0\}$ and maximal $x_3$-coordinate 0.
	\end{enumerate}
	Coordinate assignments:
	\begin{itemize}
		\item Origin at $(0,0,0) \in \partial\omega_- \cap \partial\Omega \subset \{x_3=0\}$.
		\item $X_+ = (0,0,1)$ (south pole of $\omega_+$).
		\item Hemisphere top boundary: $\{(x_1,x_2,2) : x_1^2 + x_2^2 < 1\}$.
		\item Cylinder: $\{(x_1,x_2,x_3) : x_1^2 + x_2^2 < 1,\ 2 < x_3 < 2+h\}$.
	\end{itemize}
\end{definition}

\begin{figure}[h]
	\centering
	\begin{tikzpicture}[scale=0.4]
		\draw[color=green!80!black, ultra thick] (5,1.5) arc (-180:0:1cm);
		\draw[color=Green, dashed] plot[smooth,tension=.6] coordinates{   (5,1.5) (7,1.5)  };
		\draw[color=blue, dashed] plot[smooth,tension=.6] coordinates{   (5,3) (7,3)  };
		\draw (6,2) node { $ \omega  $};
		\draw (6,1) node { $ \omega_+  $};
		\draw (5,-3) node { $ \omega_-  $};
		\filldraw (6,0.5) circle (3pt) node[below] {$X_+$};
		
		\draw[color=blue,ultra thick] (5,3) arc (00:180:2cm);
		\draw[color=blue,ultra thick] (7,3) arc (00:180:4.5cm);
		
		\draw[color=blue,ultra thick] plot[smooth,tension=.6] coordinates{   (1,0) (1,2) (1,3) };
		\draw[color=blue,ultra thick] plot[smooth,tension=.6] coordinates{   (-2,-3) (-2,2) (-2,3) };

		\draw[color=blue,ultra thick] plot[smooth,tension=.6] coordinates{   (5,1.5) (5,2) (5,3) };
		\draw[color=blue,ultra thick] plot[smooth,tension=.6] coordinates{  (7,3) (7,2) (7,1.5) };
		
		\draw[color=blue,ultra thick] (1,0) arc (180:270:1cm);
		\draw[color=blue,ultra thick] (-2,-3) arc (180:270:1.5cm);
		
		\draw[color=red!40,ultra thick] plot[smooth,tension=.6] coordinates{   (2,-1) (8,-1)  };
		\draw[color=red!40,ultra thick] plot[smooth,tension=.6] coordinates{   (2,-4.5) (8,-4.5)  };
		\draw[color=blue,ultra thick] plot[smooth,tension=.6] coordinates{   (-0.5,-4.5) (2,-4.5)  };
		
		\draw[color=blue,ultra thick] (8,-1) arc (90:0:1cm);
		\draw[color=blue,ultra thick] (8,-4.5) arc (-90:0:1cm);
		
		\draw[color=blue,ultra thick] plot[smooth,tension=.6] coordinates{   (9,-2) (9,-3.5)  };
		
		\draw (0,-2.5) node { $\Omega $};
		\draw (7.3,5) node { $\partial\Omega $};
		\draw(2.5,0.5) -- (4.5,0.5);
		\draw[->] (3.5,2.5) -- (3.5,0.6);
		\draw[->] (3.5,-3.) -- (3.5,-1.1);
		\draw (3.5,-.25) node { $1$};
		\draw[color=red!40, dashed] plot[smooth,tension=.6] coordinates{   (2,-1) (2,-4.5)  };
		\draw[color=red!40, dashed] plot[smooth,tension=.6] coordinates{   (8,-1) (8,-4.5)  };
	\end{tikzpicture}
	\tdplotsetmaincoords{45}{0} 
	\begin{tikzpicture}[scale=0.8, every node/.style={font=\small}, tdplot_main_coords]
		\def\eps{1} 
		\def\h{3}     
		\def\r{1}   
		\draw[->, thick] (-3,0,0) -- (4,0,0) node[anchor=north east]{$x_1$};
		\draw[->, thick] (0,0,-3) -- (0,0,5.5+\eps/2) node[anchor=north west]{$x_3$};
		\draw[->, thick] (-2,-2,0) -- (2,2,0) node[anchor=south west]{$x_2$};

		\begin{scope}
			\fill[red!20, opacity=0.4] (4,1.5,0) -- (-0.4,1.5,0) -- (-4,-1.5,0) -- (0.4,-1.5,0) -- cycle;
			\fill[red!20, opacity=0.4] (4,1.5,-1) -- (-0.4,1.5,-1) -- (-4,-1.5,-1) -- (0.4,-1.5,-1) -- cycle;
			\fill[red!20, opacity=0.4] (4,1.5,0) -- (4,1.5,-1)  -- (-0.4,1.5,-1) -- (-0.4,1.5,0) -- cycle;
			\fill[red!20, opacity=0.4] (-4,-1.5,0) -- (-4,-1.5,-1)  -- (-0.4,1.5,-1) -- (-0.4,1.5,0) -- cycle;
			\draw[red!20, dashed] (-0.4,1.5,0) -- (-0.4,1.5,-1);
			\draw[red!20, thick] (0.4,-1.5,0) -- (0.4,-1.5,-1);	
			\draw[red!20, dashed] (-0.4,1.5,-1) -- (4,1.5,-1);
			\draw[red!20, dashed] (-0.4,1.5,-1) -- (-4,-1.5,-1);
		\end{scope}
		\node at (0.5,-0.2,-0.5) {$\omega_-$};

		\begin{scope} [shift={(0,0,\eps+\r)}]
			\draw[fill=green!30, opacity=0.7] (-\r,0,0) arc (180:360:{\r}) -- cycle;
			\fill[green!30, opacity=0.7]
			(-\r,0,0) arc (180:360:{\r} and 0.5)
			arc (0:-180:{\r} and 0.5);
			\draw[green!80!black, thick]
			(-\r,0,0) arc (180:360:{\r} and 0.5)
			arc (0:-180:{\r} and 0.5);
			\draw[green!80!black, thick] (-\r,0,0) arc (180:360:{\r}) -- cycle;
		\end{scope}
		\node at (0.2,\eps+0.3,0) {$\omega_+$};
		
		\fill[green!40!black] (0,0,\eps) circle (1pt) node[below right] {$X_+$};
		
		\begin{scope}[shift={(0,0,\eps+\r)}]
			\draw[fill=blue!30, opacity=0.7] (0,0,0) ellipse ({\r} and 0.5);
			\draw[fill=blue!30, opacity=0.7] (0,\h,0) ellipse ({\r} and 0.5);
			\draw[blue!80!black, thick]
			(-\r,0,0) -- (-\r,\h,0)
			(\r,0,0) -- (\r,\h,0);
			\draw[blue!80!black, thick] (0,\h,0) ellipse ({\r} and 0.5);
		\end{scope}
		\node at (0.5,\h/2,1.5) {$\omega$};
		\node at (0.05,0,\eps/2) {$1$};
	\end{tikzpicture}
	
	\caption{Cross-section of domain $\Omega$ showing: Cylinder $\omega$, hemisphere $\omega_+$ with south pole $X_+$, and base $\omega_-$. }
	\label{fig:domain}
\end{figure}
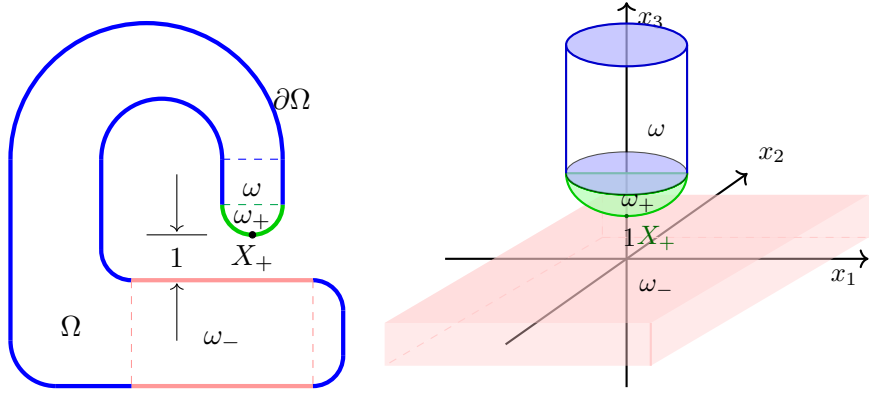

After that we define the ``Sequence of initial domain" $\Omega^{\epsilon}$.
\begin{definition}[Initial Domains $\Omega^{\epsilon}$]\label{defA.2}
	For $0 < \epsilon \ll 1$, let $\Omega \subset \mathbb{R}^3$ be the domain from Definition \ref{defA.1}. The sequence $\Omega^{\epsilon}$ is defined with the following modifications as shown in Figure \ref{fig.epsdomain}:
	\begin{enumerate}
		\item $\omega^{\epsilon}$: Vertically dilated cylinder with height $h + 1 - \epsilon$:
		\[
		\omega^{\epsilon} = \{(x_1,x_2,x_3) : x_1^2 + x_2^2 < 1, 1 + \epsilon < x_3 < 1 + \epsilon + h\}.
		\]
		
		\item $\omega^{\epsilon}_+$: Translated hemisphere centered at $(0,0,1+\epsilon)$:
		\[
		\omega^{\epsilon}_+ = \{(x_1,x_2,x_3) : x_1^2 + x_2^2 + (x_3 - 1 - \epsilon)^2 < 1, x_3 < 1 + \epsilon\}
		\]
		with south pole at $X^{\epsilon}_+ = (0,0,\epsilon)$.
		
		\item $\omega^{\epsilon}_-$: Base region satisfying:
		\begin{itemize}
			\item Located below $X^{\epsilon}_+$ with top boundary at $x_3 = 0$.
			\item $\partial\omega^{\epsilon}_- \cap \partial\Omega^{\epsilon}$ contains $\{(x_1,x_2,0) : x_1^2 + x_2^2 < \epsilon\}$.
			\item Connects smoothly to $\omega^{\epsilon}_+$ at $X^{\epsilon}_+$.
		\end{itemize}
	\end{enumerate}
	Coordinate assignments:
	\begin{itemize}
		\item Origin: $X_-:=(0,0,0) \in \partial\omega^{\epsilon}_- \cap \partial\Omega^{\epsilon} \subset \{x_3 = 0\}$.
		\item $X^{\epsilon}_+ = (0,0,\epsilon)$ (south pole).
		\item Hemisphere top boundary: $\{(x_1,x_2,1+\epsilon) : x_1^2 + x_2^2 < 1\}$.
		\item Cylinder: $\{(x_1,x_2,x_3) : x_1^2 + x_2^2 < 1, 1 + \epsilon < x_3 < 1 + \epsilon + h\}$.
	\end{itemize}
\end{definition}

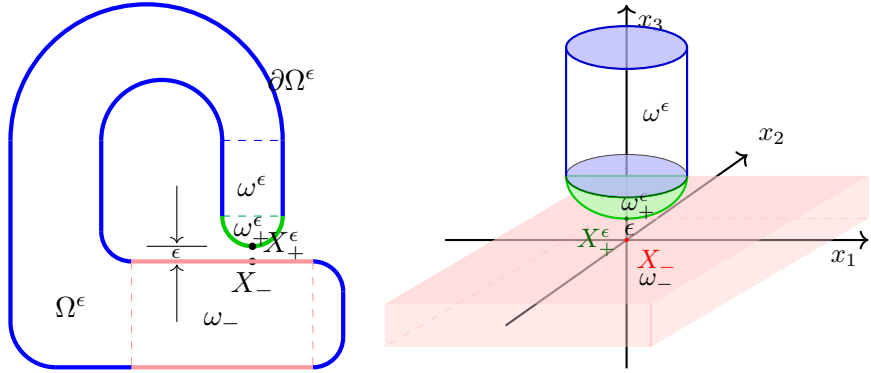
\begin{figure}[h]
	\begin{center}
		
		\begin{tikzpicture}[scale=0.4]
			
			\draw[color=green!80!black,ultra thick] (19,0.5) arc (-180:0:1cm);
			\draw[color=Green, dashed] plot[smooth,tension=.6] coordinates{   (19,0.5) (21,0.5)  };
			\draw[color=blue, dashed] plot[smooth,tension=.6] coordinates{   (19,3) (21,3)  };
			\draw (20,1.5) node { $ \omega^ \epsilon   $};
			\draw (20,0) node { $ \omega_+^ \epsilon   $};
			\draw (19,-3) node { $ \omega_-   $};
			\filldraw (20,-0.5) circle (3pt) node[right] {$X_+^\epsilon$};
			\filldraw (20,-1) circle (3pt) node[below] {$X_-$};
			
			\draw[color=blue,ultra thick] (19,3) arc (00:180:2cm);
			\draw[color=blue,ultra thick] (21,3) arc (00:180:4.5cm);
			
			\draw[color=blue,ultra thick] plot[smooth,tension=.6] coordinates{   (15,0) (15,2) (15,3) };
			\draw[color=blue,ultra thick] plot[smooth,tension=.6] coordinates{   (12,-3) (12,2) (12,3) };

			\draw[color=blue,ultra thick] plot[smooth,tension=.6] coordinates{   (19,0.5) (19,2) (19,3) };
			\draw[color=blue,ultra thick] plot[smooth,tension=.6] coordinates{  (21,3) (21,2) (21,0.5) };
			
			\draw[color=blue,ultra thick] (15,0) arc (180:270:1cm);
			\draw[color=blue,ultra thick] (12,-3) arc (180:270:1.5cm);
			
			\draw[color=red!40,ultra thick] plot[smooth,tension=.6] coordinates{   (16,-1) (22,-1)  };
			\draw[color=red!40,ultra thick] plot[smooth,tension=.6] coordinates{   (16,-4.5) (22,-4.5)  };
			\draw[color=blue,ultra thick] plot[smooth,tension=.6] coordinates{   (13.5,-4.5) (16,-4.5)  };
			
			\draw[color=blue,ultra thick] (22,-1) arc (90:0:1cm);
			\draw[color=blue,ultra thick] (22,-4.5) arc (-90:0:1cm);
			
			\draw[color=blue,ultra thick] plot[smooth,tension=.6] coordinates{   (23,-2) (23,-3.5)  };
			
			\draw (14,-2.5) node { $\Omega^ \epsilon $};
			\draw (21.3,5) node { $\partial\Omega^ \epsilon $};

			\draw(16.5,-0.5) -- (18.5,-0.5);
			\draw[->] (17.5,1.5) -- (17.5,-0.4);
			\draw[->] (17.5,-3.) -- (17.5,-1.1);
			\draw (17.5,-.75) node { $_{ \epsilon }$};

			\draw[color=red!40, dashed] plot[smooth,tension=.6] coordinates{   (16,-1) (16,-4.5)  };
			\draw[color=red!40, dashed] plot[smooth,tension=.6] coordinates{   (22,-1) (22,-4.5)  };
		\end{tikzpicture}
		\tdplotsetmaincoords{45}{0} 
		\begin{tikzpicture}[scale=0.8, every node/.style={font=\small}, tdplot_main_coords]
			\draw[->, thick] (-3,0,0) -- (4,0,0) node[anchor=north east]{$x_1$};
			\draw[->, thick] (0,0,-3) -- (0,0,5.5) node[anchor=north west]{$x_3$};
			\draw[->, thick] (-2,-2,0) -- (2,2,0) node[anchor=south west]{$x_2$};

			\def\eps{0.5} 
			\def\h{3}     
			\def\r{1}   
			
			\begin{scope}
				\fill[red!20, opacity=0.4] (4,1.5,0) -- (-0.4,1.5,0) -- (-4,-1.5,0) -- (0.4,-1.5,0) -- cycle;
				\fill[red!20, opacity=0.4] (4,1.5,-1) -- (-0.4,1.5,-1) -- (-4,-1.5,-1) -- (0.4,-1.5,-1) -- cycle;
				\fill[red!20, opacity=0.4] (4,1.5,0) -- (4,1.5,-1)  -- (-0.4,1.5,-1) -- (-0.4,1.5,0) -- cycle;
				\fill[red!20, opacity=0.4] (-4,-1.5,0) -- (-4,-1.5,-1)  -- (-0.4,1.5,-1) -- (-0.4,1.5,0) -- cycle;
				\draw[red!20, dashed] (-0.4,1.5,0) -- (-0.4,1.5,-1);
				\draw[red!20, thick] (0.4,-1.5,0) -- (0.4,-1.5,-1);	
				\draw[red!20, dashed] (-0.4,1.5,-1) -- (4,1.5,-1);
				\draw[red!20, dashed] (-0.4,1.5,-1) -- (-4,-1.5,-1);
			\end{scope}
			\node at (0.5,-0.2,-0.8) {$\omega_-$};	
			\fill[red] (0,0,0) circle (1pt) node[below right] {$X_-$};
			
			\begin{scope} [shift={(0,0,\eps+\r)}]
				\draw[fill=green!30, opacity=0.7] (-\r,0,0) arc (180:360:{\r}) -- cycle;
				\fill[green!30, opacity=0.7]
				(-\r,0,0) arc (180:360:{\r} and 0.5)
				arc (0:-180:{\r} and 0.5);
				\draw[green!80!black, thick]
				(-\r,0,0) arc (180:360:{\r} and 0.5)
				arc (0:-180:{\r} and 0.5);
				\draw[green!80!black, thick] (-\r,0,0) arc (180:360:{\r}) -- cycle;
			\end{scope}
			\node at (0.2,\eps+0.3,0) {$\omega_+^\epsilon$};
			
			\fill[green!40!black] (0,\eps,0) circle (1pt) node[below left] {$X_+^\epsilon$};
			
			\begin{scope}[shift={(0,0,\eps+\r)}]
				\draw[fill=blue!30, opacity=0.7] (0,0,0) ellipse ({\r} and 0.5);
				\draw[fill=blue!30, opacity=0.7] (0,\h,0) ellipse ({\r} and 0.5);
				\draw[blue!80!black, thick]
				(-\r,0,0) -- (-\r,\h,0)
				(\r,0,0) -- (\r,\h,0);
				\draw[blue!80!black, thick] (0,\h,0) ellipse ({\r} and 0.5);
			\end{scope}
			\node at (0.5,\h/2,1.5) {$\omega^\epsilon$};
			\node at (0.05,0,\eps/2) {$\epsilon$};
		\end{tikzpicture}
		\caption{Sequence of initial domain $\Omega^{\epsilon}$}\label{fig.epsdomain}			
	\end{center}
\end{figure}

\subsection{Local coordinate charts for $\Omega$ and $\Omega^{\epsilon}$}
\subsubsection{Local Coordinate Charts for $\Omega$}
We construct an atlas of coordinate charts as follows:

\underline{Boundary Charts:}
Let $\{U_l\}_{l=1}^K$ be an open covering of $\partial\Omega$ where each chart $\theta^l : B(0,1) \to U_l$ is a $C^\infty$ diffeomorphism satisfying:
\begin{align*}
	\theta^l(B^+) &= U_l \cap \Omega, \\
	\theta^l(B^0) &= U_l \cap \partial\Omega, \\
	\det \nabla \theta^l &= C_l > 0, \quad \text{(constant)}
\end{align*}
with $B = B(0,1)$, $B^+ = B \cap \{x_3 > 0\}$, and $B^0 = \overline{B} \cap \{x_3 = 0\}$.

To classify these charts, we introduce length scales:
\begin{align*}
	\delta_1 = \frac{h^2}{15(h+3)}, \quad
	\delta_2 = \frac{h(15+4h)}{15(h+3)} < \frac{h}{3},
\end{align*}
satisfying $0 < \delta_1 < \delta_2 < h/3$, which requires $h < 5$.

The neck region $\omega = \{(x_1,x_2,x_3) : x_1^2 + x_2^2 < 1,\ 2 < x_3 < 2+h\}$ is partitioned vertically:
\begin{align} \label{A.1}
	\begin{split}
		\text{Lower segment:} & \quad 2 <x_3< 2 + h/3. \\
		\text{Middle segment:} & \quad 2 + h/3 \leq x_3 \leq 2 + 2h/3. \\
		\text{Upper segment:} & \quad 2 + 2h/3 < x_3 < 2 + h.
	\end{split}
\end{align}
The middle segment contains a critical subcylinder:
\begin{align} \label{A.2}
	\mathcal{C}_r := \omega \cap \{2 + h/3 + \delta_1 < x_3 < 2 + h/3 + \delta_2\}.
\end{align}

The boundary charts are categorized into three disjoint classes:
\begin{enumerate}
	\item \textbf{Neck charts} ($1 \leq l \leq K_1$): Cover $\mathcal{C}_r$ and satisfy:
	\begin{align} \label{A.3}
		\mathcal{C}_r \subset \bigcup_{l=1}^{K_1} \theta^l(B^+) \subset \omega \cap \{2 + h/3 < x_3 < 2 + 2h/3\}.
	\end{align}
	
	\item \textbf{Non-hemisphere charts} ($K_1+1 \leq l \leq K_2$): Disjoint from $\omega_+$ and $\mathcal{C}_r$:
	\begin{align} \label{A.4}
		\theta^l(B^+) \cap \mathcal{C}_r = \varnothing, \quad \theta^l(B^+) \cap \omega_+ = \varnothing.
	\end{align}
	
	\item \textbf{Hemisphere charts} ($K_2+1 \leq l \leq K$): Intersect $\omega_+$ but disjoint from $\mathcal{C}_r$:
	\begin{align} \label{A.5}
		\theta^l(B^+) \cap \mathcal{C}_r = \varnothing, \quad \theta^l(B^+) \cap \omega_+ \neq \varnothing.
	\end{align}
\end{enumerate}
The images of non-hemisphere and hemisphere charts are mutually disjoint.

\medskip
\underline{Interior Charts:}
For the domain interior, let $\{U_l\}_{l=K+1}^L$ be a fimily of open sets contained in $\Omega$ such that $\{U_l\}_{l=1}^L$ is an open cover of $\Omega$ with $C^\infty$ diffeomorphisms $\theta^l : B(0,1) \to U_l$ satisfying $\det \nabla \theta^l = C_l > 0$. These are similarly partitioned:
\begin{enumerate}
	\item \textbf{Neck charts} ($K+1 \leq l \leq L_1$): Cover $\mathcal{C}_r$ and satisfy:
	\begin{align} \label{A.6}
		\mathcal{C}_r \subset \bigcup_{l=K+1}^{L_1} \theta^l(B) \subset \omega \cap \{2 + h/3 < x_3 < 2 + 2h/3\}.
	\end{align}
	
	\item \textbf{Non-hemisphere charts} ($L_1+1 \leq l \leq L_2$): Disjoint from $\omega_+$ and $\mathcal{C}_r$:
	\begin{align} \label{A.7}
		\theta^l(B) \cap \mathcal{C}_r = \varnothing, \quad \theta^l(B) \cap \omega_+ = \varnothing.
	\end{align}
	
	\item \textbf{Hemisphere charts} ($L_2+1 \leq l \leq L$): Intersect $\omega_+$ but disjoint from $\mathcal{C}_r$:
	\begin{align} \label{A.8}
		\theta^l(B) \cap \mathcal{C}_r = \varnothing, \quad \theta^l(B) \cap \omega_+ \neq \varnothing.
	\end{align}
\end{enumerate}
The images of non-hemisphere and hemisphere interior charts are mutually disjoint.

\subsubsection{Local charts for $\Omega^{\epsilon}$}\label{sssec.222}
The coordinate charts for $\Omega^{\epsilon}$ are obtained by modifying the charts $\{\theta^{l}\}_{l=1}^{L}$ of $\Omega$ as follows. For sufficiently small $\epsilon > 0$:

\begin{enumerate}
	\item For $l \in I_1$ (neck charts):
	\begin{align*}
		(\theta^{\epsilon})^{l} &= H^{\epsilon} \circ \theta^{l}, \\
		H^{\epsilon}(x_1,x_2,x_3) &= \left( x_1, x_2, \frac{h+3+3\epsilon}{h} \left( x_3 - 2 - \frac{h}{3} \right) + \frac{h}{3} + 1 - \epsilon \right),
	\end{align*}
	with $\det \nabla (\theta^{\epsilon})^{l} = \frac{h+3+3\epsilon}{h} C_l$.
	
	\item For $l \in I_2$ (non-hemisphere charts):
	\[
	(\theta^{\epsilon})^{l} = \theta^{l}.
	\]
	
	\item For $l \in I_3$ (hemisphere charts):
	\[
	(\theta^{\epsilon})^{l}(x) = \theta^{l}(x) - (1 - \epsilon) e_3.
	\]
\end{enumerate}

The cut-off functions $\xi^{\epsilon,l}$ satisfy $\xi^{\epsilon,l} \circ (\theta^{\epsilon})^{l} = \xi^l \circ \theta^l$, leading to:
\[
\sum_{l=1}^{L} \xi^{\epsilon,l}(x) = \sum_{l \in I_2} \xi^l(x) + \sum_{l \in I_3} \xi^l(x + (1-\epsilon)e_3) + \sum_{l \in I_1} \xi^l(x_1, x_2, g^{\epsilon}(x_3)),
\]
where $g^{\epsilon}(x_3) = \frac{h}{h+3+3\epsilon} \left( x_3 - 1 - \frac{h}{3} + \epsilon \right) + 2 + \frac{h}{3}$.

While the sum equals 1 in most regions, it may decrease below 1 in the stretched neck region $\omega^{\epsilon} \cap \left\{1 - \epsilon + \frac{h}{3} < x_3 < 2 + \frac{2h}{3}\right\}$ due to $\epsilon$-dependence in the vertical derivative. For more details, one can see \cite{CS5}.

\subsection{Initial Velocity Field Construction}
We construct the initial velocity fields \(u^{\epsilon}\) following \cite{CS5}. Define a smooth boundary function \(a_{0}^{\epsilon}\) satisfying:
\begin{enumerate}[(1)]
	\item \(a_{0}^{\epsilon} = 1\) near \(X_{+}^{\epsilon}\) on \(\partial\omega_{+}^{\epsilon}\), and \(0\) on \(\partial\omega_{-} \cup (\partial\omega^{\epsilon} \cap \partial\Omega^{\epsilon})\).
	\item \(\int_{\partial\Omega^{\epsilon}} a_{0}^{\epsilon} \, dS = 0\) with \(\|a_{0}^{\epsilon}\|_{H^{2.5}(\partial\Omega^{\epsilon})} \leq b_{0}\) (\(\epsilon\)-independent).
\end{enumerate}

The initial velocity \(u_{0}^{\epsilon}\) solves the Stokes problem:
\begin{align}\label{3.1}
	\begin{cases}
		-\Delta u_{0}^{\epsilon} + \nabla s_{0}^{\epsilon} = 0 & \text{in } \Omega^{\epsilon}, \\
		\divop u_{0}^{\epsilon} = 0 & \text{in } \Omega^{\epsilon}, \\
		[\sym(u_{0}^{\epsilon}) \mathcal{N}^{\epsilon}] \cdot \tau_{\alpha}^{\epsilon} = 0 & \text{on } \partial\Omega^{\epsilon}, \\
		u_{0}^{\epsilon} \cdot \mathcal{N}^{\epsilon} = a_{0}^{\epsilon} & \text{on } \partial\Omega^{\epsilon},
	\end{cases}
\end{align}
where \(\mathcal{N}^{\epsilon}\) is the outward unit normal, \(\{\tau_{\alpha}^{\epsilon}\}\) an orthonormal tangent basis, and \(s_{0}^{\epsilon}\) the initial pressure. By elliptic regularity \cite{AS,SS}:
\begin{align}\label{3.2}
	\|u_{0}^{\epsilon}\|_{H^{3}(\Omega^{\epsilon})} \leq C \|a_{0}^{\epsilon}\|_{H^{2.5}(\partial\Omega^{\epsilon})} \leq C b_{0}.
\end{align}
This solution automatically satisfies:
\[
[\sym(u_{0}^{\epsilon}) \mathcal{N}^{\epsilon}] \times \mathcal{N}^{\epsilon} = 0 \quad \text{on } \partial\Omega^{\epsilon}.
\]
Additionally, we require compatibility with the full boundary condition:
\begin{align}\label{3.3}
	[(\sym(u_{0}^{\epsilon}) - p_{0}^{\epsilon} I) + (F_{0}^{\epsilon} (F_{0}^{\epsilon})^\top  - I)] \mathcal{N}^{\epsilon} = 0 \quad \text{on } \partial\Omega^{\epsilon},
\end{align}
where \(p_{0}^{\epsilon}\) will be determined later.

\subsection{Initial deformation tensor construction}
The compatibility condition derived from \eqref{3.1} and \eqref{3.3} requires:
\begin{align}\label{3.4}
	\left[(F_{0}^{\epsilon} (F_{0}^{\epsilon})^{\top} - I) \mathcal{N}^{\epsilon} \right] \cdot \tau_{\alpha}^{\epsilon} = 0 \quad \text{on } \partial\Omega^{\epsilon}.
\end{align}

To satisfy this, we define smooth divergence-free vector fields \( G_{0i} \in H^2(\Omega) \) (\(i=1,2,3\)) satisfying:
\begin{enumerate}[(1)]
	\item \(\divop G_{0i} = 0\) in \(\Omega\),
	\item \(\int_{\partial\omega \cap \partial\Omega} G_{0i} \cdot \mathcal{N}  dS = 0\), \(\int_{\partial\Omega} G_{0i} \cdot \mathcal{N}  dS = 0\),
	\item \((G_{0i})_3 = 0\) and \(\partial_3 G_{0i} = 0\) on \(\partial\omega \setminus (\overline{\omega} \cap \partial\Omega)\).
\end{enumerate}
Such fields exist as shown in \cite{HLZ}.

For \(\epsilon > 0\), define \(G_{0i}^\epsilon\) via vertical scaling in \(\omega^\epsilon\):
\begin{align}\label{3.5}
	(G_{0i}^{\epsilon})_j(x) =
	\begin{cases}
		(G_{0i})_j\left(x_1,x_2, 2 + h - \tfrac{h(2+h-x_3)}{h+1-\epsilon}\right) & j=1,2, \\
		\tfrac{h+1-\epsilon}{h}(G_{0i})_j\left(x_1,x_2, 2 + h - \tfrac{h(2+h-x_3)}{h+1-\epsilon}\right) & j=3,
	\end{cases}
\end{align}
and translation elsewhere:
\begin{align}\label{3.6}
	G_{0i}^{\epsilon}(x) =
	\begin{cases}
		G_{0i}(x_h, x_3 + 1 - \epsilon) & x \in \omega_{+}^{\epsilon}, \\
		G_{0i}(x) & \text{otherwise}.
	\end{cases}
\end{align}
This preserves \(\divop G_{0i}^\epsilon = 0\) and \(\int_{\partial\Omega} G_{0i}^\epsilon \cdot \mathcal{N}^\epsilon  dS = 0\).

The deformation tensor solves:
\begin{align}\label{3.7}
	\begin{cases}
		\curl F_{0i}^{\epsilon} = G_{0i}^{\epsilon} & \text{in } \Omega^{\epsilon}, \\
		\divop F_{0i}^{\epsilon} = 0 & \text{in } \Omega^{\epsilon}, \\
		(F_{0i}^{\epsilon}-I) \cdot \mathcal{N}^{\epsilon} = 0 & \text{on } \partial\Omega^{\epsilon},
	\end{cases}
\end{align}
where $F_{0i}^{\epsilon} := ((F_{0}^{\epsilon})_{1i} ,(F_{0}^{\epsilon})_{2i}, (F_{0}^{\epsilon})_{3i} )$ denotes the $i$-th column of $F_{0}$.

\begin{remark}
	The boundary condition \((F_{0i}^{\epsilon}-I) \cdot \mathcal{N}^{\epsilon} = 0\) is just a restriction on initial data, different from the usual initial condition \(F_{0i}^{\epsilon} \cdot \mathcal{N}^{\epsilon} = 0\), it does not hpld on the whole interval $[0,T]$, so it makes sense when we give another boundary condition which holds on $[0,T]$. And menwhile, it ensures compatibility with \eqref{3.4} and \eqref{1.2}, without requiring \(F_{0}^{\epsilon} \equiv 0\) on \(\partial\Omega^{\epsilon}\).
\end{remark}

By Lemma \ref{lemB.2} and boundary regularity,  we obtain
$$
\|F_{0}^{\epsilon}\|_{H^3(\Omega^{\epsilon})} \leq C
(\|F_{0}^{\epsilon}-I\|_{H^3(\Omega^{\epsilon})} +(\text{vol}(\Omega^{\epsilon}))^{1/2}) \leq C b_1,
$$
with \(C > 0\) independent of \(\epsilon\), depends only on \(\text{vol}(\Omega)\) and \(|\partial\Omega|_{H^{3.5}}\).
\subsection{Initial pressure construction}
According to \eqref{1.2} and \eqref{3.7}, for given \(u_{0}^{\epsilon}\) and \(F_{0}^{\epsilon}\), we define \(p_{0}^{\epsilon}\) via:
\begin{align}\label{3.10}
	\begin{cases}
		-\Delta p_{0}^{\epsilon} = \partial_j u_{0i}^{\epsilon} \partial_i u_{0j}^{\epsilon} - \partial_i F_{0jk}^{\epsilon} \partial_j F_{0ik}^{\epsilon} & \text{in } \Omega^{\epsilon}, \\
		p_{0}^{\epsilon} = (\mathcal{N}^{\epsilon})^\top \sym(u_{0}^{\epsilon}) \mathcal{N}^{\epsilon} & \text{on } \partial\Omega^{\epsilon}.
	\end{cases}
\end{align}

Standard elliptic regularity theory yields \(\epsilon\)-independent estimates:
\begin{align}\label{3.12}
	\|p_{0}^{\epsilon}\|_{H^2(\Omega^{\epsilon})}
	&\leq C \left( \|u_{0}^{\epsilon}\|_{H^3(\Omega^{\epsilon})}^2 + \|u_{0}^{\epsilon}\|_{H^3(\Omega^{\epsilon})} + \|F_{0}^{\epsilon}\|_{H^3(\Omega^{\epsilon})}^2 + 1 \right) \nonumber\\
	&\leq C P(b_0, b_1),
\end{align}
where \(P\) is a polynomial with constant term, and \(C > 0\) depends only on \(\text{vol}(\Omega)\) and \(|\partial\Omega|_{H^{3.5}}\).

\section{Reformulation in Lagrangian coordinates}\label{sec2}

For $\epsilon > 0$, let $\Omega^\epsilon$ denote the reference domain with boundary $\partial\Omega^\epsilon$, as defined in Appendix A. Consider system \eqref{1.1}-\eqref{1.2} with initial data $(\Omega_0^\epsilon, u_0^\epsilon, F_0^\epsilon)$, and denote its solutions by $(\Omega^\epsilon, u^\epsilon, F^\epsilon)$.

The Eulerian-Lagrangian correspondence is established via the flow map
$$\eta^\epsilon(t,\cdot) \colon \Omega^\epsilon \to \Omega(t)$$
satisfying
\begin{align} \label{2.1}
	\begin{cases}
		\partial_t \eta^\epsilon(t,x) = u^\epsilon(t, \eta^\epsilon(t,x)), & t > 0, \\
		\eta^\epsilon(0,x) = x.
	\end{cases}
\end{align}
Let $\nabla \eta^\epsilon$ denotes the deformation gradient. The incompressibility condition $\divop u^\epsilon = 0$ implies:
\[
J := \det \nabla \eta^\epsilon \equiv 1.
\]
Consequently, the boundary evolves according to:
\[
\Gamma(t) := \partial \Omega(t) = \eta^\epsilon(t, \partial \Omega^\epsilon).
\]

Next, define the Lagrangian quantities:
\begin{align*}
	v^\epsilon &= u^\epsilon \circ \eta^\epsilon, & & \text{(Lagrangian velocity)} \\
	\mathbb{F}^\epsilon &= F^\epsilon \circ \eta^\epsilon, & & \text{(Lagrangian deformation tensor)} \\
	q^\epsilon &= p^\epsilon \circ \eta^\epsilon, & & \text{(Lagrangian pressure)} \\
	A^\epsilon &= [\nabla \eta^\epsilon]^{-1}, \\
	g_{\alpha\beta}^\epsilon &= \partial_\alpha \eta^\epsilon \cdot \partial_\beta \eta^\epsilon, & & \text{(Induced metric on $\Omega^\epsilon$)} \\
	g^\epsilon &= \det(g_{\alpha\beta}^\epsilon).
\end{align*}
For simplicity, we suppress the superscript $\epsilon$ for $(\eta, v, \mathbb{F}, q, u_0, F_0, A)$.

Using the flow map \eqref{2.1}, the differential operators transform as:
\begin{alignat*}{2}
	\partial_j^\eta &= A_{kj} \partial_k, \quad  \nabla^\eta = A^\top \nabla, \\
	\divop^\eta v &= (\divop u) \circ \eta = \partial_j^\eta v_j = 0, \\
	(\divop^\eta \mathbb{F}^\top)_i &= [(\divop F^\top) \circ \eta]_i = \partial_j^\eta \mathbb{F}_{ji} = 0, \\
	(\divop^\eta (\mathbb{F}\mathbb{F}^\top))_i &= [(\divop (F F^\top)) \circ \eta]_i = \partial_j^\eta (\mathbb{F}_{ir} \mathbb{F}_{jr}), \\
	(\curl^\eta v)_i &= [(\curl u) \circ \eta]_i = \varepsilon_{ijk} \partial_j^\eta v_k, \\
	(\sym^\eta v)_{ij} &= [\sym(u) \circ \eta]_{ij} = \partial_j^\eta v_i + \partial_i^\eta v_j, \\
	\Delta^\eta v &= (\Delta u) \circ \eta = \partial_j^\eta \partial_j^\eta v,
\end{alignat*}
where $\varepsilon_{ijk}$ denotes the Levi-Civita permutation symbol:
\[
\varepsilon_{ijk} =
\begin{cases}
	1 & \text{if } (i,j,k) \text{ is an even permutation of } (1,2,3), \\
	-1 & \text{if } (i,j,k) \text{ is an odd permutation of } (1,2,3), \\
	0 & \text{otherwise}.
\end{cases}
\]

From the derivative formulas for inverse matrices and determinants (cf. \cite{Hao}), we have:
\begin{align}\label{deriv.A}
	\partial_m A_{ki} &= -A_{ji} \partial_j \partial_m \eta_l A_{kl} = -\partial_i^\eta \partial_m \eta_l A_{kl}, \\
	\partial_t A_{ki} &= -A_{ji} \partial_j v_l A_{kl} = -\partial_i^\eta v_l A_{kl}.
\end{align}
The Piola identity holds due to incompressibility ($J \equiv 1$):
\begin{align}\label{Piola}
	\partial_k A_{ki} = \partial_k (J A_{ki}) = 0,
\end{align}
which follows from the cofactor matrix representation of $A$:
\begin{align}\label{4.6}
	A = \begin{bmatrix}
		\partial_2 \eta \times \partial_3 \eta \\
		\partial_3 \eta \times \partial_1 \eta \\
		\partial_1 \eta \times \partial_2 \eta
	\end{bmatrix}.
\end{align}

\begin{lemma}[Commutator Identities] \label{lem:commutators}
	The following commutation relations hold:
	\begin{align*}
		[\partial_t, \partial_i^\eta] &= -\partial_i^\eta v_l \partial_l^\eta, \\
		[\partial_j, \partial_i^\eta] &= -\partial_i^\eta \partial_j \eta_l \partial_l^\eta, \\
		[\partial_j^\eta, \partial_i^\eta] &= 0.
	\end{align*}
\end{lemma}

\begin{proof}
	The first two identities follow directly from \eqref{deriv.A}. For the last identity, consider any smooth function $f$:
	\begin{align*}
		\partial_j^\eta (\partial_i^\eta f)
		&= A_{kj} \partial_k (\partial_i^\eta f) \\
		&= A_{kj} [\partial_k, \partial_i^\eta] f + A_{kj} A_{mi} \partial_m \partial_k f \\
		&= -A_{kj} A_{mi} \partial_m \partial_k \eta_l \partial_l^\eta f + A_{mi} \partial_m (A_{kj} \partial_k f) - A_{mi} (\partial_m A_{kj}) \partial_k f \\
		&= -A_{kj} A_{mi} \partial_m \partial_k \eta_l \partial_l^\eta f + \partial_i^\eta (\partial_j^\eta f) + A_{mi} A_{nj} \partial_n \partial_m \eta_l A_{kl} \partial_k f \\
		&= \partial_i^\eta (\partial_j^\eta f) + \left( A_{ki} A_{mj} \partial_m \partial_k \eta_l - A_{kj} A_{mi} \partial_m \partial_k \eta_l \right) \partial_l^\eta f \\
		&= \partial_i^\eta (\partial_j^\eta f) + \left( A_{mi} A_{kj} \partial_k \partial_m \eta_l - A_{kj} A_{mi} \partial_m \partial_k \eta_l \right) \partial_l^\eta f \\
		&= \partial_i^\eta (\partial_j^\eta f). \qedhere
	\end{align*}
\end{proof}

By incompressibility $\divop^\eta v = 0$ and Lemma \ref{lem:commutators}, we derive:
\begin{align*}
	\Delta^\eta v_i
	&= \partial_j^\eta \partial_j^\eta v_i + \partial_i^\eta (\partial_j^\eta v_j) \\
	&= \partial_j^\eta \partial_j^\eta v_i + \partial_j^\eta \partial_i^\eta v_j \\
	&= \partial_j^\eta (\sym^\eta v)_{ij},
\end{align*}
yielding the vector identity:
\begin{align} \label{laplacian_identity}
	\Delta^\eta v = \divop^\eta (\sym^\eta v).
\end{align}
The system \eqref{1.3} is therefore expressed in Lagrangian coordinates on $\Omega^\epsilon$ as:
\begin{equation}\label{2.2}
	\begin{cases}\eta_i(t, x)=x_i+\int_0^t v_i(s, x) d s & \text { in }[0, T] \times \Omega^\epsilon, \\ \partial_t v_i+\partial_i^\eta q-\Delta^\eta v_i=\sum_{j=1}^3\left(\left(F_0\right)_j \cdot \nabla\right)^2 \eta_i & \text { in }[0, T] \times \Omega^\epsilon, \\ \operatorname{div}^\eta v=0, & \text { in }[0, T] \times \Omega^\epsilon, \\ {\left[\left(\mathcal{S}^\eta v\right)_{i j}-q \delta_{i j}+\mathbb{F}_{i m} \mathbb{F}_{j m}-\delta_{i j}\right] n_j=0} & \text { on }[0, T] \times \partial \Omega^\epsilon, \\ (\eta, v)=\left(e, u_0\right) & \text { in }\{t=0\} \times \Omega^\epsilon .\end{cases}
\end{equation}
where $e(x) = x$ is the identity map. The unit normal $n^\epsilon$ at $\eta(t,x)$ is given by:
\begin{equation} \label{normal_transform}
	n^\epsilon = \frac{(A^\epsilon)^\top \mathcal{N}^\epsilon}{|(A^\epsilon)^\top \mathcal{N}^\epsilon|},
\end{equation}
with $\mathcal{N}^\epsilon$ being the outward unit normal to $\partial\Omega^\epsilon$.

Define  $\mathbb{F}_j:=\left(\mathbb{F}_{1 j}, \mathbb{F}_{2 j}, \mathbb{F}_{3 j}\right)$ to be the $j$-th column of $\mathbb{F}$. By (\ref{deriv.A})-(\ref{Piola}), next, applying $A$ to  (\ref{2.2}), we obtain, for any $i,j$,
\[
\partial_t(A_{ik}\mathbb{F}_{kj})=\partial_t A_{ik}\mathbb{F}_{kj}+A_{ik}\partial_t\mathbb{F}_{kj}=-A_{lk}\partial_l v_r A_{ir}\mathbb{F}_{kj}+A_{ik}\mathbb{F}_{mj}A_{lm}\partial_{l}v_k=0,
\]
which implies that, $A_{ik}\mathbb{F}_{kj}=\delta_{ik}(F_0)_{kj}=(F_0)_{ij}$, which means
\[
\mathbb{F}_{kj}=\mathbb{F}_{lj}A_{li}\partial_l\eta_k=(F_0)_{lj}\partial_l\eta_k=((F_0)_j\cdot\nabla)\eta_k.
\]
Consequently, it motivates us to eliminate the elastic field $\mathbb{F}$ from the system (\ref{2.3}). Indeed, we may rewrite the term:
\begin{equation}\label{NEW1}
	A_{kj} \partial_k (\mathbb{F}_{im} \mathbb{F}_{jm})=A_{kj} \partial_k \mathbb{F}_{im} \mathbb{F}_{jm}=(F_{0})_{km}\partial_{k}[((F_{0})_{m}\cdot\nabla)\eta_{i}].
\end{equation}

Then the system (\ref{2.2}) can be simplified  equivalently as  the following  free-surface incompressible Euler system with a forcing term induced by the flow map.
\begin{equation}\label{2.3}
	\begin{cases}
		\eta_i(t,x) = x_i + \int_0^t v_i(s,x)  ds & \text{in } [0,T] \times \Omega^\epsilon, \\
		\partial_t v_i + \partial_i^\eta q - \Delta^\eta v_i = \sum_{j=1}^3((F_0)_j\cdot\nabla)^2\eta_i & \text{in } [0,T] \times \Omega^\epsilon,  \\
		\divop^\eta v = 0, & \text{in } [0,T] \times \Omega^\epsilon,  \\
		\left[ (\sym^\eta v)_{ij} - q\delta_{ij} +\mathbb{F}_{im}\mathbb{F}_{jm} - \delta_{ij} \right] n_j = 0 & \text{on } [0,T] \times \partial\Omega^\epsilon, \\
		(\eta, v) = (e, u_0) & \text{in } \{t=0\} \times \Omega^\epsilon.
	\end{cases}
\end{equation}

In the system (\ref{2.3}) and initial condition (\ref{3.3}), the initial field $\mathbb{F}$ can be regarded as a parameter vector that satisfies
$$\begin{cases}
	\partial_j\left(F_0\right)_{j i}=0 \quad &\text { in } \quad \Omega^\epsilon,\\ \left(F_0-I\right)_i \cdot \mathcal{N}=0, \quad &\text { on } \quad \partial \Omega^\epsilon.
\end{cases}
$$
\section{Preserved physical energies}\label{sec.conserved}
\begin{proposition}
	The system \eqref{2.3} conserves the energy functional:
	
	\begin{align*}
		E_0(t) &:= \int_{\Omega^{\epsilon}} \left( |v(t,x)|^2 + \sum_{j=1}^3|((F_0)_j\cdot\nabla)\eta|^2 \right) dx + \int_0^t \int_{\Omega^{\epsilon}} |\sym^\eta v(s,x)|^2  dx  ds \\
		&= \int_{\Omega^{\epsilon}} \left( |u_0(x)|^2 + |F_0(x)|^2 \right) dx = E_0(0).
	\end{align*}
	
\end{proposition}

\begin{proof}
	
	Taking the $L_{2}$  inner product of the \eqref{2.3} with $v$,  and apply the divergence theorem, we have
	\begin{align}\label{energy1}
		\frac{1}{2} \frac{d}{dt} \int_{\Omega^{\epsilon}} |v|^2  dx
		+ \int_{\Omega^{\epsilon}} (\sym^\eta v)_{ij} \partial_j^\eta v_i  dx
		&= \int_{\partial\Omega^{\epsilon}} \left[ n_j (\sym^\eta v)_{ij} - n_i q \right] v_i  dS \\
		&\quad + \int_{\Omega^{\epsilon}} \sum_{j=1}^3((F_0)_j\cdot\nabla)^2\eta_i v_i  dx, \nonumber
	\end{align}
	since $\partial_j\left(F_0\right)_{j i}=0$ and  $A_{ik}\mathbb{F}_{kj}=(F_0)_{ij}$, we obtain
	\begin{equation}\label{energy2}
		\begin{aligned}
			& -\int_{\Omega^{\varepsilon}}\left(\left(F_0\right)_j \cdot \nabla\right)^2 \eta_k v_k d x \\
			= & -\int_{\Omega^{\varepsilon}}\left(F_0\right)_{i j} \partial_i\left(\left(F_0\right)_{m j} \partial_m \eta_k\right) v_k d x \\
			= & -\int_{\Omega^{\varepsilon}} \partial_i\left(\left(F_0\right)_{i j}\left(F_0\right)_{m j} \partial_m \eta_k v_k\right) d x+\int_{\Omega^{\varepsilon}}\left(F_0\right)_{i j}\left(F_0\right)_{m j} \partial_m \eta_k \partial_i v_k d x\\
			= & -\int_{\Omega^{\varepsilon}} \partial_i\left(A_{i l} \mathbb{F}_{l j} A_{m h} \mathbb{F}_{h j} \partial_m \eta_k v_k\right) d x  +\int_{\Omega^{\varepsilon}}\left(F_0\right)_j\cdot \nabla \eta_k\left(F_0\right)_j \cdot \nabla \partial_t \eta_k d x \\
			= & -\int_{\Omega^{\varepsilon}} \partial_l^\eta\left(\mathbb{F}_{lj} {\mathbb{F}}_{k j}v_k\right)  d x+\frac{1}{2} \frac{d}{d t} \int_{\Omega^{\varepsilon}}\left|\left(F_0\right)_j \cdot \nabla \eta\right|^2 d x.
		\end{aligned}
	\end{equation}
	Summing \eqref{energy1} and \eqref{energy2}, and use the boundary condition, we get
	\begin{align*}
		\frac{1}{2} \frac{d}{dt}& \int_{\Omega^{\epsilon}} (|v|^2 +\sum_{j=1}^3 |((F_0)_j\cdot\nabla)\eta|^2) dx + \int_{\Omega^{\epsilon}} (\sym^\eta v)_{ij} \partial_j^\eta v_i  dx \\
		&= \int_{\partial\Omega^{\epsilon}} \left[ n_j (\sym^\eta v)_{ij} - n_i q \right] v_i  dS
		+ \sum_{j=1}^3 \int_{\Omega^{\varepsilon}} \partial_l^\eta\left(\mathbb{F}_{lj} {\mathbb{F}}_{k j} v_k\right) d x  dx.
	\end{align*}
	and the boundary condition \eqref{2.3}, we get
	$$\begin{aligned}
		\int_{\partial\Omega^{\epsilon}} & \left[ n_k (\sym^\eta v)_{lk} - n_l q +\sum_{j=1}^3  n_k (\mathbb{F}_{lj} \mathbb{F}_{kj} - \delta_{lk}) \right] v_l dS \\
		&= \int_{\partial\Omega^{\epsilon}} n_k \left[ (\sym^\eta v)_{lk} - q \delta_{lk} + \sum_{j=1}^3 \mathbb{F}_{lj} \mathbb{F}_{kj} - \delta_{lk} \right] v_l  dS = 0.
	\end{aligned}
	$$
	The bulk term simplifies to
	\begin{align*}
		\int_{\Omega^{\epsilon}} (\sym^\eta v)_{ij} \partial_j^\eta v_i  dx = \frac{1}{2} \int_{\Omega^{\epsilon}} |\sym^\eta v|^2  dx,
	\end{align*}
	since
	\begin{align}\label{eq.sysnv}
		(\sym^\eta v)_{ij} \partial_j^\eta v_i = \frac{1}{2} |\sym^\eta v|^2
	\end{align}
	for incompressible flows.
	Thus we obtain
	\begin{align*}
		\frac{1}{2} \frac{d}{dt} \int_{\Omega^{\epsilon}} (|v|^2 +\sum_{j=1}^3 |((F_0)_j\cdot\nabla)\eta|^2) dx
		+ \frac{1}{2} \int_{\Omega^{\epsilon}} |\sym^\eta v|^2  dx = 0.
	\end{align*}
	Integrating over $[0,t]$ completes the proof.
	
\end{proof}

\section{A priori estimates}\label{sec4}

In this section, we aim to demonstrate the existence of a finite-time splash singularity by establishing fundamental a priori estimates for the equations \eqref{2.3} in Lagrangian coordinates. From this point onward, we assume that \( (\eta, v, \mathbb{F}, q) \) represents a smooth solution to the problem \eqref{2.3} over the time interval \( [0, T] \), where \( T > 0 \). For each \( t \in [0, T] \), we define a energy function as follows:
\begin{align}\label{4.1a}
	E^{\epsilon}(t)=&1+ \|\eta(t,\cdot)\|_{H^3(\Omega^{\epsilon})}^{2}+\|v(t,\cdot)\|_{H^2(\Omega^{\epsilon})}^{2}+\|v_{t}(t,\cdot)\|_{L^2(\Omega^{\epsilon})}^{2}\nonumber\\
	&+\int_{0}^{t}\|\nabla v(s,\cdot)\|_{H^2(\Omega^{\epsilon})}^{2}ds
	+\int_{0}^{t}\|q(s,\cdot)\|_{H^2(\Omega^{\epsilon})}^{2}ds+\int_{0}^{t}\|\nabla v_{s}(s,\cdot)\|_{L^2(\Omega^{\epsilon})}^{2}ds.
\end{align}
\begin{proposition}[A Priori Energy Estimate]\label{prop4.6}
	Assuming $\partial\Omega(t)$ remains self-intersection free and $\epsilon$-independent. Define $\mathcal{M}_0 = E^\epsilon(0)+ \|F_{0}\|^{2}_{H^3(\Omega^{\epsilon})}$, and there exist $T>0$  such that the smooth solution $(\eta,v,\mathbb{F},q)$ to \eqref{2.3} satisfies:
	\begin{align}\label{4.59}
		\sup_{t\in[0,T]} E^{\epsilon}(t)\leq \mathcal{Q}(\mathcal{M}_{0}),
	\end{align}
	$\mathcal{Q}$ is a polynomial with $\epsilon$-independent coefficients.
\end{proposition}
\begin{remark}
	In fact, the definition of energy should include "deformation tensor" term such that $\sum_{j=1}^3\|((F_0)_j\cdot\nabla)\eta(t,\cdot)\|_{H^2(\Omega^{\epsilon})}^{2}$ and $\sum_{j=1}^3\|((F_0)_j\cdot\nabla)\partial_t\eta(t,\cdot)\|_{L^2(\Omega^{\epsilon})}^{2}$, but using our initial value construction and the Cauchy-Schwarz inequality, these higher-order norms can be controlled by the $E^{\epsilon}(t)$ and $\|F_{0}\|_{H^3(\Omega^{\epsilon})}$. Meanwhile, the order of "deformation tensor" term the is lower than that in \cite{ZJ},  in contrast to Zhang \cite{ZJ}, in this paper, because the total energy we adopt involves a spacetime function, we no longer need to perform the div-curl-tangential decomposition for the high-order norms of $F_{0}\cdot\nabla\eta$ and $v$, this is the most significant difference between our paper and other works.
\end{remark}
\subsection{A priori assumption}
We make the following fundamental assumption: There exist constants $0 < \theta \ll 1$ and $T > 0$ (independent of $\epsilon$) such that:
\begin{align}\label{4.2}
	\sup_{t \in [0, T]} \|\nabla \eta(t) - I\|_{L^{\infty}(\Omega^{\epsilon})} \leq \theta^{2}.
\end{align}
This implies $\nabla\eta$ remains uniformly close to the identity matrix throughout $[0,T]$.

\begin{lemma}[Deformation Gradient Estimates]\label{lem4.1}
	Under assumption \eqref{4.2}, for sufficiently small $\theta > 0$:
	\begin{align}\label{4.3}
		\sup_{t \in [0, T]} \left( \| A(t) - I \|_{L^{\infty}} + \| A A^\top (t) - I \|_{L^{\infty}} \right) \leq \theta,
	\end{align}
	where $A = (\nabla \eta)^{-1}$ is the deformation gradient.
\end{lemma}

\begin{proof}
	From \eqref{4.2} we have:
	\begin{align}
		\|\nabla \eta\|_{L^\infty} &\leq C, \label{4.4} \\
		\|\partial_j \eta_i - \delta_{ij}\|_{L^\infty} &\leq \theta^2, \label{4.5}
	\end{align}
	with $C$ independent of $\epsilon$.
	
	Since $A = (\nabla \eta)^{-1}$, \eqref{4.5} implies:
	\begin{align}
		\|A - I\|_{L^\infty} \leq C \theta^2. \label{4.7}
	\end{align}
	Using the identity $AA^\top - I = A(A^\top - I) + (A - I)$:
	\begin{align}
		\|AA^\top - I\|_{L^\infty} &\leq \|A\|_{L^\infty} \|A^\top - I\|_{L^\infty} + \|A - I\|_{L^\infty} \nonumber \\
		&\leq (1 + C\theta^2)C\theta^2 + C\theta^2 \leq C'\theta^2. \label{4.9}
	\end{align}
	Choosing $\theta < \min(1, 1/\sqrt{3C'})$ yields \eqref{4.3}.
\end{proof}
\subsection{The boundary regularity of velocity}
Next, since we are dealing with boundary conditions involving mixed pressure, velocity, and elasticity, which differ from those considered in \cite{CS5}, we need to revisit the boundary regularity of velocity.

To obtain $\epsilon$-independent a priori estimates, we pull back system \eqref{2.3} to the boundary charts $(\theta^{\epsilon})^{l}$. For each $l=1,\dots,K$, define the chart parametrization:
\begin{align}\label{2.4}
	(\eta^{\epsilon})^{l}(t,\cdot) := \eta^{\epsilon}(t,\cdot) \circ (\theta^{\epsilon})^{l} : B^{+} \to \Omega(t).
\end{align}

The Lagrangian quantities on each chart are:
\begin{align*}
	(v^{\epsilon})^{l} &:= u \circ (\eta^{\epsilon})^{l}, &
	(q^{\epsilon})^{l} &:= p \circ (\eta^{\epsilon})^{l}, \\
	(\mathbb{F}^{\epsilon})^{l} &:= F \circ (\eta^{\epsilon})^{l}, &
	(A^{\epsilon})^{l} &:= \left[\nabla (\eta^{\epsilon})^{l}\right]^{-1}, \\
	(n^{\epsilon})^{l} &:= (g^{\epsilon})^{-1/2} \partial_{y_1} (\eta^{\epsilon})^{l} \times \partial_{y_2} (\eta^{\epsilon})^{l},
\end{align*}
where $g^{\epsilon} = \det(g_{\alpha\beta}^\epsilon)$ is the induced metric determinant. Note that $\det \nabla (\theta^{\epsilon})^{l} = C^l > 0$ remains constant.

Dropping $\epsilon$-superscripts for simplicity, the pulled-back system becomes:

\begin{align}\label{2.5}
	\begin{cases}
		\eta^{l}(t,y) = \theta^{l}(y) + \int_0^t v^{l}(s,y)  ds & \text{in } [0,T]\times B^{+}, \\
		\partial_t v_i^l + A_{ji}^l \partial_j q^l - \Delta^\eta v_i^l = \sum_{j=1}^3((F_0^l)_j\cdot\nabla)^2\eta_i^l & \text{in } [0,T]\times B^{+}, \\
		A_{ji}^l \partial_j v_i^l = 0, \quad\partial_j(F_{0}^{l})_{ji}=0& \text{in } [0,T]\times B^{+}, \\
		\left[ (A_{rj}^l \partial_r v_i^l + A_{ri}^l \partial_r v_j^l - q^l \delta_{ij}) + (\mathbb{F}_{im}^l\mathbb{F}_{jm}^l - \delta_{ij}) \right] n_j^l = 0 & \text{on } [0,T]\times B^{0}, \\
		(\eta^l, v^l)|_{t=0} = (\theta^l, u_0 \circ \theta^l) & \text{in } B^{+},
	\end{cases}
\end{align}

where index $l$ is not summed in any equations.

By the Norm Equivalence Lemma \cite{CS2} and Sobolev-type inequalities in Appendix A, system \eqref{2.5} enables uniform analysis of $(\eta^l, v^l, q^l, \mathbb{F}^l)$ independent of $\epsilon$.

\begin{proposition}[Boundary Velocity Estimate]\label{prop4.2}
	For any $\delta > 0$, there exists $C_\delta > 0$ independent of $\epsilon$ such that
	\begin{align}\label{4.10}
		\int_0^T \|v(t)\|_{H^{5/2}(\partial\Omega^\epsilon)}^2 dt \leq \mathcal{Q}(\mathcal{M}_0) + \delta \sup_{t \in [0,T]} E^\epsilon(t) + C_\delta T \mathcal{P}\left( \sup_{t \in [0,T]} E^\epsilon(t) \right),
	\end{align}
	which  $\mathcal{P}(\cdot)$ or $\mathcal{Q}(\cdot)$ a polynomial with $\epsilon$-independent coefficients.
\end{proposition}

\begin{proof}
	Let $\varsigma^{l}$ be the cut-off function for $l=1,\cdots,K$. By \eqref{2.5}, we have the following:
	
	\begin{align}\label{4.11}
		\int_{B^{+}}\varsigma^{l}\bar{\partial}_{\alpha}\bar{\partial}_{\beta}(\partial_{t}v_{i}^{l}+A_{ji}^{l}\partial_{j}q^{l}-\partial_{j}(A_{jr}^{l}A_{kr}^{l}\partial_{k}v_{i}^{l})-\sum_{j=1}^3((F_0^l)_j\cdot\nabla)^2\eta_i^l))\varsigma^{l}\bar{\partial}_{\alpha}\bar{\partial}_{\beta}v_{i}^{l}dy=0.
	\end{align}
	To simplify the notation, we fix  $l\in\{1,\cdots,K\}$ and drop the superscript. Using the Piola identity \eqref{Piola}, where $J^{l}$ is constant, we obtain
	
	\begin{align}\label{4.12}
		&\int_{B^{+}}\varsigma^{2}\bar{\partial}_{\alpha}\bar{\partial}_{\beta}\partial_{t}v_{i}\bar{\partial}_{\alpha}\bar{\partial}_{\beta}v_{i}dy+\int_{B^{+}}\varsigma^{2}\bar{\partial}_{\alpha}\bar{\partial}_{\beta}\partial_{j}(A_{ji}q)\bar{\partial}_{\alpha}\bar{\partial}_{\beta}v_{i}dy \nonumber\\
		&-\int_{B^{+}}\varsigma^{2}\bar{\partial}_{\alpha}\bar{\partial}_{\beta}\partial_{j}(A_{jr}A_{kr}\partial_{k}v_{i})\bar{\partial}_{\alpha}\bar{\partial}_{\beta}v_{i}dy=\int_{B^{+}}\varsigma^{2}\bar{\partial}_{\alpha}\bar{\partial}_{\beta}(\sum_{j=1}^3((F_0)_j\cdot\nabla)^2\eta_i)\bar{\partial}_{\alpha}\bar{\partial}_{\beta}v_{i}dy.
	\end{align}
	
	By applying integration by parts and the divergence theorem, equation \eqref{4.12} can be rewritten as:
	\begin{align}\label{4.13}	&\frac{1}{2}\frac{d}{dt}\|\varsigma\bar{\partial}^{2}v(t,\cdot)\|_{L^2(B^+)}^{2}+\|\varsigma\bar{\partial}^{2}\nabla v(t,\cdot)\|_{L^2(B^+)}^{2}\nonumber \\
		=&-\int_{B^{+}}\bar{\partial}_{\alpha}\bar{\partial}_{\beta}(A_{ji}q)\partial_{j}(\varsigma^{2}\bar{\partial}_{\alpha}\bar{\partial}_{\beta}v_{i})dy+\int_{\partial B^{+}}\mathcal{N}_{j}\bar{\partial}_{\alpha}\bar{\partial}_{\beta}(A_{ji}q)\varsigma^{2}\bar{\partial}_{\alpha}\bar{\partial}_{\beta}v_{i}dS \nonumber\\
		&-\int_{B^{+}}\bar{\partial}_{\alpha}\bar{\partial}_{\beta}((A_{jr}A_{kr}-\delta_{jk})\partial_{k}v_{i})(\varsigma^{2}\bar{\partial}_{\alpha}\bar{\partial}_{\beta}\partial_{j}v_{i})dy\nonumber\\
		&+2\int_{B^{+}}\bar{\partial}_{\alpha}\bar{\partial}_{\beta}(A_{jr}A_{kr}\partial_{k}v_{i})\varsigma\partial_{j}\varsigma\bar{\partial}_{\alpha}\bar{\partial}_{\beta}v_{i}dy \nonumber\\
		&-\int_{\partial B^{+}}\mathcal{N}_{j}\bar{\partial}_{\alpha}\bar{\partial}_{\beta}(A_{jr}A_{kr}\partial_{k}v_{i})\varsigma^{2}\bar{\partial}_{\alpha}\bar{\partial}_{\beta}v_{i}dS\nonumber\\
		&+\int_{B^{+}}\bar{\partial}_{\alpha}\bar{\partial}_{\beta}(\sum_{j=1}^3((F_0)_j\cdot\nabla)^2\eta_i)(\varsigma^{2}\bar{\partial}_{\alpha}\bar{\partial}_{\beta}v_{i})dy.
	\end{align}
	
	Since $A_{ik}\mathbb{F}_{kj}=(F_0)_{ij} $ and div-free condition $\partial_{k}(F_{0})_{kj}=0$,
	$$
	\begin{aligned}
		& \int_{B^{+}}\bar{\partial}_\alpha \bar{\partial}_\beta\left(\sum _ { j = 1 } ^ { 3 } \left(F_0\right)_{k j} \partial_k\left(\left(F_0\right)_{m j} \partial_m \eta_i\right)\right)\left(\zeta^2 \bar{\partial}_\alpha \bar{\partial}_\beta v_i\right) d y \\
		& =-\int_{B^{+}}\bar{\partial}_\alpha \bar{\partial}_\beta\left(\sum _ { j = 1 } ^ { 3 } \left(\left(F_0\right)_{k j} \left(F_0\right)_{m j} \partial_m \eta_i\right)\right)\partial_k\left(\zeta^2 \bar{\partial}_\alpha \bar{\partial}_\beta v_i \right)d y\\
		&+\int_{B^{+}}\bar{\partial}_\alpha \bar{\partial}_\beta\partial_k\left(\sum _ { j = 1 } ^ { 3 } \left(\left(F_0\right)_{k j} \left(F_0\right)_{m j} \partial_m \eta_i\right)\left(\zeta^2 \bar{\partial}_\alpha \bar{\partial}_\beta v_i \right)\right)d y
		\\&=-\int_{B^{+}}\bar{\partial}_\alpha \bar{\partial}_\beta\left(\sum _ { j = 1 } ^ { 3 } \left(\left(F_0\right)_{k j} \left(F_0\right)_{m j} \partial_m \eta_i\right)\right)\partial_k\left(\zeta^2 \bar{\partial}_\alpha \bar{\partial}_\beta v_i \right)d y\\&+\int_{\partial{B}^{+}}\mathcal{N}_{k}\bar{\partial}_\alpha \bar{\partial}_\beta\left(\sum _ { j = 1 } ^ { 3 } \left(A_{kl}\mathbb{F}_{lj}\mathbb{F}_{ij} \right)\right)\left(\zeta^2 \bar{\partial}_\alpha \bar{\partial}_\beta v_i \right)dS.
	\end{aligned}
	$$
	
	Since $\partial B^{+}=B^{0}\cup(\partial B\cap\{y_{3}>0\})$, $\mathcal{N}=(0,0,-1)$ on $B^{0}$, and $\varsigma$ vanishing near $(\partial B\cap\{y_{3}>0\})$, together with the boundary condition \eqref{2.5} , we can show that:
	
	\begin{align*}
		&\int_{\partial B^{+}}\mathcal{N}_{j}(\bar{\partial}_{\alpha}\bar{\partial}_{\beta}(A_{ji}q)\varsigma^{2}\bar{\partial}_{\alpha}\bar{\partial}_{\beta}v_{i})dS-\int_{\partial B^{+}}\mathcal{N}_{j}(\bar{\partial}_{\alpha}\bar{\partial}_{\beta}(A_{jr}A_{kr}\partial_{k}v_{i})\varsigma^{2}\bar{\partial}_{\alpha}\bar{\partial}_{\beta}v_{i})dS \\
		&-\int_{\partial B^{+}}\mathcal{N}_{k}(\bar{\partial}_{\alpha}\bar{\partial}_{\beta}\left(\sum _ { j = 1 } ^ { 3 } \left(A_{kl}\mathbb{F}_{lj}\mathbb{F}_{ij} \right)\right)\varsigma^{2}\bar{\partial}_{\alpha}\bar{\partial}_{\beta}v_{i})dS \\
		=&\int_{B^{0}}\bar{\partial}_{\alpha}\bar{\partial}_{\beta}(A_{ji}\mathcal{N}_{j}q-A_{jr}\mathcal{N}_{j}A_{kr}\partial_{k}v_{i}-\left(\mathcal{N}_{k}\sum _ { j = 1 } ^ { 3 } \left(A_{kl}\mathbb{F}_{lj}\mathbb{F}_{ij} \right)\right))\varsigma^{2}\bar{\partial}_{\alpha}\bar{\partial}_{\beta}v_{i}dS \\
		=&\int_{B^{0}}\bar{\partial}_{\alpha}\bar{\partial}_{\beta}(q\delta_{ij}A_{kj}\mathcal{N}_{k}-A_{rj}\partial_{r}v_{i}A_{kj}\mathcal{N}_{k}-A_{jr}\mathcal{N}_{j}A_{kr}\partial_{k}v_{i}\\&-\left(\mathcal{N}_{k}\sum _ { j = 1 } ^ { 3 } \left(A_{kl}\mathbb{F}_{lj}\mathbb{F}_{ij} \right)\right))\varsigma^{2}\bar{\partial}_{\alpha}\bar{\partial}_{\beta}v_{i}dS \\
		=&\int_{B^{0}}\bar{\partial}_{\alpha}\bar{\partial}_{\beta}(A_{kj}\mathcal{N}_{k}A_{ri}\partial_{r}v_{j}-A_{kj}\mathcal{N}_{k}\delta_{ij})\varsigma^{2}\bar{\partial}_{\alpha}\bar{\partial}_{\beta}v_{i}dS \\
		=&\int_{B^{+}}\partial_{k}(\varsigma^{2}\bar{\partial}_{\alpha}\bar{\partial}_{\beta}(\mathcal{N}_{k}A_{kj}A_{ri}\partial_{r}v_{j})\bar{\partial}_{\alpha}\bar{\partial}_{\beta}v_{i})dy-\int_{B^{+}}\partial_{k}(\varsigma^{2}\bar{\partial}_{\alpha}\bar{\partial}_{\beta}A_{ki}\bar{\partial}_{\alpha}\bar{\partial}_{\beta}v_{i})dy.
	\end{align*}
	
	Notice that
	\begin{align}\label{4.16a}
		&\frac{1}{2}\frac{d}{dt}\|\varsigma\bar{\partial}^{2}((F_0)_j\cdot\nabla)
		\eta_{i})\|_{L^2(B^+)}^{2}\\&=\frac{1}{2}\frac{d}{dt}\int_{B^{+}}\varsigma^{2}\bar{\partial}_{\alpha}\bar{\partial}_{\beta}[
		(F_0)_{kj}\partial_k\eta_{i}]\bar{\partial}_{\alpha}\bar{\partial}_{\beta}[(F_0)_{mj}\partial_m\eta_{i}]dy\nonumber
		\\&=\int_{B^{+}}\varsigma^{2}\bar{\partial}_{\alpha}\bar{\partial}_{\beta}[
		(F_0)_{kj}\partial_k\eta_{i}]\bar{\partial}_{\alpha}\bar{\partial}_{\beta}[(F_0)_{mj}\partial_mv_{i}]dy\nonumber,
	\end{align}
	
	we get
	\begin{equation}\label{4.16b}
		\begin{aligned}
			&\int_{B^{+}}\bar{\partial}_\alpha \bar{\partial}_\beta\left( \sum _ { j = 1 } ^ { 3 }\left(\left(F_0\right)_{k j} \left(F_0\right)_{m j} \partial_m \eta_i\right)\right)\partial_k\left(\zeta^2 \bar{\partial}_\alpha \bar{\partial}_\beta v_i \right)d y\\&=\frac{1}{2}\frac{d}{dt}\|\varsigma\bar{\partial}^{2}(F_0)_j\cdot\nabla)
			\eta\|_{L^2(B^+)}^{2}+\text{Other terms},
		\end{aligned}
	\end{equation}
	which other terms  have similar structures and can yield similar results, such as:
	$$\begin{aligned}
		&
		\int_{B^{+}}\bar{\partial}_\alpha\left(F_0\right)_{k j} \bar{\partial}_\beta\left( \left( \left(F_0\right)_{m j} \partial_m \eta_i\right)\right)\partial_k\left(\varsigma^2 \bar{\partial}_\alpha \bar{\partial}_\beta v_i \right)d y, \\&\int_{B^{+}}\left(F_0\right)_{k j} \bar{\partial}_\alpha\bar{\partial}_\beta\left( \left( \left(F_0\right)_{m j} \partial_m \eta_i\right)\right)\partial_k\left(\varsigma^2 \bar{\partial}_\alpha \bar{\partial}_\beta v_i \right)d y,\\&\int_{B^{+}}\zeta^2\bar{\partial}_\beta\left(\left(F_0\right)_{k j}\right) \bar{\partial}_\alpha \bar{\partial}_\beta( \left(F_0\right)_{m j} \partial_m \eta_i)\partial_k\bar{\partial}_\alpha v_i d y, \cdots. \end{aligned}$$
	
	Next, we integrate \eqref{4.13} over the time interval [0, T] and by \eqref{4.16b}:
	
	\begin{align}\label{4.15}		&\frac{1}{2}\|\varsigma\bar{\partial}^{2}v(t,\cdot)\|_{L^2(B^+)}^{2}+\int_{0}^{t}\|\varsigma\bar{\partial}^{2}\nabla v(s,\cdot)\|_{L^2(B^+)}^{2}ds+\frac{1}{2}\|\varsigma\bar{\partial}^{2}(F_0)_j\cdot\nabla)
		\eta\|_{L^2(B^+)}^{2}\\ &\leq \mathcal{Q}(\mathcal{M}_{0})+\sum_{i=1}^{6}\mathcal{R}_{i},\nonumber
	\end{align}
	
	where $\mathcal{R}_{i}$ (for $i=1,2,\dots,7$) are given by
	
	\begin{align*}
		\mathcal{R}_{1}&=-\int_{0}^{t}\int_{B^{+}}\bar{\partial}_{\alpha}\bar{\partial}_{\beta}(A_{ji}q)\partial_{j}(\varsigma^{2}\bar{\partial}_{\alpha}\bar{\partial}_{\beta}v_{i})dyds, \\
		\mathcal{R}_{2}&=-\int_{0}^{t}\int_{B^{+}}\bar{\partial}_{\alpha}\bar{\partial}_{\beta}((A_{jr}A_{kr}-\delta_{jk})\partial_{k}v_{i})(\varsigma^{2}\bar{\partial}_{\alpha}\bar{\partial}_{\beta}\partial_{j}v_{i})dyds, \\
		\mathcal{R}_{3}&=2\int_{0}^{t}\int_{B^{+}}\bar{\partial}_{\alpha}\bar{\partial}_{\beta}(A_{jr}A_{kr}\partial_{k}v_{i})\varsigma\partial_{j}\varsigma\bar{\partial}_{\alpha}\bar{\partial}_{\beta}v_{i}dyds, \\
		\mathcal{R}_{4}&=\int_{0}^{t}\int_{B^{+}}\partial_{k}(\varsigma^{2}\bar{\partial}_{\alpha}\bar{\partial}_{\beta}(A_{kj}A_{ri}\partial_{r}v_{j})\bar{\partial}_{\alpha}\bar{\partial}_{\beta}v_{i})dyds, \\
		\mathcal{R}_{5}&=-\int_{0}^{t}\int_{B^{+}}\partial_{k}(\varsigma^{2}\bar{\partial}_{\alpha}\bar{\partial}_{\beta}A_{ki}\bar{\partial}_{\alpha}\bar{\partial}_{\beta}v_{i})dyds, \\
		\mathcal{R}_{6}&= \text{Some terms containing $F_{0}\cdot \nabla\eta$, $F_{0}$ and $v$ integrated in time over $[0,T]$}.
	\end{align*}
	
	By the definition of $A$, \eqref{deriv.A} and \eqref{4.4}, we can easily derive the following estimate:
	\begin{align}\label{4.16}
		|A|\le C|\nabla\eta\|\nabla\eta|, \quad
		|\partial A|\le C|\nabla\eta\|\nabla^{2}\eta|, \quad
		|\partial^{2}A|\le C(|\nabla\eta\|\nabla^{3}\eta|+|\nabla^{2}\eta|^{2}),
	\end{align}
	where $C$  is a constant independent of $\epsilon$.
	
	Additionally, for $k=2,3$, we have
	\begin{align}\label{4.17}
		\|\nabla^{k}\eta\|_{L^2(\Omega^{\epsilon})}
		\le\int_{0}^{t}\|\nabla^{k}v(s,x)\|_{L^2(\Omega^{\epsilon})}ds \le C\sqrt{t}\bigg(\int_{0}^{t}\|\nabla^{k}v(s,x)\|_{L^2(\Omega^{\epsilon})}^{2}ds\bigg)^{\frac{1}{2}}.
	\end{align}
	
	Using \eqref{4.16}, \eqref{4.17}, the Sobolev embedding theorem,  the Cauchy-Schwarz inequality, and Lemma \ref{lemB.3},  we obtain
	\begin{align*}
		|\mathcal{R}_{1}|\le&\int_{0}^{t}\int_{B^{+}}(|\bar{\partial}^{2}A\|q|+|\bar{\partial}A\|\bar{\partial}q|)(2|\xi\|\nabla\xi\|\bar{\partial}^{2}v|+|\xi|^{2}|\bar{\partial}^{2}\nabla v|)dyds \\
		&+\int_{0}^{t}\int_{B^{+}}2|A\|\bar{\partial}^{2}q\|\xi\|\nabla\xi\|\bar{\partial}^{2}v|dyds+C\int_{0}^{t}\int_{B^{+}}|A\|\nabla^{2}q\|\xi|^{2}|\nabla^{3}v|dyds \\
		\le&C\int_{0}^{t}\int_{B^{+}}(|\nabla\eta\|\nabla^{3}\eta\|q|+|\nabla^{2}\eta|^{2}|q|+|\nabla\eta\|\nabla^{2}\eta\|\bar{\partial}q|)(|\nabla^{2}v|+|\nabla^{3}v|)dyds \\
		&+C\int_{0}^{t}\int_{B^{+}}|\nabla\eta|^{2}|\bar{\partial}^{2}q\|\nabla^{2}v|dyds+\int_{0}^{t}\int_{B^{+}}|\nabla\eta|^{2}|\nabla^{2}q\|\nabla^{3}v|dyds\\
		\le &C\int_{0}^{t}\|\nabla^{3}\eta\|_{L^2(\Omega^{\epsilon})}\|q\|_{L^{\infty}(\Omega^{\epsilon})}\|\nabla^{2}v\|_{L^2(\Omega^{\epsilon})}ds\\
		&+C\int_{0}^{t}\|\nabla^{3}\eta\|_{L^2(\Omega^{\epsilon})}\|q\|_{L^{\infty}(\Omega^{\epsilon})}\|\nabla^{3}v\|_{L^2(\Omega^{\epsilon})}ds \\
		&+C\int_{0}^{t}\|\nabla^{2}\eta\|_{L^{6}(\Omega^{\epsilon})}^{2}\|q\|_{L^{6}(\Omega^{\epsilon})}\|\nabla^{2}v\|_{L^2(\Omega^{\epsilon})}ds\\
		&+C\int_{0}^{t}\|\nabla^{2}\eta\|_{L^{6}(\Omega^{\epsilon})}^{2}\|q\|_{L^{6}(\Omega^{\epsilon})}\|\nabla^{3}v\|_{L^2(\Omega^{\epsilon})}ds \\
		&+C\int_{0}^{t}\|\nabla^{2}\eta\|_{L^{4}(\Omega^{\epsilon})}\|\nabla q\|_{L^{4}(\Omega^{\epsilon})}\|\nabla^{2}v\|_{L^2(\Omega^{\epsilon})}ds\\
		&+C\int_{0}^{t}\|\nabla^{2}\eta\|_{L^{4}(\Omega^{\epsilon})}\|\nabla q\|_{L^{4}(\Omega^{\epsilon})}\|\nabla^{3}v\|_{L^2(\Omega^{\epsilon})}ds \\
		&+C\int_{0}^{t}\|\nabla^{2}q\|_{L^2(\Omega^{\epsilon})}\|\nabla^{2}v\|_{L^2(\Omega^{\epsilon})}ds\\
		&+C\int_{0}^{t}\|\nabla\eta\|_{L^{\infty}(\Omega^{\epsilon})}\|\nabla^{2}q\|_{L^2(\Omega^{\epsilon})}\|\nabla^{3}v\|_{L^2(\Omega^{\epsilon})}ds \\
		\le&C\sup_{t\in[0,T]}\|\nabla^{3}\eta\|_{L^2(\Omega^{\epsilon})}\|\nabla^{2}v\|_{L^2(\Omega^{\epsilon})}\sqrt{t}\bigg(\int_{0}^{t}\|q\|_{H^2(\Omega^{\epsilon})}^{2}ds\bigg)^{\frac{1}{2}} \\
		&+C\sup_{t\in[0,T]}\sqrt{t}\bigg(\int_{0}^{t}\|\nabla^{3}v\|_{L^2(\Omega^{\epsilon})}^{2}ds\bigg)^{\frac{1}{2}}\bigg(\int_{0}^{t}\|q\|_{H^2(\Omega^{\epsilon})}^{2}ds\bigg)^{\frac{1}{2}}\\
		&\cdot\bigg(\int_{0}^{t}\|\nabla^{3}v\|_{L^2(\Omega^{\epsilon})}^{2}ds\bigg)^{\frac{1}{2}} \\
		&+C\sup_{t\in[0,T]}\|\nabla^{2}\eta\|_{H^1(\Omega^{\epsilon})}^{2}\|\nabla^{2}v\|_{L^2(\Omega^{\epsilon})}\sqrt{t}\bigg(\int_{0}^{t}\|q\|_{H^1(\Omega^{\epsilon})}^{2}ds\bigg)^{\frac{1}{2}} \\
		&+C\sup_{t\in[0,T]}\sqrt{t}\bigg(\int_{0}^{t}\|\nabla^{2}v\|_{L^2(\Omega^{\epsilon})}^{2}ds\bigg)^{\frac{1}{2}}\|\nabla^{2}\eta\|_{H^1(\Omega^{\epsilon})}\bigg(\int_{0}^{t}\|q\|_{H^1(\Omega^{\epsilon})}^{2}ds\bigg)^{\frac{1}{2}}\\
		&\cdot\bigg(\int_{0}^{t}\|\nabla^{3}v\|_{L^2(\Omega^{\epsilon})}^{2}ds\bigg)^{\frac{1}{2}} \\
		&+C\sup_{t\in[0,T]}\|\nabla^{2}\eta\|_{H^1(\Omega^{\epsilon})}\|\nabla^{2}v\|_{L^2(\Omega^{\epsilon})}\sqrt{t}\bigg(\int_{0}^{t}\|\nabla q\|_{H^1(\Omega^{\epsilon})}^{2}ds\bigg)^{\frac{1}{2}} \\
		&+C\sup_{t\in[0,T]}\sqrt{t}\bigg(\int_{0}^{t}\|\nabla^{2}v\|_{H^1(\Omega^{\epsilon})}^{2}ds\bigg)^{\frac{1}{2}}\bigg(\int_{0}^{t}\|\nabla q\|_{H^1(\Omega^{\epsilon})}^{2}ds\bigg)^{\frac{1}{2}}\\
		&\cdot\bigg(\int_{0}^{t}\|\nabla^{3}v\|_{L^2(\Omega^{\epsilon})}^{2}ds\bigg)^{\frac{1}{2}} \\
		&+C\sup_{t\in[0,T]}\|\nabla^{2}v\|_{L^2(\Omega^{\epsilon})}\sqrt{t}\bigg(\int_{0}^{t}\|\nabla^{2}q\|_{L^2(\Omega^{\epsilon})}^{2}ds\bigg)^{\frac{1}{2}} \\
		&+C\sup_{t\in[0,T]}\sqrt{t}\bigg(\int_{0}^{t}\|\nabla v\|_{H^2(\Omega^{\epsilon})}^{2}ds\bigg)^{\frac{1}{2}}\bigg(\int_{0}^{t}\|\nabla^{2}q\|_{L^2(\Omega^{\epsilon})}^{2}ds\bigg)^{\frac{1}{2}}\\
		&\cdot\bigg(\int_{0}^{t}\|\nabla^{3}v\|_{L^2(\Omega^{\epsilon})}^{2}ds\bigg)^{\frac{1}{2}} \\
		\le&\delta\int_{0}^{t}\|q\|_{H^2(\Omega^{\epsilon})}^{2}ds+C_{\delta}T\big(\sup_{t\in[0,T]}\|\nabla^{3}\eta\|_{L^2(\Omega^{\epsilon})}\|\nabla^{2}v\|_{L^2(\Omega^{\epsilon})}\big)^{2} \\
		&+\delta\int_{0}^{t}\|q\|_{H^2(\Omega^{\epsilon})}^{2}ds+C_{\delta}T\big(\sup_{t\in[0,T]}\int_{0}^{t}\|\nabla^{3}v\|_{L^2(\Omega^{\epsilon})}^{2}ds\big)^{2} \\
		&+\delta\int_{0}^{t}\|q\|_{H^1(\Omega^{\epsilon})}^{2}ds+C_{\delta}T\big(\sup_{t\in[0,T]}\|\nabla^{2}\eta\|_{H^1(\Omega^{\epsilon})}^{2}\|\nabla^{2}v\|_{L^2(\Omega^{\epsilon})}\big)^{2} \\
		&+\delta\int_{0}^{t}\|q\|_{H^1(\Omega^{\epsilon})}^{2}ds+C_{\delta}T\big(\sup_{t\in[0,T]}(\int_{0}^{t}\|\nabla^{2}v\|_{H^1(\Omega^{\epsilon})}^{2}ds)\|\nabla^{2}\eta\|_{H^1(\Omega^{\epsilon})}\big)^{2} \\
		&+\delta\int_{0}^{t}\|\nabla q\|_{H^1(\Omega^{\epsilon})}^{2}ds+C_{\delta}T\big(\sup_{t\in[0,T]}\|\nabla^{2}\eta\|_{H^1(\Omega^{\epsilon})}\|\nabla^{2}v\|_{L^2(\Omega^{\epsilon})}\big)^{2} \\
		&+\delta\int_{0}^{t}\|\nabla q\|_{H^1(\Omega^{\epsilon})}^{2}ds+C_{\delta}T\big(\sup_{t\in[0,T]}(\int_{0}^{t}\|\nabla^{2}v\|_{H^1(\Omega^{\epsilon})}ds)\big)^{2} \\
		&+\delta\int_{0}^{t}\|\nabla^{2}q\|_{L^2(\Omega^{\epsilon})}ds+C_{\delta}T\big(\sup_{t\in[0,T]}\|\nabla^{2}v\|_{L^2(\Omega^{\epsilon})}\big)^{2} \\
		&+\delta\int_{0}^{t}\|\nabla^{2}q\|_{L^2(\Omega^{\epsilon})}ds+C_{\delta}T\big(\sup_{t\in[0,T]}\int_{0}^{t}\|\nabla v\|_{H^2(\Omega^{\epsilon})}ds\big)^{2}.
	\end{align*}
	
	Therefore, we obtain
	\begin{align}\label{1}
		|\mathcal{R}_{1}|\le\delta\sup_{t\in[0,T]}E^{\epsilon}(t)+C_{\delta}T\mathcal{P}(\sup_{t\in[0,T]}E^{\epsilon}(t)).
	\end{align}
	
	Next,  we consider $\mathcal{R}_{2}$,  for which we have
	\begin{align}\label{4.19}
		|\mathcal{R}_{2}|\le&\underbrace{\int_{0}^{t}\int_{B^{+}}|(A_{jr}A_{kr}-\delta_{jk})\bar{\partial}_{\alpha}\bar{\partial}_{\beta}\partial_{k}v_{i}\bar{\partial}_{\alpha}\bar{\partial}_{\beta}\partial_{j}v_{i}|dyds }_ {\mathcal{R}_{2a}}\nonumber\\
		&+\underbrace{\int_{0}^{t}\int_{B^{+}}|\bar{\partial}_{\alpha}\bar{\partial}_{\beta}(A_{jr}A_{kr}-\delta_{jk})\partial_{k}v_{i}\bar{\partial}_{\alpha}\bar{\partial}_{\beta}\partial_{j}v_{i}|dyds }_{\mathcal{R}_{2b}}\nonumber\\
		&\underbrace{+2\int_{0}^{t}\int_{B^{+}}|\bar{\partial}_{\alpha}(A_{jr}A_{kr}-\delta_{jk})\bar{\partial}_{\beta}\partial_{k}v_{i}\bar{\partial}_{\alpha}\bar{\partial}_{\beta}\partial_{j}v_{i}|dyds }_{\mathcal{R}_{2c}}.
	\end{align}
	By choosing $0<\theta<\delta$ and using \eqref{4.2}, we obtain
	$$
	|\mathcal{R}_{2a}|\le \int_{0}^{t}\|AA^\top -\mathrm{I}\|_{L^{\infty}(\Omega^{\epsilon})}\|\nabla^{3}v\|_{L^2(\Omega^{\epsilon})}^{2}ds 	\le\delta\sup_{t\in[0,T]}E^{\epsilon}(t).
	$$
	Next, we consider  $\mathcal{R}_{2b}$ and $\mathcal{R}_{2c}$ as given in \eqref{4.19}. By the Cauchy-Schwarz inequality, \eqref{4.16}, and \eqref{4.17}, for any $\delta>0$, we have
	\begin{align*}
		|\mathcal{R}_{2b}|\le&C\int_{0}^{t}\int_{B^{+}}\big(|\nabla\eta|^{2}(|\nabla\eta\|\nabla^{3}\eta|+|\nabla^{2}\eta|^{2})+|\nabla\eta|^{2}|\nabla^{2}\eta|^{2}\big)|\nabla v\|\nabla^{3}v|dyds \\
		\le&C\int_{0}^{t}\|\nabla^{3}\eta\|_{L^2(\Omega^{\epsilon})}\|\nabla v\|_{L^{\infty}(\Omega^{\epsilon})}\|\nabla^{3}v\|_{L^2(\Omega^{\epsilon})}ds\\
		& +C\int_{0}^{t}\|\nabla^{2}\eta\|_{L^{4}(\Omega^{\epsilon})}\|\nabla v\|_{L^{4}(\Omega^{\epsilon})}\|\nabla^{3}v\|_{L^2(\Omega^{\epsilon})}ds \\
		\le&\delta\sup_{t\in[0,T]}E^{\epsilon}(t)+C_{\delta}T\mathcal{P}\big(\sup_{t\in[0,T]}E^{\epsilon}(t)\big),
	\end{align*}
	and
	\begin{align*}
		|\mathcal{R}_{2c}|\le&C\int_{0}^{t}\int_{B^{+}}|\nabla\eta|^{3}|\nabla^{2}\eta\|\nabla^{2}v\|\nabla^{3}v|dyds \\
		\le&C\int_{0}^{t}\|\nabla^{2}\eta\|_{L^{4}(\Omega^{\epsilon})}\|\nabla^{2}v\|_{L^{4}(\Omega^{\epsilon})}\|\nabla^{3}v\|_{L^2(\Omega^{\epsilon})}ds \\
		\le&\delta\sup_{t\in[0,T]}E^{\epsilon}(t)+C_{\delta}T\mathcal{P}\big(\sup_{t\in[0,T]}E^{\epsilon}(t)\big).
	\end{align*}
	
	Thus,  by \eqref{4.19}, we obtain
	\begin{align}\label{2} |\mathcal{R}_{2}|\le\delta\sup_{t\in[0,T]}E^{\epsilon}(t)+C_{\delta}T\mathcal{P}\big(\sup_{t\in[0,T]}E^{\epsilon}(t)\big)
	\end{align}
	for any $\delta>0$.
	
	The integral $\mathcal{R}_{3}$ can also be split  into the following parts:
	\begin{align}\label{4.21}
		\mathcal{R}_{3}=&-2\int_{0}^{t}\int_{B^{+}}(\bar{\partial}_{\alpha}\bar{\partial}_{\beta}A_{jr}A_{kr}\partial_{k}v_{i}\varsigma\partial_{j}\varsigma\bar{\partial}_{\alpha}\bar{\partial}_{\beta}v_{i}+A_{jr}\bar{\partial}_{\alpha}\bar{\partial}_{\beta}A_{kr}\partial_{k}v_{i}\varsigma\partial_{j}\varsigma\bar{\partial}_{\alpha}\bar{\partial}_{\beta}v_{i})dyds \nonumber\\
		&-2\int_{0}^{t}\int_{B^{+}}A_{jr}A_{kr}\bar{\partial}_{\alpha}\bar{\partial}_{\beta}\partial_{k}v_{i}\varsigma\partial_{j}\varsigma\bar{\partial}_{\alpha}\bar{\partial}_{\beta}v_{i}dyds \nonumber\\
		&-2\int_{0}^{t}\int_{B^{+}}2\bar{\partial}_{\alpha}A_{jr}\bar{\partial}_{\beta}A_{kr}\partial_{k}v_{i}\varsigma\partial_{j}\varsigma\bar{\partial}_{\alpha}\bar{\partial}_{\beta}v_{i}dyds \nonumber\\
		&-2\int_{0}^{t}\int_{B^{+}}2(\bar{\partial}_{\alpha}A_{jr}A_{kr}\bar{\partial}_{\beta}\partial_{k}v_{i}\varsigma\partial_{j}\varsigma\bar{\partial}_{\alpha}\bar{\partial}_{\beta}v_{i}+A_{jr}\bar{\partial}_{\alpha}A_{kr}\bar{\partial}_{\beta}\partial_{k}v_{i}\varsigma\partial_{j}\varsigma\bar{\partial}_{\alpha}\bar{\partial}_{\beta}v_{i})dyds.
	\end{align}
	Using a similar argument as in the proof of \eqref{2}, we obtain
	\begin{align}\label{3}
		|\mathcal{R}_{3}|
		\le&C\int_{0}^{t}\int_{B^{+}}(|\nabla\eta\|\nabla^{3}\eta|+|\nabla^{2}\eta|^{2})|\nabla\eta|^{2}|\nabla v\|\varsigma\|\nabla\varsigma\|\nabla^{2}v|dyds\nonumber\\
		&+C\int_{0}^{t}\int_{B^{+}}|\nabla\eta|^{4}|\nabla^{3}v\|\varsigma\|\nabla\varsigma\|\nabla^{2}v|dyds\nonumber \\
		&+C\int_{0}^{t}\int_{B^{+}}|\nabla\eta|^{2}|\nabla^{2}\eta|^{2}|\nabla v\|\varsigma\|\nabla\varsigma\|\nabla^{2}v|dyds\nonumber \\
		&+C\int_{0}^{t}\int_{B^{+}}|\nabla\eta|^{3}|\nabla^{2}\eta\|\nabla^{2}v\|\varsigma\|\nabla\varsigma\|\nabla^{2}v|dyds \nonumber\\
		\le&\delta\sup_{t\in[0,T]}E^{\epsilon}(t)+C_{\delta}T\mathcal{P}\big(\sup_{t\in[0,T]}E^{\epsilon}(t)\big),
	\end{align}
	for any $\delta>0$.
	
	Next, we estimate several representative terms in $\mathcal{R}_{6}$. The other terms have similar structures and can yield similar results. By the Hölder inequality, we get
	$$
	\sum_{j=1}^3(\left\|\left(\left(F_0\right)_j \cdot \nabla\right) \eta\right\|_{H^2(\Omega^{\varepsilon})}
	\leq \left\|F_0\right\|_{H^3\left(\Omega^{\varepsilon}\right)}\|\eta\|_{H^3\left(\Omega^{\varepsilon}\right)} .
	$$
	By the Cauchy-Schwarz inequality
	$$
	\begin{aligned}
		& \left|\int_0^t \int_{B^{+}} \bar{\partial}_\alpha\left(F_0\right)_{k j} \bar{\partial}_\beta\left(\left(\left(F_0\right)_{m j} \partial_m \eta_i\right)\right) \partial_k\left(\zeta^2 \bar{\partial}_\alpha \bar{\partial}_\beta v_i\right) d y\right| \\
		\leq & C \int_0^t \int_{B^{+}}\left|\nabla F_0\right||\nabla ((F_{0})_{j}\cdot \nabla\eta_i)| \left(2|\varsigma||\nabla \varsigma|\left|\nabla^2 v\right|+|\varsigma|^2\left|\nabla^3 v\right|\right) d y d s \\
		\leq & C \int_0^t\left\|\nabla F_0\right\|_{L^4\left(\Omega^\epsilon\right)}\|\nabla ((F_{0})_{j} \cdot \nabla\eta_i)\|_{L^4\left(\Omega^\epsilon\right)}\left(\left\|\nabla^2 v\right\|_{L^2\left(\Omega^\epsilon\right)}+\left\|\nabla^3 v\right\|_{L^2\left(\Omega^\epsilon\right)}\right) d s \\
		\leq & \mathcal{Q}\left(\mathcal{M}_0\right)+\delta \sup _{t \in[0, T]} E^\epsilon(t)+C_\delta T \mathcal{P}\left(\sup _{t \in[0, T]} E^\epsilon(t)\right).
	\end{aligned}
	$$
	By the Hölder inequality, we also have, for any $\delta > 0$, that
	$$
	\begin{aligned}
		&\left|\int_{B^{+}}\left(F_0\right)_{k j} \bar{\partial}_\alpha \bar{\partial}_\beta\left(\left(\left(F_0\right)_{m j} \partial_m \eta_i\right)\right) \partial_k\left(\varsigma^2 \bar{\partial}_\alpha \bar{\partial}_\beta v_i\right) d y\right| \\
		& \leq C \int_0^t \int_{B^{+}}\left|F_0\right|\left|\nabla^2\left((F_{0})_{j} \cdot \nabla \eta\right)\right|\left|\left(2|\varsigma||\nabla \varsigma|\left|\nabla^2 v\right|+|\varsigma|^2\left|\nabla^3 v\right|\right)\right| d y d s \\
		& \leq C \int_0^t\left\|F_0\right\|_{L^{\infty}\left(\Omega^\epsilon\right)}\left\|\nabla^{2}((F_{0})_{j}\cdot \nabla \eta)\right\|_{L^2\left(\Omega^\epsilon\right)}\left\|\left(2|\varsigma||\nabla \varsigma|\left|\nabla^2 v\right|+|\varsigma|^2\left|\nabla^3 v\right|\right)\right\|_{L^2\left(\Omega^\epsilon\right)} d s \\
		& \leq C \sup _{t \in[0, T]}\left\|F_0\right\|_{H^2\left(\Omega^\epsilon\right)}^2\left\|(F_{0})_{j} \cdot \nabla \eta\right\|_{H^2\left(\Omega^\epsilon\right)}^2 \sqrt{t}\left(\int_0^t\left\|\nabla^3 v\right\|_{L^2\left(\Omega^\epsilon\right)}^2 d s\right)^{\frac{1}{2}} \\
		& \leq \mathcal{Q}\left(\mathcal{M}_0\right)+ \delta \sup _{t \in[0, T]} E^\epsilon(t)+C_\delta T \mathcal{P}\left(\sup _{t \in[0, T]} E^\epsilon(t)\right),
	\end{aligned}
	$$
	and
	$$
	\begin{aligned}
		&\left|\int_{B^{+}}\zeta^2\bar{\partial}_\beta\left(\left(F_0\right)_{k j}\right) \bar{\partial}_\alpha \bar{\partial}_\beta( \left(F_0\right)_{m j} \partial_m \eta_i)\partial_k\bar{\partial}_\alpha v_i d y\right| \\
		\leq & C\int_{0}^{t}\int_{B^{+}}|\varsigma|^{2}|\nabla F_{0}||\nabla^{2}v\|\nabla^{2}((F_0)_{j}\cdot \nabla \eta)|dyds \\
		\leq&C\int_{0}^{t}\|\nabla F_{0}\|_{L^{4}(\Omega^{\epsilon})}\|\nabla^{2}v\|_{L^{4}(\Omega^{\epsilon})}\|\nabla^{2}((F_0)_{j})\cdot \nabla \eta))\|_{L^2(\Omega^{\epsilon})}ds \\
		\le&C\sup_{t\in[0,T]}\|F_{0}\|_{H^2(\Omega^{\epsilon})}\|(F_0)_{j})\cdot \nabla \eta\|_{H^2(\Omega^{\epsilon})}\sqrt{t}\bigg(\int_{0}^{t}\|\nabla^{2}v\|_{H^1(\Omega^{\epsilon})}^{2}ds\bigg)^{\frac{1}{2}} \\	\leq&\mathcal{Q}\left(\mathcal{M}_0\right)+ \delta \sup _{t \in[0, T]} E^\epsilon(t)+C_\delta T \mathcal{P}\left(\sup _{t \in[0, T]} E^\epsilon(t)\right).
	\end{aligned}
	$$
	Therefore, Since the other terms in $\mathcal{R}_{5}$ have the same structurecombining with the inequality mentioned above, we have
	\begin{align}\label{4}		|\mathcal{R}_{6}|\le\mathcal{Q}\left(\mathcal{M}_0\right)+\delta\sup_{t\in[0,T]}E^{\epsilon}(t)+C_{\delta}T\mathcal{P}\big(\sup_{t\in[0,T]}E^{\epsilon}(t)\big),
	\end{align}
	for any $\delta>0$.

	Then, we consider $\mathcal{R}_{4}$ and $\mathcal{R}_{5}$. For $\mathcal{R}_{4}$, we have calculated that in Lagrangian coordinates, $\Delta u\circ\eta=\Delta^{\eta}v=\mathrm{div}^{\eta}\sym^{\eta}v=\partial_{j}(A_{jr}A_{kr}\partial_{k}v)$, so we can replace the term by "div"-term. As a result, using integrating by parts we can find $\mathcal{R}_{4}$ term vanish with the boundary condition.
	\\
	For $\mathcal{R}_{5}$, notice
	\begin{align*}
		\mathcal{R}_{5}=&-\int_{0}^{t}\int_{B^{+}}2\varsigma\partial_{k}\xi\bar{\partial}_{\alpha}\bar{\partial}_{\beta}A_{ki}\bar{\partial}_{\alpha}\bar{\partial}_{\beta}v_{i}dyds-\int_{0}^{t}\int_{B^{+}}\varsigma^{2}\bar{\partial}_{\alpha}\bar{\partial}_{\beta}A_{ki}\partial_{k}\bar{\partial}_{\alpha}\bar{\partial}_{\beta}v_{i}dyds, \\
		=:&\mathcal{R}_{5a}+\mathcal{R}_{5b}
	\end{align*}
	and the terms can be estimated below:
	\begin{align*}
		\mathcal{R}_{5a}\le&\int_{0}^{t}\int_{B{+}}2|\varsigma\|\nabla\varsigma\|\nabla^{2}A\|\nabla^{2}v|dyds \\
		\le&C\int_{0}^{t}\|\varsigma\|_{L^{\infty}(\Omega^{\epsilon})}\|\nabla\varsigma\|_{L^{\infty}(\Omega^{\epsilon})}\|\nabla\eta\|_{L^{\infty}(\Omega^{\epsilon})}\|\nabla^{3}\eta\|_{L^{2}(\Omega^{\epsilon})}\|\nabla^{2}v\|_{L^{2}}(\Omega^{\epsilon})ds \\
		&+C\int_{0}^{t}\|\varsigma\|_{L^{\infty}(\Omega^{\epsilon})}\|\nabla\varsigma\|_{L^{\infty}(\Omega^{\epsilon})}\|\nabla^{2}\eta\|_{L^{4}(\Omega^{\epsilon})}^{2}\|\nabla^{2}v\|_{L^{2}(\Omega^{\epsilon})}ds \\
		\le&C\sup_{t\in[0,T]}\|\nabla^{3}\eta\|_{L^{2}(\Omega^{\epsilon})}\sqrt{t}(\int_{0}^{t}\|\nabla^{2}v\|_{L^{2}(\Omega^{\epsilon})}^{2}ds)^{\frac{1}{2}} \\
		&+C\sup_{t\in[0,T]}\|\nabla^{2}\eta\|_{H^1(\Omega^{\epsilon})}^{2}\sqrt{t}(\int_{0}^{t}\|\nabla^{2}v\|_{L^{2}(\Omega^{\epsilon})}^{2}ds)^{\frac{1}{2}} \\
		\le&\delta\sup_{t\in[0,T]}E^{\epsilon}(t)+C_{\delta}T\mathcal{P}\big(\sup_{t\in[0,T]}E^{\epsilon}(t)\big),\\
		\mathcal{R}_{5b}\le&\int_{0}^{t}\int_{B^{+}}|\varsigma|^{2}|\nabla^{2}A\|\nabla^{3}v|dyds \\
		\le&C\int_{0}^{t}\|\varsigma\|_{L^{\infty}(\Omega^{\epsilon})}^{2}\|\nabla\eta\|_{L^{\infty}(\Omega^{\epsilon})}\|\nabla^{3}\eta\|_{L^{2}(\Omega^{\epsilon})}\|\nabla^{3}v\|_{L^{2}(\Omega^{\epsilon})}ds \\
		&+C\int_{0}^{t}\|\varsigma\|_{L^{\infty}(\Omega^{\epsilon})}^{2}\|\nabla^{2}\eta\|_{L^{4}(\Omega^{\epsilon})}^{2}\|\nabla^{2}v\|_{L^{2}(\Omega^{\epsilon})}ds \\
		\le&\delta\sup_{t\in[0,T]}E^{\epsilon}(t)+C_{\delta}T\mathcal{P}\big(\sup_{t\in[0,T]}E^{\epsilon}(t)\big),
	\end{align*}
	yielding
	\begin{align}\label{6}
		|\mathcal{R}_{5}|\le\delta\sup_{t\in[0,T]}E^{\epsilon}(t)+C_{\delta}T\mathcal{P}\big(\sup_{t\in[0,T]}E^{\epsilon}(t)\big).
	\end{align}
	
	Therefore, by summing over all boundary charts \( l \) in \eqref{4.15}, and applying \eqref{1}, \eqref{2}, \eqref{3}, \eqref{4}, \eqref{6}, and the trace theorem, we obtain
	\begin{align}\label{4.28}
		\int_{0}^{T}\|v(t,\cdot)\|_{2.5,\partial\Omega^{\epsilon}}^{2}dt\le \mathcal{Q}\left(\mathcal{M}_0\right)+ \delta \sup _{t \in[0, T]} E^\epsilon(t)+C_\delta T \mathcal{P}\left(\sup _{t \in[0, T]} E^\epsilon(t)\right).
	\end{align}
	This completes the proof of Proposition \ref{prop4.2}.
\end{proof}

\subsection{Time-differentiated system estimates}

Differentiating system \eqref{2.3} with respect to time yields:

\begin{equation}\label{4.32}\begin{cases}
		\partial_{t}\eta_i=v_i,  &\text{in} \ [0,T]\times\Omega^{\epsilon}, \\
		\partial_{t}^{2}v_{i}+\peta_{i}\partial_{t}q-\divop^\eta \partial_t \symeta v_{i} -\partial_t(\sum_{j=1}^3(F_0)_j\cdot\nabla)^2\eta_i)\\
		\qquad\quad=\peta_i v_l\peta_lq-\peta_j v_l \peta_l(\symeta v)_{ij}& \text{in} \ [0,T]\times\Omega^{\epsilon},  \\
		\divop^\eta \partial_{t}v=\peta_i v_l\peta_lv_{i},  & \text{in} \ [0,T]\times\Omega^{\epsilon}, \\
		\partial_{t}[(((\symeta v)_{ij}-q\delta_{ij})+(\mathbb{F}_{im}\mathbb{F}_{jm}-\delta_{ij}))n_j]=0, & \text{on} \ [0,T]\times\partial\Omega^{\epsilon}, \\
		(\eta,v,\partial_{t}v,\partial_{t}((F_{0})_{j}\cdot\nabla \eta))=(e,u_{0},u_{1}, F_{1j}) &\text{in} \ \{t=0\}\times\Omega^{\epsilon},
	\end{cases}
\end{equation}
where initial accelerations are defined as:
\begin{align*}
	&u_{1i} = \Delta u_{0i} - \partial_i p_0 + \partial_j (F_{0im} F_{0jm}), \\
	&F_{1j}=(F_{0})_{j}\cdot\nabla v(0,\cdot),
\end{align*}
satisfying the estimate:
\begin{equation}\label{4.39a}
	\begin{aligned}
		\|u_1\|_{L^2(\Omega^\epsilon)} + \sum_{j=1}^{3}\|F_{1j}\|_{L^2(\Omega^\epsilon)} \leq C P(\mathcal{M}_0),
	\end{aligned}
\end{equation}
with $P(\mathcal{M}_0)$ being an $\epsilon$-independent polynomial. Therefore, we have:
\begin{proposition}\label{prop4.4}
	The time derivatives satisfy the following uniform estimate:
	
	\begin{align}\label{4.34A}
		\begin{split}
			\sup_{t\in[0,T]}  \|v_t(t)\|_{L^2(\Omega^\epsilon)}^2   & + \int_0^T \|\nabla v_t\|_{L^2(\Omega^\epsilon)}^2 dt \\
			\leq \mathcal{Q}(\mathcal{M}_0) & + \delta \sup_{t\in[0,T]} E^\epsilon(t) + C_\delta T^{\frac{1}{2}} \mathcal{P}\left( \sup_{t\in[0,T]} E^\epsilon(t) \right)
		\end{split}
	\end{align}
	
	for any $\delta > 0$, where $\mathcal{Q}(\mathcal{M}_0)$ is an $\epsilon$-independent polynomial in the initial energy.
\end{proposition}

\begin{proof}
	We consider test functions $\phi \in \mathcal{V}(t) := \{ \phi \in H^1(\Omega^\epsilon; \mathbb{R}^3) : \divop^\eta \phi = A_{ji} \partial_j \phi_i = 0 \}$. Taking the $L^2$ inner product of \eqref{4.32} with $\phi_i$ gives
	\begin{align*}
		\int_{\Omega^\epsilon} \partial_t^2 v_i \phi_i  dx
		&+ \int_{\Omega^\epsilon} \partial_i^\eta \partial_t q  \phi_i  dx
		- \int_{\Omega^\epsilon} (\operatorname{div}^\eta \partial_t \sym^\eta v_i) \phi_i  dx \\
		&- \int_{\Omega^\epsilon} \partial_t(\sum_{j=1}^3((F_0)_j\cdot\nabla)^2\eta_i) \phi_i  dx \\
		= \int_{\Omega^\epsilon} & \partial_i^\eta v_l \partial_l^\eta q  \phi_i  dx
		- \int_{\Omega^\epsilon} \partial_j^\eta v_l \partial_l^\eta (\sym^\eta v)_{ij} \phi_i  dx.
	\end{align*}
	Applying integration by parts to key terms, we have
	\begin{align*}
		\int_{\Omega^\epsilon} \partial_i^\eta \partial_t q  \phi_i  dx
		&= \int_{\partial \Omega^\epsilon} n_i \partial_t q  \phi_i  dS, \\
		\int_{\Omega^\epsilon} (\operatorname{div}^\eta \partial_t \sym^\eta v_i) \phi_i  dx
		&= \int_{\partial \Omega^\epsilon} n_j \partial_t (\sym^\eta v)_{ij} \phi_i  dS
		- \int_{\Omega^\epsilon} \partial_t (\sym^\eta v)_{ij} \partial_j^\eta \phi_i  dx
	\end{align*}
	and
	\begin{align*}
		&\int_{\Omega^{\varepsilon}}\partial_ t\left(\sum_{j=1}^3\left[\left(F_0\right)_{m j} \partial_ m\left(F_0\right)_{k j} \partial_k \eta_ i\right]\right) \phi_i d x
		\\&=\int_{\Omega^{\varepsilon}} \partial_t\left(\sum_{j=1}^3\partial_m\left[\left(F_0\right)_{m j}\left(F_0\right)_{k j} \partial_k \eta_ i\right] \right)\phi_i d x
		\\&=\sum_{j=1}^3\int_{\Omega^{\varepsilon}}\partial_t\partial_m\left(A_{m k} \mathbb{F}_{k j} A_{k h}  \mathbb{F}_{h j} \partial_k \eta_i\right) \phi_idx
		\\&=\sum_{j=1}^3\left(-\int_{\Omega^{\varepsilon}}\partial_k^{\eta}\left( \partial_t(\mathbb{F}_{k j} \mathbb{F}_{i j}) \phi_i\right)dx+\int_{\Omega^\epsilon} \partial_j^\eta v_l \partial_l^\eta (\mathbb{F}_{k j} \mathbb{F}_{i j}) \phi_i dx+\int_{\Omega^\epsilon} \partial_t (\mathbb{F}_{k j} \mathbb{F}_{i j}) \partial_k^\eta \phi_idx\right)
		\\&= \sum_{j=1}^3\left(-\int_{\partial \Omega^\epsilon} n_k \partial_t (\mathbb{F}_{kj} \mathbb{F}_{ij}) \phi_i dS+\int_{\Omega^\epsilon} \partial_k^\eta v_l \partial_l^\eta ((F_{0})_{j}\cdot\nabla\eta_{k}(F_{0})_{j}\cdot\nabla\eta_{i}) \phi_i dx\right.\\&\left.+ \int_{\Omega^\epsilon} \partial_t ((F_{0})_{j}\cdot\nabla\eta_{k}(F_{0})_{j}\cdot\nabla\eta_{i}) \partial_k^\eta \phi_i dx\right).
	\end{align*}
	Due to the boundary condition in \eqref{4.32} and the fact that $n$ is dependent of $t$, we obtain
	\begin{align*}
		&-\int_{\partial \Omega^{\epsilon}}n_i\partial_{t}q\phi_{i}dS+\int_{\partial\Omega^{\epsilon}}n_k\partial_t(\symeta v)_{ik}\phi_{i}dS+\int_{\partial\Omega^{\epsilon}} \sum_{j=1}^3n_{k}\partial_{t}(\mathbb{F}_{kj}\mathbb{F}_{ij})\phi_{i}\, dS\\
		=&\int_{\partial \Omega^{\epsilon}}\partial_t[(((\symeta v)_{ik}-q\delta_{ik})+(\sum_{j=1}^3\mathbb{F}_{kj}\mathbb{F}_{ij}-\delta_{ki}))n_k]\phi_{i}\, dS
		\\ & -\int_{\partial \Omega^{\epsilon}} \partial_t n_k\left[\left(\left(\symeta v\right)_{i k}-q
		\delta_{i k}\right)+\left(\sum_{j=1}^3\mathbb{F}_{kj} \mathbb{F}_{ij}-\delta_{i j}\right)\right] \phi_i dS
		\\ =&-\int_{\partial \Omega^{\epsilon}} \partial_t n_k\left[\left(\left(\symeta v\right)_{i k}-q
		\delta_{i k}\right)+\sum_{j=1}^3\left(\mathbb{F}_{kj} \mathbb{F}_{ij}-\delta_{k i}\right)\right] \phi_i dS.
	\end{align*}
	
	Combining these results yields the weak formulation
	\begin{align}\label{4.42a}
		\int_{\Omega^{\epsilon}}\partial_{t}^{2}v_{i}\phi_{i}\, dx
		=&\int_{\Omega^{\epsilon}}\peta_i v_l\peta_lq \phi_{i}\, dx
		-\int_{\Omega^{\epsilon}}\peta_j v_l \peta_l(\symeta v)_{ij}\phi_{i}\, dx\\
		&- \int_{\Omega^\epsilon} \partial_k^\eta v_l \partial_l^\eta ((F_{0})_{j}\cdot\nabla\eta_{k}(F_{0})_{j}\cdot\nabla\eta_{i}) \phi_i\, dx-\int_{\Omega^{\epsilon}} \partial_t(\symeta v)_{ij}\peta_j\phi_{i}\, dx\nonumber\\
		&-\int_{\Omega^\epsilon} \partial_t ((F_{0})_{j}\cdot\nabla\eta_{k}(F_{0})_{j}\cdot\nabla\eta_{i}) \partial_k^\eta \phi_i \, dx\nonumber\\
		&-\int_{\partial \Omega^{\epsilon}} \partial_t n_k\left[\left(\left(\symeta v\right)_{i k}-q
		\delta_{i k}\right)+\sum_{j=1}^3\left(\mathbb{F}_{kj} \mathbb{F}_{ij}-\delta_{k i}\right)\right] \phi_i dx.\nonumber
	\end{align}
	Inspired by \cite{CS5}, we define a vector field $w$ satisfying
	\begin{equation}\label{4.36}\begin{cases}
			\divop^\eta w=\peta_i v_l\peta_lv_{i} & \text{in }  \Omega^{\epsilon},  \\
			w_i=\varphi(t) n_i &\text{on }  \partial\Omega^{\epsilon},
		\end{cases}
	\end{equation}
	where $\varphi(t)$ is determined by the compatibility condition:
	$$\varphi(t):=\frac{1}{|\partial\Omega^{\epsilon}|}\int_{\Omega^{\epsilon}}\peta_i v_l\peta_lv_{i}\, dx.$$
	This ensures
	\begin{align*}
		\int_{\Omega^\epsilon} \divop^\eta w  dx = \int_{\partial\Omega^\epsilon} w \cdot n  dS = \varphi(t) |\partial\Omega^\epsilon|.
	\end{align*}
	
	Since $0 < C_0 \leq |A| \leq C$ (Lemma \ref{lem4.1}), we can express $w = \nabla^\eta \psi$ where $\psi$ solves the Neumann problem:
	\begin{align*}
		\begin{cases}
			\Delta^\eta \psi = \partial_i^\eta v_l \partial_l^\eta v_i \quad &\text{in } \Omega^\epsilon, \\
			\partial_n^\eta \psi = \varphi(t) \quad &\text{on } \partial\Omega^\epsilon.
		\end{cases}
	\end{align*}
	By elliptic regularity \cite{Girault1986,CS2,CS5}, for $k=1,2$, we have
	\begin{align}\label{4.37}
		\|w(t)\|_{H^k(\Omega^\epsilon)} \leq C \left( \|\partial_i^\eta v_l \partial_l^\eta v_i\|_{H^{k-1}(\Omega^\epsilon)} + \|\varphi(t)\|_{H^{k-\frac{3}{2}}(\partial\Omega^\epsilon)} \right),
	\end{align}
	with $C > 0$ independent of $\epsilon$.
	
	
	Since $\varphi(t)$ is constant on $\partial\Omega^\epsilon$, its $H^{-\frac{1}{2}}$ norm simplifies to
	\[
	\|\varphi(t)\|_{H^{-\frac{1}{2}}(\partial\Omega^\epsilon)}^2 \leq C |\varphi(t)|^2,
	\]
	where $C$ depends on $|\partial\Omega^\epsilon|$. Expanding $|\varphi(t)|^2$, we get
	\begin{align*}
		|\varphi(t)|^2 &= \frac{1}{|\partial\Omega^\epsilon|^2} \left( \int_{\Omega^\epsilon} \partial_i^\eta v_l \partial_l^\eta v_i  dx \right)^2 \\
		&\leq C \left( \int_{\Omega^\epsilon} \left[ \partial_i^\eta v_l(0) + \int_0^t \partial_s (\partial_i^\eta v_l(s))  ds \right] \left[ \partial_l^\eta v_i(0) + \int_0^t \partial_s (\partial_l^\eta v_i(s))  ds \right] dx \right)^2.
	\end{align*}
	
	By Lemma \ref{lem:commutators}, we have
	\begin{align*}
		\peta_i v_l(0)&=\partial_iv_l(0), \quad \peta_l v_i(0)=\partial_l v_i(0), \\
		\partial_s (\partial_i^\eta v_l(s)) &= [\partial_s, \partial_i^\eta] v_l + \partial_i^\eta \partial_s v_l = -\partial_i^\eta v_k \partial_k^\eta v_l + \partial_i^\eta \partial_s v_l, \\
		\partial_s (\partial_l^\eta v_i(s)) &= [\partial_s, \partial_l^\eta] v_i + \partial_l^\eta \partial_s v_i = -\partial_l^\eta v_m \partial_m^\eta v_i + \partial_l^\eta \partial_s v_i.
	\end{align*}
	
	We estimate a representative term using H\"older's inequality and Sobolev embedding $H^1 \subset L^4$
	\begin{align*}
		& \int_{\Omega^\epsilon} \partial_i^\eta v_l(0) \left( \int_0^t \partial_s (\partial_l^\eta v_i(s))  ds \right) dx \\
		&\quad \leq \left\| \partial_i^\eta v_l(0) \right\|_{L^2(\Omega^{\epsilon})} \left\| \int_0^t \left( -\partial_l^\eta v_m \partial_m^\eta v_i + \partial_l^\eta \partial_s v_i \right) ds \right\|_{L^2(\Omega^{\epsilon})} \\
		&\quad \leq \mathcal{M}_0^{1/2} \left( \int_0^t \| \nabla v_m \|_{L^4(\Omega^{\epsilon})} \| \nabla v_i \|_{L^4(\Omega^{\epsilon})} ds + \int_0^t \| \nabla v_s \|_{L^2(\Omega^{\epsilon})} ds \right) \\
		&\quad \leq \mathcal{M}_0^{1/2} \left( \int_0^t \| \nabla v \|_{H^1(\Omega^{\epsilon})}^2 ds + \int_0^t \| \nabla v_s \|_{L^2(\Omega^{\epsilon})} ds \right) \\
		&\quad \leq \mathcal{M}_0^{1/2} \left( T^{1/2} \left( \int_0^t \|\nabla v \|_{H^1(\Omega^{\epsilon})}^4 ds \right)^{1/2} + T^{1/2} \left( \int_0^t \| \nabla v_s \|_{L^2(\Omega^{\epsilon})}^2 ds \right)^{1/2} \right) \\
		&\quad \leq \mathcal{M}_0 +T \mathcal{P} \left( \sup_{t \in [0,T]} E^\epsilon(t) \right).
	\end{align*}
	Combining all estimates with elliptic regularity \eqref{4.37},  we have
	\begin{align}\label{4.38}
		\sup_{t\in[0,T]}\|w(t,\cdot)\|_{H^1(\Omega^{\epsilon})}^{2}+\int_{0}^{T}\|w(s,\cdot)\|_{H^2(\Omega^{\epsilon})}^{2}ds\le\mathcal{M}_{0}+T^{\frac{1}{2}}\mathcal{P}\big(\sup_{t\in[0,T]}E^{\epsilon}(t)\big).
	\end{align}
	
	Similarly, the time derivative $w_t$ satisfies
	\begin{align}\label{4.39}
		\begin{cases}
			\operatorname{div}^\eta w_t = \partial_i^\eta v_l \partial_l^\eta w_i + \partial_t (\partial_i^\eta v_l \partial_l^\eta v_i) & \text{in } \Omega^\epsilon, \\
			w_t = \varphi'(t) n + \varphi(t) \partial_t n & \text{on } \partial\Omega^\epsilon.
		\end{cases}
	\end{align}
	Since $n \cdot n = 1$, we have $n \cdot \partial_t n = 0$, implying
	\begin{align}\label{4.40}
		w_t \cdot n = \varphi'(t) \quad \text{on } \partial\Omega^\epsilon.
	\end{align}
	Elliptic regularity gives
	\begin{align}\label{4.41}
		\|w_t\|_{H^1(\Omega^\epsilon)} \leq C \left( \|\partial_i^\eta v_l \partial_l^\eta w_i + \partial_t (\partial_i^\eta v_l \partial_l^\eta v_i)\|_{L^2(\Omega^\epsilon)} + \|\varphi'(t)\|_{H^{-\frac{1}{2}}(\partial\Omega^\epsilon)} \right).
	\end{align}
	From Lemma \ref{lem:commutators}, we get
	\begin{align*}
		\int_0^T|\varphi'(t)|^2 dt\leq & \frac{4}{|\partial\Omega^{\epsilon}|^2}\int_0^T\left(\int_{\Omega^{\epsilon}}|(\peta_i \partial_t v_l-\peta_iv_m\peta_m v_l)\peta_lv_{i}|\, dx\right)^2dt\\
		\leq &C\int_0^T\|\nabla^\eta v_t\|_{L^2(\Omega^\epsilon)}^2dt\sup_{t\in [0,T]}\|\nabla^\eta v\|_{L^2(\Omega^\epsilon)}^2\\
		&+C\sup_{t\in [0,T]}\|\nabla^\eta v\|_{L^2(\Omega^\epsilon)}^4\int_0^T\|\nabla^\eta v\|_{H^2(\Omega^\epsilon)}^2dt\\
		\leq & \mathcal{P}\big(\sup_{t\in[0,T]}E^{\epsilon}(t)\big).
	\end{align*}
	Combining with \eqref{4.41}, we obtain
	\begin{align}\label{4.42} \int_{0}^{T}\|w_{t}\|_{H^1(\Omega^{\epsilon})}^{2}dt\le\mathcal{P}\big(\sup_{t\in[0,T]}E^{\epsilon}(t)\big).
	\end{align}
	
	Due to \eqref{4.36} and \eqref{4.32},  $\partial_{t}v-w\in\mathcal{V}(t)$. Therefore, we can set  $\phi=\partial_{t}v-w$ in \eqref{4.4}, yielding
	
	\begin{align}\label{4.43a}
		&\frac{1}{2}\frac{d}{dt}\|\partial_{t}v\|_{L^2(\Omega^{\epsilon})}^{2}\, dx\nonumber
		\\=&\int_{\Omega^{\epsilon}}\partial_t^2 v_iw_i\, dx+\int_{\Omega^{\epsilon}}\peta_i v_l\peta_lq (\partial_{t}v_i-w_i)\, dx
		-\int_{\Omega^{\epsilon}}\peta_j v_l \peta_l(\symeta v)_{ij}(\partial_{t}v_i-w_i)\, dx\nonumber\\
		&- \sum_{j=1}^3\int_{\Omega^\epsilon} \partial_k^\eta v_l \partial_l^\eta ((F_{0})_{j}\cdot\nabla\eta_{k}(F_{0})_{j}\cdot\nabla\eta_{i}) (\partial_{t}v_i-w_i)\, dx\\&-\int_{\Omega^{\epsilon}} \partial_t(\symeta v)_{ij}\peta_j(\partial_{t}v_i-w_i)\, dx\nonumber\\
		&-\sum_{j=1}^3\int_{\Omega^\epsilon} \partial_t ((F_{0})_{j}\cdot\nabla\eta_{k}(F_{0})_{j}\cdot\nabla\eta_{i}) \partial_k^\eta (\partial_{t}v_i-w_i) \, dx\nonumber\\
		&-\int_{\partial \Omega^{\epsilon}} \partial_t n_k\left[\left(\left(\symeta v\right)_{i k}-q
		\delta_{i k}\right)+\sum_{j=1}^3\left(\mathbb{F}_{kj} \mathbb{F}_{ij}-\delta_{k i}\right)\right] (\partial_{t}v_i-w_i) dS.
	\end{align}
	For the integral involving double $\nabla^\eta\partial_t v$, we have, by \eqref{eq.sysnv}, that
	\begin{align}\label{4.44}
		&-\int_{\Omega^{\epsilon}} \partial_t(\symeta v)_{ij}\peta_j\partial_{t}v_i\, dx
		=-\int_{\Omega^{\epsilon}} \partial_t(\peta_i v_j+\peta_j v_i)\peta_j\partial_{t}v_i\, dx\nonumber\\
		=&-\int_{\Omega^{\epsilon}} (\peta_i \partial_tv_j-\peta_i v_l\peta_l v_j+\peta_j\partial_t v_i-\peta_j v_l\peta_l v_i)\peta_j\partial_{t}v_i\, dx\nonumber\\
		=&-\frac{1}{2}\int_{\Omega^{\epsilon}}|\symeta \partial_tv|^2\, dx+\int_{\Omega^{\epsilon}}\peta_i v_l\peta_l v_j\peta_j\partial_{t}v_i\, dx
		+\int_{\Omega^{\epsilon}}\peta_j v_l\peta_l v_i\peta_j\partial_{t}v_i\, dx\nonumber\\
		=&-\frac{1}{2}\|\symeta v_t\|_{L^2(\Omega^{\epsilon})}^2+\int_{\Omega^{\epsilon}}\peta_i v_l\peta_l v_j\peta_j\partial_{t}v_i\, dx+\int_{\Omega^{\epsilon}}\peta_j v_l\peta_l v_i\peta_j\partial_{t}v_i\, dx.
	\end{align}
	Combining  equations \eqref{4.43a} and \eqref{4.44}, we arrive at
	
	\begin{align}\label{4.45}
		&\frac{1}{2}\|\partial_{t}v\|_{L^2(\Omega^{\epsilon})}^{2}+\frac{1}{2}\int_{0}^{t}\|\symeta v_s(s,\cdot)\|_{L^2(\Omega^{\epsilon})}^{2}ds \nonumber\\
		=&\frac{1}{2}(\|u_{1}\|_{L^2(\Omega^{\epsilon})}^{2}
		+\underbrace{\int_{0}^{t}\int_{\Omega^{\epsilon}}\peta_i v_l\peta_lq (\partial_{s}v_i-w_i)\, dxds}_{\sym_2}
		\\& \underbrace{-\int_{0}^{t}\int_{\Omega^{\epsilon}}\peta_j v_l \peta_l(\symeta v)_{ij}(\partial_{s}v_i-w_i)\, dxds}_{\sym_3}\nonumber\\
		& \underbrace{- \sum_{j=1}^3 \int_{0}^{t}\int_{\Omega^{\epsilon}}\partial_k^\eta v_l \partial_l^\eta ((F_{0})_{j}\cdot\nabla\eta_{k}(F_{0})_{j}\cdot\nabla\eta_{i}) (\partial_{s}v_i-w_i)\, dxds}_{\sym_4}\nonumber\\
		&+\underbrace{\int_{0}^{t}\int_{\Omega^{\epsilon}}\peta_i v_l\peta_l v_j\peta_j\partial_{s}v_i\, dxds}_{\sym_5}+\underbrace{\int_{0}^{t}\int_{\Omega^{\epsilon}}\peta_j v_l\peta_l v_i\peta_j\partial_{s}v_i\, dxds}_{\sym_6}\nonumber\\
		&+\underbrace{\int_{0}^{t}\int_{\Omega^{\epsilon}} \partial_s(\symeta v)_{ij}\peta_jw_i\, dxds}_{\sym_7}\\&\underbrace{-\sum_{j=1}^3\int_{0}^{t} \int_{\Omega^\epsilon} \partial_t ((F_{0})_{j}\cdot\nabla\eta_{k}(F_{0})_{j}\cdot\nabla\eta_{i}) \partial_k^\eta (\partial_{s}v_i-w_i) \, dxds}_{\sym_8}\nonumber\\
		&
		\underbrace{-\int_{0}^{t}\int_{\partial \Omega^{\epsilon}} \partial_s n_k\left[\left(\left(\symeta v\right)_{i k}-q
			\delta_{i k}\right)+\sum_{j=1}^3\left(\mathbb{F}_{kj} \mathbb{F}_{ij}-\delta_{k i}\right)\right] (\partial_{s}v_i-w_i)dSds.}_{\sym_9}
	\end{align}
	
	Next, we  estimate the terms $\sym_{i}$ ($i=1,\cdots,9$). By \eqref{4.2} and \eqref{4.3}, the matrix $A$ satisfies the uniform bounds
	\begin{align*}
		c_0 \leq \|A\|_{L^\infty(\Omega^\epsilon)} \leq C
	\end{align*}
	for constants $c_0, C > 0$ independent of $\epsilon$, where the upper bound follows from
	\begin{align*}
		\|A\|_{L^\infty} &\leq C \|\nabla \eta\|_{L^\infty(\Omega^\epsilon)}^2 \\
		&\leq C (\|\nabla \eta - I\|_{L^\infty(\Omega^\epsilon)} + \|I\|_{L^\infty(\Omega^\epsilon)})^2 \\
		&\leq C (\theta^2 + 1)^2 \leq C,
	\end{align*}
	and the lower bound holds since $A$ is invertible with $\det A = 1$. These uniform bounds will be used implicitly in subsequent estimates.
	By integrating by parts, \eqref{4.38} and \eqref{4.42}, we obtain
	\begin{align}\label{S1}
		|\sym_1|\le& \int_{0}^{t}\int_{\Omega^{\epsilon}}|\partial_{s}v\cdot \partial_{s}w|\, dxds+\bigg|\left.\int_{\Omega^{\epsilon}}\partial_{s}v\cdot w\, dx\right|_{0}^{t}\bigg| \nonumber\\
		\le &\big(\sup_{t\in[0,T]}\|v_t\|_{L^2(\Omega^{\epsilon})}\big) t^{1/2}\bigg(\int_{0}^{t}\|w_s\|_{L^2(\Omega^{\epsilon})}^2\, ds\bigg)^{1/2}\nonumber\\
		&+\frac{\delta}{2}\sup_{t\in[0,T]}\|v_{t}\|_{L^2(\Omega^{\epsilon})}^{2}+C_{\delta}\sup_{t\in[0,T]}\|w\|_{L^2(\Omega^{\epsilon})}^{2}\nonumber \\
		\le &\delta \sup_{t\in[0,T]}\|v_{t}\|_{L^2(\Omega^{\epsilon})}^{2}+C_{\delta}\sup_{t\in[0,T]}\|w\|_{L^2(\Omega^{\epsilon})}^{2}+tC_\delta \int_{0}^{t}\|w_s\|_{L^2(\Omega^{\epsilon})}^2\, ds\nonumber \\
		\le &\delta\sup_{t\in[0,T]}E^{\epsilon}(t)+C_\delta \big(\mathcal{M}_0+T^{1/2}\mathcal{P}\big(\sup_{t\in[0,T]}E^{\epsilon}(t)\big)\big)+C_\delta T\mathcal{P}\big(\sup_{t\in[0,T]}E^{\epsilon}(t)\big) \nonumber \\
		\le &C_\delta \mathcal{M}_0+\delta\sup_{t\in[0,T]}E^{\epsilon}(t)+ C_{\delta}T^{\frac{1}{2}}\mathcal{P}\big(\sup_{t\in[0,T]}E^{\epsilon}(t)\big),
	\end{align}
	for any $\delta>0$.
	By the H\"older inequality ($L^4$-$L^4$-$L^2$), the Sobolev embedding theorem ($H^1\subset L^4$) and \eqref{4.38}, we get
	\begin{align}\label{S2}
		&|\sym_2|\leq C\int_{0}^{t} \|\nabla v\|_{H^1(\Omega^{\epsilon})}\|\nabla q\|_{H^1(\Omega^{\epsilon})}(\|v_s\|_{L^2(\Omega^{\epsilon})}+\|w\|_{L^2(\Omega^{\epsilon})})ds\nonumber\\
		\leq  & CT^{1/2}\sup_{t \in [0, T]}\|v\|_{H^2(\Omega^{\epsilon})}\sup_{t \in [0, T]}(\|v_t\|_{L^2(\Omega^{\epsilon})}+ \|w\|_{L^2(\Omega^{\epsilon})})\left(\int_0^T\|q\|_{H^2(\Omega^{\epsilon})}^2ds\right)^{1/2}\nonumber\\
		\leq &\delta \sup_{t \in [0, T]}\|v\|_{H^2(\Omega^{\epsilon})}^2+C_\delta T\sup_{t \in [0, T]}(\|v_t\|_{L^2(\Omega^{\epsilon})}^2+ \|w\|_{L^2(\Omega^{\epsilon})}^2)\int_0^T\|q\|_{H^2(\Omega^{\epsilon})}^2ds\nonumber\\
		\leq &\delta \sup_{t \in [0, T]}\mathcal{E}(t)+C_\delta T\left(\sup_{t \in [0, T]}\mathcal{E}(t)+\mathcal{M}_0+T^{1/2}\mathcal{P}\big(\sup_{t\in[0,T]}E^{\epsilon}(t)\big)\right)\sup_{t\in[0,T]}E^{\epsilon}(t)\nonumber\\
		\leq &C_\delta \mathcal{M}_0+\delta\sup_{t\in[0,T]}E^{\epsilon}(t)+ C_{\delta}T\mathcal{P}\big(\sup_{t\in[0,T]}E^{\epsilon}(t)\big),
	\end{align}
	Similarly, by the H\"older inequality ($L^\infty$-$L^2$-$L^2$), the Sobolev embedding theorem ($H^2\subset L^\infty$) and \eqref{4.38}, we have
	\begin{align}\label{S3}
		|\sym_3|&\leq C\int_{0}^{t} \|\nabla v\|_{H^2(\Omega^{\epsilon})}\|v\|_{H^2(\Omega^{\epsilon})}(\|v_s\|_{L^2(\Omega^{\epsilon})}+\|w\|_{L^2(\Omega^{\epsilon})})ds\nonumber\\
		&\leq C_\delta \mathcal{M}_0+\delta\sup_{t\in[0,T]}E^{\epsilon}(t)+ C_{\delta}T\mathcal{P}\big(\sup_{t\in[0,T]}E^{\epsilon}(t)\big).
	\end{align}
	
	From the H\"older inequality ($L^6$-$L^6$-$L^6$-$L^2$), the Sobolev embedding theorem ($H^1 \subset L^6$) and \eqref{4.41}, it follows	
	\begin{align*}
		|\sym_4|&\leq C\int_{0}^{t} \|\nabla v\|_{H^1(\Omega^{\epsilon})}\|\nabla\left(\left(F_0\right)_j \cdot \nabla\right)  \eta\|_{H^1(\Omega^{\epsilon})}\|\left(\left(F_0\right)_j \cdot \nabla\right) \eta\|_{H^1(\Omega^{\epsilon})}\\&\quad\left(\|v_s\|_{L^2(\Omega^{\epsilon})}+\|w\|_{L^2(\Omega^{\epsilon})}\right)ds\\
		&\leq C_\delta \mathcal{M}_0+\delta\sup_{t\in[0,T]}E^{\epsilon}(t)+ C_{\delta}T^2\mathcal{P}\big(\sup_{t\in[0,T]}E^{\epsilon}(t)\big).
	\end{align*}
	
	By the H\"older inequality  and Sobolev embedding theorem, we get
	\begin{align}\label{S45}
		|\mathcal{S}_5+\mathcal{S}_6|\le C\int_0^t\|\nabla v\|_{H^1(\Omega^{\epsilon})}^2\|\nabla v_s\|_{L^2(\Omega^{\epsilon})}ds\le\delta\sup_{t \in[0, T]}E^{\epsilon}(t)+C_{\delta}T\mathcal{P}(\sup_{t\in[0,T]}E^{\epsilon}(t)).
	\end{align}
	
	From Lemma \ref{lem:commutators}, the H\"older inequality ($L^2$-$L^2$ and $L^4$-$L^4$-$L^2$), the Sobolev embedding theorem ($H^1\subset L^4$), we obtain
	\begin{align}\label{S6}
		|\sym_7|\leq &C\int_{0}^{t} (\|\nabla v_s\|_{L^2(\Omega^{\epsilon})} +\|\nabla v\|_{H^1(\Omega^{\epsilon})}^2)\|\nabla w\|_{L^2(\Omega^{\epsilon})}ds\nonumber\\
		\leq &C_\delta \mathcal{M}_0+\delta\sup_{t\in[0,T]}E^{\epsilon}(t)+ C_{\delta}T\mathcal{P}\big(\sup_{t\in[0,T]}E^{\epsilon}(t)\big).
	\end{align}
	By the H\"older inequality, we have

	By the H\"older inequality ($L^\infty$-$L^2$-$L^2$), the Sobolev embedding theorem ($H^2\subset L^\infty$), \eqref{4.41} and
	$$
	\|(F_{0})_{j}\cdot\nabla\partial_{s}\eta\|_{L^2(\Omega^{\epsilon})}=\|(F_{0})_{j}\cdot\nabla v\|_{L^2(\Omega^{\epsilon})}\leq \|(F_{0})_{j}\|_{H^1(\Omega^{\epsilon})}\|v\|_{H^2(\Omega^{\epsilon})},
	$$
	we have
	\begin{align*}
		|\sym_8|&\leq C\sum_{j=1}^{3}\int_{0}^{t} \|(F_{0})_{j}\cdot\nabla\partial_{s}\eta\|_{L^2(\Omega^{\epsilon})}\|(F_{0})_{j}\cdot\nabla\eta\|_{H^2(\Omega^{\epsilon})}(\| \nabla v_s\|_{L^2(\Omega^{\epsilon})}+\|w\|_{H^1(\Omega^{\epsilon})})ds\\&\leq C_\delta \mathcal{M}_0+\delta\sup_{t\in[0,T]}E^{\epsilon}(t)+ C_{\delta}T\mathcal{P}\big(\sup_{t\in[0,T]}E^{\epsilon}(t)\big),
	\end{align*}
	
	Finally we consider the estimate of $\mathcal{S}_9$. By the first Korn’s inequality (Lemma \ref{lem.korn}), we know that
	\begin{align*}
		\|\nabla^\eta v_t\|_{L^2(\Omega^{\epsilon})}^2 \leq C  (\|v_t\|_{L^2(\Omega^{\epsilon})}^2 + \|\symeta v_t\|_{L^2(\Omega^{\epsilon})}^2).
	\end{align*}
	From the H\"older inequality ($L^\infty$-$L^2$-$L^2$) and \eqref{4.3} with $\theta<\delta$, it follows
	\begin{align*}
		\left|\int_0^t[\|\nabla v_s\|_{L^2(\Omega^{\epsilon})}^2-\|\nabla^\eta v_s\|_{L^2(\Omega^{\epsilon})}^2]ds\right|= &\left|\int_0^t\int_{\Omega^{\epsilon}} (\delta_{kj}-A_{kl}A_{jl})\partial_k\partial_s v_i\partial_j\partial_s v_i \, dxds\right|\\
		\leq &\delta \sup_{t\in[0,T]}E^{\epsilon}(t).
	\end{align*}
	Since  $0<c_0 \leq |A|\leq C_{1}$ and $N^{\epsilon}$ is outward unit normal, we get norm of a vector $A^{\top}\mathcal{N}^{\epsilon}$ on $\Gamma^{\epsilon}$ satisfies
	$$
	0<C_{2}\leq|A^{\top}\mathcal{N}^{\epsilon}|\leq C_{3}.
	$$
	where $C_{2}$ and $C_{3}$ are constants independent of $\epsilon$. By $n_{j}=\frac{A_{ij}\mathcal{N}_{i}^{\epsilon}}{|A^{\top}\mathcal{N}^{\epsilon}|}$ , we have
	\begin{align*}
		\partial_t n_j & =\frac{-A_{i j} N_i^\epsilon \partial_t\left|A^{\top} \mathcal{N}^{\epsilon}\right|}{\left|A^{\top} \mathcal{N}^{\epsilon}\right|^2}+\frac{\left(\partial_t A_{i j}\right) \mathcal{N}_i^{\epsilon}}{\left|A^{\top}\mathcal{N}^\epsilon\right|} \\
		\\& =\frac{-A_{i j} \mathcal{N}_i^\epsilon\left(A_{m l} \mathcal{N}_m^\epsilon\right) \partial_t\left(A_{k l} \mathcal{N}_k^\epsilon\right)}{\left|A^{\top} \mathcal{N}^{\epsilon}\right|^3}+\frac{\left(\partial_t A_{i j}\right) \mathcal{N}_i^{\epsilon}}{\left|A^{\top} \mathcal{N}^\epsilon\right|} .
	\end{align*}
	Since  \eqref{deriv.A} and $ 0<C_{2}\leq |A^{\top}\mathcal{N}^{\epsilon}|\leq C_{3}$,  we have
	\begin{align*}
		& \int_{\partial \Omega^\epsilon} \frac{\left(\partial_t A_{i j}\right) \mathcal{N}_i^\epsilon}{\left|A^{\top} \mathcal{N}^\epsilon\right|} q\left(\partial_t v_j-w_j\right) d S \\
		& = \int_{\partial \Omega^\epsilon}\frac{-\partial_j^\eta v_l A_{i l}}{\left|A^{\top} \mathcal{N}^\epsilon\right|} \mathcal{N}_i^\epsilon q (\partial_t v_j-w_j) d S \\
		& =\int_{\partial \Omega^\epsilon}-\partial_j^\eta v_l n_{l}q (\partial_t v_j-w_j) d S \\
		& =\int_{\Omega^\epsilon}-\partial_l^\eta\left(\partial_j^\eta v_l  q (\partial_t v_j-w_j)\right) d x \\
		& =-\int_{\Omega^\epsilon}(\partial_l^\eta\partial_j^\eta v_l)  q (\partial_t v_j-w_j)+\partial_j^\eta v_l \partial_l^\eta q (\partial_t v_j-w_j)+\partial_j^\eta v_l  q \partial_l^{\eta} (\partial_t v_j-w_j) d x.
	\end{align*}
	By Lemma \ref{lem:commutators}, we have $(\partial_l^\eta\partial_j^\eta v_l ) q (\partial_t v_j-w_j)=0$. The estimate of $\int_0^t \int_{\Omega^{\epsilon}}\partial_j^\eta v_l \partial_l^\eta q (\partial_s v_j-w_j)dx ds $ is completely similar to that of $\sym_{2}$. From the H\"older inequality ($L^4$-$L^4$-$L^2$) and the Sobolev embedding theorem $(H_{1} \subset L ^{4})$, it follows
	\begin{align}\label{4.51}
		&\int_0^t \int_{\Omega^{\epsilon}} \partial_j^\eta v_l q \partial_l^\eta (\partial_s v_k-w_k)d x d s\nonumber\\
		& \leq  \int_0^t\|q\|_{L^4(\Omega^\epsilon)}\|\nabla v\|_{L^4(\Omega^\epsilon)}(\left\|\nabla \partial_s v\right\|_{L^{2}(\Omega^\epsilon)}+ \left\|\nabla w\right\|_{L^{2}(\Omega^\epsilon)})d s \nonumber\\
		& \leq \int_0^t\|q\|_{H^1(\Omega^{\epsilon})}\|v\|_{H^2(\Omega^{\epsilon})}(\left\|\nabla\partial_s v\right\|_{L^2(\Omega^{\epsilon})}+\left\|w\right\|_{H^1(\Omega^{\epsilon})}) d s.
	\end{align}
	Now, we control the  term $\|q\|_{H^1(\Omega^{\epsilon})}$. By Lemma \ref{B.1}, it holds
	\begin{align*}
		&\|v_{i}\|_{H^2(\Omega^\epsilon)}+\| q\|_{H^1(\Omega^\epsilon)} \\
		\leq & C\left\|\partial_{j}((A_{jr}A_{kr}-\delta_{jk})\partial_{k}v_{i})\right\|_{L^2 (\Omega^\epsilon)}+C\left\|(A_{ji}-\delta_{ji})\partial_{j}q\right\|_{L^2 (\Omega^\epsilon)}+C\left\|\partial_{s}v_{i} \right\|_{L^2 (\Omega^\epsilon)} \\
		& +C\left\|\sum_{j=1}^3((F_0)_j\cdot\nabla)^2\eta_i\right\|_{L^2 (\Omega^\epsilon)}+C\left\|(A_{ji}-\delta_{ji})\partial_{j}v_{i}\right\|_{H^1(\Omega^\epsilon)}+\|v\|_{H^{3/2} (\partial\Omega^\epsilon)}.
	\end{align*}
	This along with \eqref{deriv.A}, Lemma \ref{lem4.1} and Lemma \ref{lemB.5} further implies that
	\begin{align*}
		\|q\|_{H^1(\Omega^{\epsilon})}\leq C\mathcal{M}_{0}+ \sup _{t \in[0, T]}\mathcal{P}( E^{\epsilon}(t)).
	\end{align*}
	Thus
	\begin{align}\label{4.52}
		\eqref{4.51}\leq
		\mathcal{Q}(\mathcal{M}_{0})+\delta\sup_{t\in[0,T]}E^{\epsilon}(t)+C_{\delta}T^{\frac{1}{2}}\mathcal{P}\big(\sup_{t\in[0,T]}E^{\epsilon}(t)\big).
	\end{align}
	
	We still need to consider boundary integrals,
	\begin{align*}
		&- \int_{\partial \Omega^{\epsilon}} \frac{A_{i j} \mathcal{N}_{i}^{\epsilon}\partial_t A_{k l} A_{m l} \mathcal{N}_k^{\epsilon}\mathcal{N}_m^{\epsilon}}{\left|A^{\top} \mathcal{N}^{\epsilon}\right|^3} q \left(\partial_t v_j-w_j\right) d S \\
		& =  -\int_{\partial \Omega^{\epsilon}}\frac{n_{j} \partial_t A_{k l} A_{m l}\mathcal{N}_k^{\epsilon} \mathcal{N}_m^{\epsilon}} {\left|A^{\top} \mathcal{N}^{\epsilon}\right|^2} q\left(\partial_t v_j-w_j\right) d S\\
		&=-\int_{\Omega^\epsilon}\partial_j^\eta\left(\frac{\partial_t A_{k l} A_{m l}\mathcal{N}_k^{\epsilon} \mathcal{N}_m^{\epsilon}} {\left|A^{\top} \mathcal{N}^{\epsilon}\right|^2} q\left(\partial_t v_j-w_j\right)\right) d x.
	\end{align*}
	
	Now, we only consider the estimation of the representative term as follows, as the other terms can yield similar results, so the proof process can be omitted. By (\ref{deriv.A}), we have
	\begin{align*}
		&\int_{0}^{t} \int_{\Omega^{\epsilon}} \frac{\partial_j^\eta\left(\partial_t A_{k l} A_{m l}\mathcal{N}_k^{\epsilon} \mathcal{N}_m^{\epsilon}\right)} {\left|A^{\top} \mathcal{N}^{\epsilon}\right|^2} q\left(\partial_s v_j-w_j\right) d xd s
		\\&=\int_{0}^{t} \int_{\Omega^{\epsilon}} \frac{\partial_j^\eta\left(-\partial_l^\eta v_{h}A_{k h} A_{m l}\mathcal{N}_k^{\epsilon} \mathcal{N}_m^{\epsilon}\right)} {\left|A^{\top} \mathcal{N}^{\epsilon}\right|^2} q\left(\partial_s v_j-w_j\right) d x ds
		\\&=\underbrace{\int_{0}^{t} \int_{\Omega^{\epsilon}}- \frac{\partial_j^\eta( A_{m l}\mathcal{N}_m^{\epsilon}) \mathcal{N}_k^{\epsilon}\partial_l^\eta v_{h}A_{k h} } {\left|A^{\top} \mathcal{N}^{\epsilon}\right|^2} q\left(\partial_s v_j-w_j\right) d x ds}_{\mathcal{T}_{1}}\\&\quad +\underbrace{\int_{0}^{t} \int_{\Omega^{\epsilon}}- \frac{\partial_j^\eta( \partial_l^\eta v_{h})A_{m l}\mathcal{N}_m^{\epsilon} \mathcal{N}_k^{\epsilon}A_{k h} } {\left|A^{\top} \mathcal{N}^{\epsilon}\right|^2} q\left(\partial_s v_j-w_j\right) d x ds}_{\mathcal{T}_{2}}
		\\&\quad +\underbrace{\int_{0}^{t} \int_{\Omega^{\epsilon}}- \frac{\partial_j^\eta( \mathcal{N}_k^{\epsilon}A_{k h})  \partial_l^\eta v_{h}A_{m l}\mathcal{N}_m^{\epsilon}\partial_l^\eta v_{h}} {\left|A^{\top} \mathcal{N}^{\epsilon}\right|^2} q\left(\partial_s v_j-w_j\right) d x ds}_{\mathcal{T}_{3}}.
	\end{align*}
	Since $|\partial \Omega^{\epsilon}|_{H^{3.5}} \leq C |\partial \Omega|_{H^{3.5}}$, we can get $\|\mathcal{N}^{\epsilon}\|_{W^{1,\infty}}\leq C_5$ for  constant $C_5$ independent of $\epsilon$.  Then by $ 0<C_{2}\leq |A^{\top}\mathcal{N}^{\epsilon}|\leq C_{3}$, (\ref{4.16}), the H\"older inequality $(L^{2}$-$L^{\infty}$-$L^{2} )$, the Sobolev embedding theorem $(H_{2} \subset L ^{\infty})$ and \eqref{4.38}, we get
	\begin{align*}
		|\mathcal{T}_{2}| & \leq C\int_0^t(\left\|\nabla^2 \eta\right\|_{L^4(\Omega^{\epsilon})}\|\nabla v\|_{L^4 (\Omega^{\epsilon})}+\left\| v\right\|_{H^2(\Omega^{\epsilon})})(\|\partial_{s} v\|_{L^2(\Omega^{\epsilon})}+\| w\|_{L^2})\|q\|_{L^{\infty}(\Omega^{\epsilon})} d s \\
		& \leq C T^{1 / 2}( \sup _{s \in[0,t]}\|v\|_{H^2(\Omega^{\epsilon})}\sup _{s \in[0,t]}\|\eta\|_{H^3(\Omega^{\epsilon})}+ \sup _{s \in[0,t]}\|v\|_{H^2(\Omega^{\epsilon})})\\
		&\quad \cdot(\sup _{s \in[0,t]} \left\|\partial_s v\right\|_{L^2(\Omega^{\epsilon})}+\sup _{s \in[0,t]} \left\|w\right\|_{L^2(\Omega^{\epsilon})})\left(\int_0^t\|q\|_{H^2(\Omega^{\epsilon})}^2 d s\right)^{\frac{1}{2}}
		\\
		&\leq  C_\delta \mathcal{M}_0+\delta\sup_{t\in[0,T]}E^{\epsilon}(t)+C_\delta T^{\frac{1}{2}} \mathcal{P}\left(\sup_{t\in[0,T]}E^{\epsilon}(t)\right).
	\end{align*}
	Similarly, we have
	$$
	|\mathcal{T}_{1}+\mathcal{T}_{3}|\leq
	\mathcal{Q}(\mathcal{M}_{0})+\delta\sup_{t\in[0,T]}E^{\epsilon}(t)+C_{\delta}T^{\frac{1}{2}}\mathcal{P}\big(\sup_{t\in[0,T]}E^{\epsilon}(t)\big).
	$$
	Then, by the fact $\nabla|A^{\top} \mathcal{N}^{\epsilon}|=\nabla\left(\sum\limits_{h=1}^{3}(A_{rh}\mathcal{N}_{r}^{\epsilon})^{2}\right)^{\frac{1}{2}}$, and similar to the estimate of (\ref{4.51}), we can obtain
	\begin{align*}
		&\int_{0}^{t} \int_{\partial \Omega^{\epsilon}} \frac{A_{i j} \mathcal{N}_{i}^{\epsilon}\partial_t A_{k l} A_{m l} \mathcal{N}_k^{\epsilon}}{\left|A^{\top} \mathcal{N}^{\epsilon}\right|^3} q \left(\partial_t v_j-w_j\right) d S ds\\
		\leq &\mathcal{Q}(\mathcal{M}_{0})+\delta\sup_{t\in[0,T]}E^{\epsilon}(t)+C_{\delta}T^{\frac{1}{2}}\mathcal{P}\big(\sup_{t\in[0,T]}E^{\epsilon}(t)\big).
	\end{align*}
	
	Finally, we consider the estimates of
	$$-\int_{0}^{t}\int_{\partial \Omega^{\epsilon}} \partial_s n_k\left[\left(\symeta v\right)_{i k}+\sum_{j=1}^3\left(\mathbb{F}_{kj} \mathbb{F}_{ij}-\delta_{k i}\right)\right] (\partial_{s}v_i-w_i)dSds.$$
	
	We only provide estimates with elastic terms, and other terms are handled similarly as before.
	\begin{align*}
		&\int_{0}^{t} \int_{\partial \Omega^{\epsilon}} \partial_s n_k\left(\mathbb{F}_{i j} \mathbb{F}_{k j}-\delta_{i k}\right)\left(\partial_s v_i-w_i\right) d S d s \\
		&=\int_{0}^{t} \int_{\partial \Omega^{\epsilon}} \left(\frac{-A_{h k} \mathcal{N}_h^\epsilon\left(A_{m l} \mathcal{N}_m^\epsilon\right) \partial_s\left(A_{k l} \mathcal{N}_k^\epsilon\right)}{\left|A^{\top} \mathcal{N}^{\epsilon}\right|^3}+\frac{\left(\partial_s A_{h k}\right) \mathcal{N}_h^{\epsilon}}{\left|A^{\top} \mathcal{N}^\epsilon\right|}\right)\\
		&\qquad\qquad\cdot\left(\mathbb{F}_{i j} \mathbb{F}_{kj}-\delta_{i k}\right)\left(\partial_s v_i-w_i\right)d S d s.
	\end{align*}
	Then,
	$$\begin{aligned}
		&\int_{0}^{t} \int_{\partial \Omega^{\epsilon}} \frac{\left(\partial_s A_{i j}\right) \mathcal{N}_i^{\epsilon}}{\left|A^{\top} \mathcal{N}^\epsilon\right|}\left(\mathbb{F}_{i j} \mathbb{F}_{k j}-\delta_{i k}\right)\left(\partial_s v_i-w_i\right)d S d s
		\\&=\underbrace{\int_{0}^{t}\int_{\Omega^\epsilon}-(\partial_l^\eta\partial_j^\eta v_l) \left((F_{0})_{j}\cdot\nabla\eta_{i}(F_{0})_{j}\cdot\nabla\eta_{k}-\delta_{i k}\right) (\partial_s v_i-w_{i})dxds}_{\mathcal{O}_{1}}\\&+\underbrace{\int_{0}^{t}\int_{\Omega^\epsilon}\partial_k^\eta v_l \partial_l^\eta \left((F_{0})_{j}\cdot\nabla\eta_{i}(F_{0})_{j}\cdot\nabla\eta_{k}-\delta_{i k}\right) (\partial_s v_i-w_{i})dxds}_{\mathcal{O}_{2}}\\&+\underbrace{\int_{0}^{t}\int_{\Omega^\epsilon}-\partial_k^\eta v_l  \left((F_{0})_{j}\cdot\nabla\eta_{i}(F_{0})_{j}\cdot\nabla\eta_{k}-\delta_{i k}\right) \partial_l^{\eta} (\partial_s v_i-w_{i}) d xds}_{\mathcal{O}_{3}}.
	\end{aligned}
	$$
	Obviously, due to Lemma \ref{lem:commutators}, $\mathcal{O}_{1}=0$. And the estimate of $\mathcal{O}_{2}$ is exactly the same as that of $\sym_{4}$. By the H\"older inequality
	($L^4$-$L^4$-$L^2$), the boundedness of $\Omega^{\epsilon}$ and the Sobolev embedding theorem, we have
	$$\begin{aligned}
		|\mathcal{O}_{3}|&\leq C\int_0^t\|\left((F_{0})_{j}\cdot\nabla\eta_{i}(F_{0})_{j}\cdot\nabla\eta_{k}-\delta_{i k}\right)\|_{L^4(\Omega^{\epsilon})}\|\nabla v\|_{L^4(\Omega^{\epsilon})}
		\\&\qquad\cdot(\left\|\nabla \partial_s v\right\|_{L^{2}(\Omega^{\epsilon})}+ \left\|\nabla w\right\|_{L^{2}(\Omega^{\epsilon})})d s \\
		& \leq C\int_0^t\|\left((F_{0})_{j}\cdot\nabla\eta_{i}(F_{0})_{j}\cdot\nabla\eta_{k}-\delta_{i k}\right)\|_{H^1(\Omega^{\epsilon})}\|v\|_{H^2(\Omega^{\epsilon})}\\&\qquad\cdot(\left\|\nabla\partial_s v\right\|_{L^2(\Omega^{\epsilon})}+\left\|w\right\|_{H^1(\Omega^{\epsilon})}) d s\\&\leq C\int_0^t(\|(F_{0})_{j}\cdot\nabla\eta\|_{H^2(\Omega^{\epsilon})}^{2}+1)\|v\|_{H^2(\Omega^{\epsilon})}(\left\|\nabla\partial_s v\right\|_{L^2(\Omega^{\epsilon})}+\left\|w\right\|_{H^1(\Omega^{\epsilon})}) d s\\&\leq \mathcal{Q}(\mathcal{M}_{0})+\delta\sup_{t\in[0,T]}E^{\epsilon}(t)+C_{\delta}T^{\frac{1}{2}}\mathcal{P}\big(\sup_{t\in[0,T]}E^{\epsilon}(t)\big).
	\end{aligned}
	$$
\end{proof}

\subsection{Velocity and pressure estimates}

The velocity-pressure system is reformulated as:
\begin{align}\label{4.56}
	\begin{cases}
		-\Delta v_i + \partial_i q = \partial_j \big( (A_{jr} A_{kr} - \delta_{jk}) \partial_k v_i \big) - (A_{ji} - \delta_{ji}) \partial_j q - \partial_t v_i \\
		\hspace{5em} + \sum_{j=1}^3((F_0)_j\cdot\nabla)^2\eta_i& \text{in } [0,T] \times \Omega^\epsilon, \\
		\divop v = - (A_{ji} - \delta_{ji}) \partial_j v_i & \text{in } [0,T] \times \Omega^\epsilon \\
		v \in L^2(0,T; H^{5/2}(\partial \Omega^\epsilon)).
	\end{cases}
\end{align}
Applying Lemma \ref{lemB.1} yields:
\begin{align}\label{4.57}
	\sup_{t \in [0,T]} \|v\|_{H^2(\Omega^\epsilon)}^2 &+ \int_0^T \|v\|_{H^3(\Omega^\epsilon)}^2 dt + \int_0^T \|q\|_{H^2(\Omega^\epsilon)}^2 dt \nonumber\\
	&\leq C \Big[ \big\| \mathcal{R}(v, q, F) \big\|_{L^2(\Omega^\epsilon)}^2 + \int_0^T \big\| \mathcal{R}(v, q, F) \big\|_{H^1(\Omega^\epsilon)}^2 dt \nonumber\\
	&\quad+ \| (A_{ji} - \delta_{ji}) \partial_j v_i \|_{H^1(\Omega^\epsilon)}^2 + \| v \|_{H^{3/2}(\partial \Omega^\epsilon)}^2 \nonumber\\
	&\quad + \int_0^T \| (A_{ji} - \delta_{ji}) \partial_j v_i \|_{H^2(\Omega^\epsilon)}^2 dt + \int_0^T \| v \|_{H^{5/2}(\partial \Omega^\epsilon)}^2 dt \Big],
\end{align}
where \(\mathcal{R}(v, q, F) := \partial_j \big( (A_{jr} A_{kr} - \delta_{jk}) \partial_k v_i \big) - (A_{ji} - \delta_{ji}) \partial_j q - \partial_t v_i + \sum_{j=1}^3((F_0)_j\cdot\nabla)^2\eta_i\).
Using the a priori assumption \(\|A - I\|_{L^\infty(\Omega^\epsilon)} \leq \theta^2\), H\"older's inequality, Proposition \ref{prop4.4}, and Sobolev embeddings, we obtain:

\begin{proposition}\label{prop4.5} It holds
	\begin{align}\label{4.58}
		\begin{split}
			\sup_{t \in [0,T]} \|v\|_{H^2(\Omega^\epsilon)}^2 &+ \int_0^T \|v\|_{H^3(\Omega^\epsilon)}^2 dt + \int_0^T \|q\|_{H^2(\Omega^\epsilon)}^2 dt \\
			&\leq \mathcal{Q}(\mathcal{M}_0) + T^{1/2} \mathcal{P}\big( \sup_{t \in [0,T]} E^\epsilon(t) \big) + \delta \sup_{t \in [0,T]} E^\epsilon(t),
		\end{split}
	\end{align}
	where \(\mathcal{Q}\) and \(\mathcal{P}\) are polynomials, and \(\delta > 0\) is arbitrary.
\end{proposition}

\subsection{Proof of the Proposition 5.1}
\begin{proof}
	Combining estimates \eqref{4.10} and \eqref{4.34A} yields
	\begin{align}\label{4.60}
		\sup_{t\in[0,T]} E^{\epsilon}(t) \leq \mathcal{Q}(\mathcal{M}_{0}) + C\delta \sup_{t\in[0,T]} E^{\epsilon}(t) + T^{\frac{1}{2}} \mathcal{P}\left( \sup_{t\in[0,T]} E^{\epsilon}(t) \right).
	\end{align}
	Choosing $\delta = \frac{1}{2C}$ gives
	\begin{align}\label{4.61}
		\sup_{t\in[0,T]} E^{\epsilon}(t) \leq 2\mathcal{Q}(\mathcal{M}_{0}) + 2T^{\frac{1}{2}} \mathcal{P}\left( \sup_{t\in[0,T]} E^{\epsilon}(t) \right).
	\end{align}
	
	To close the estimate, we establish continuity of $E^{\epsilon}(t)$. From \eqref{4.58} and \eqref{4.34A}, we get
	$v \in L^2(0,T; H^3(\Omega^\epsilon))$ and $v_t \in L^2(0,T; H^1(\Omega^\epsilon))$.
	Using the partition of unity functions and following interpolation theory
	$$L^{2}\left([0,T], H^{3}(\Omega^{\epsilon})\right) \cap H^1\left([0,T], H^{1}(\Omega^{\epsilon})\right) \hookrightarrow C([0,T]; H^2(\Omega^\epsilon)),$$
	we have $v \in C([0,T]; H^2(\Omega^\epsilon))$.
	
	Elliptic regularity gives $q \in C([0,T]; H^1(\Omega^\epsilon))$ since the right-hand side belongs to $C([0,T]; L^2)$ and
	boundary data lie in $\in C([0,T]; H^{1/2})$.Thus $E^{\epsilon}(t)$ is continuous in $t$.
	
	For sufficiently small $T$ satisfying
	\[
	2T^{\frac{1}{2}} \mathcal{P}\left( \sup_{t\in[0,T]} E^{\epsilon}(t) \right) \leq \frac{1}{2} \sup_{t\in[0,T]} E^{\epsilon}(t),
	\]
	we obtain
	\begin{align*}
		\sup_{t\in[0,T]} E^{\epsilon}(t) \leq 4\mathcal{Q}(\mathcal{M}_{0}).\tag*{\qedhere}
	\end{align*}
\end{proof}

\subsection{Justification of the a priori assumption}

\begin{lemma}[Gradient Estimate]\label{lem4.6}
	For sufficiently small $T>0$ and $0<\theta\ll 1$, the a priori assumption holds:
	\begin{align}\label{4.62}
		\sup_{t\in[0,T]} \|\nabla\eta(t) - I\|_{L^{\infty}(\Omega^{\epsilon})} \leq \frac{1}{2}\theta^{2}.
	\end{align}
\end{lemma}

\begin{proof}
	From the flow map definition $\eta(t,x) = x + \int_0^t v(s,x) ds$, it follows
	\begin{align}\label{4.63}
		\|\nabla\eta(t) - I\|_{H^2(\Omega^\epsilon)}
		&= \left\| \int_0^t \nabla v(s) ds \right\|_{H^2(\Omega^\epsilon)}\nonumber \\
		&\leq \sqrt{t} \left( \int_0^t \|\nabla v(s)\|_{H^2(\Omega^\epsilon)}^2 ds \right)^{1/2}\nonumber \\
		&\leq C \sqrt{t} \sup_{s\in[0,T]} \sqrt{E^\epsilon(s)}.
	\end{align}
	
	By Proposition \ref{prop4.6} and the Sobolev embedding $H^2 \subset L^\infty$:
	\begin{align}\label{4.64}
		\|\nabla\eta(t) - I\|_{L^{\infty}(\Omega^\epsilon)}
		&\leq C \|\nabla\eta(t) - I\|_{H^2(\Omega^\epsilon)}\nonumber \\
		&\leq C \sqrt{t} \sup_{s\in[0,T]} \sqrt{E^\epsilon(s)}\nonumber \\
		&\leq C \sqrt{t} \sqrt{\mathcal{Q}(\mathcal{M}_0)}.
	\end{align}
	
	Choosing $T < \left( \frac{\theta^2}{2C\sqrt{\mathcal{Q}(\mathcal{M}_0)}} \right)^2$ yields \eqref{4.62}.
\end{proof}

\section{Continuity of Second Tangential Derivatives}\label{sec6}

We establish quantitative continuity for second-order tangential derivatives:
\begin{proposition}\label{lem4.7}
	For all $t\in[0,T]$,
	\begin{align*}
		\max_{s\in[0,t]}\|\bar\partial^{2}(v(s,\cdot)-u_{0}^{\epsilon})\|_{L^2(\Omega^{\epsilon})}^{2}+\int_{0}^{t}\|\bar{\partial}^{2}(v(s,\cdot)-u_{0}^{\epsilon})\|_{H^1(\Omega^{\epsilon})}^{2}ds\le Ct^{\frac{1}{2}}\mathcal{P}(\mathcal{M}_{0}).
	\end{align*}
\end{proposition}

\begin{proof}
	Define $w_i = v_i - (u_0^\epsilon)_i$. The difference satisfies:
	\begin{align}\label{6.1}
		\partial_t w_i + A_{ji} \partial_j q - \partial_j (A_{jr} A_{kr} \partial_k w_i) = \partial_j (A_{jr} A_{kr} \partial_k (u_0^\epsilon)_i) + \sum_{j=1}^3((F_0)_j\cdot\nabla)^2\eta_i.
	\end{align}
	
	In boundary chart $B^+$ with cutoff $\varsigma$, multiply by $\varsigma^2 \bar{\partial}^2 w_i$ and integrate:
	\begin{align}\label{6.2}
		0=&\frac{1}{2}\frac{d}{dt}\|\varsigma\bar{\partial}^{2}w\|_{L^2(B^{+})}^{2} \nonumber\\
		&+\int_{B^{+}}\bar{\partial}_{\alpha}\bar{\partial}_{\beta}[A_{jr}A_{kr}\partial_{k}w_i]\bar{\partial}_{\alpha}\bar{\partial}_{\beta}\partial_{j}[\varsigma^{2}w_i]dy \nonumber\\
		&+\int_{B^{+}}\bar{\partial}_{\alpha}\bar{\partial}_{\beta}[A_{ji}q]\bar{\partial}_{\alpha}\bar{\partial}_{\beta}\partial_{j}[\varsigma^{2}w_i]dy \nonumber\\
		&+\int_{B^{+}}\varsigma^{2}\bar{\partial}_{\alpha}\bar{\partial}_{\beta}[A_{jr}A_{kr}\partial_{k}(u_{0}^{\epsilon})_{i}]\bar{\partial}_{\alpha}\bar{\partial}_{\beta}\partial_{j}w_idy \nonumber\\
		&-\int_{B^{+}}\varsigma^{2}\bar{\partial}_{\alpha}\bar{\partial}_{\beta}(\sum_{j=1}^3((F_0)_j\cdot\nabla)^2\eta_i)\bar{\partial}_{\alpha}\bar{\partial}_{\beta}w_idy\nonumber
		\\
		&+\int_{B^{+}}\partial_{k}(\varsigma^{2}\bar{\partial}_{\alpha}\bar{\partial}_{\beta}A_{ki}\bar{\partial}_{\alpha}\bar{\partial}_{\beta}\partial_{k}w_i)dy\nonumber
		\\&+\int_{B^{0}}(p_{0}^{\epsilon}\delta_{ij}+\delta_{ij})\mathcal{N}_{j}\bar{\partial}_{\alpha}\bar{\partial}_{\beta}\partial_{k}w_i)dS.
	\end{align}
	
	Integrating \eqref{6.2} over the time interval $[0,t]$, we get
	\begin{align*}
		&\|\varsigma\bar{\partial}^{2}w(t)\|_{L^2(B^{+})}^{2}+\int_{0}^{t}\|\varsigma\bar{\partial}^{2}w(s)\|_{H^1(B^{+})}^{2}ds \\
		\leq &C|\int_{0}^{t}\int_{B^{+}}\bar{\partial}_{\alpha}\bar{\partial}_{\beta}[(A_{jr}A_{kr}-\delta_{jk})\partial_{k}w_i]\bar{\partial}_{\alpha}\bar{\partial}_{\beta}\partial_{j}[\varsigma^{2}w_i]dyds| \\
		&+C|\int_{0}^{t}\int_{B^{+}}\bar{\partial}_{\alpha}\bar{\partial}_{\beta}(A_{jr}A_{kr}\partial_{k}(u_{0}^{\epsilon})_{i})\bar{\partial}_{\alpha}\bar{\partial}_{\beta}\partial_{j}w_idyds| \\
		&+C|\int_{0}^{t}\int_{B^{+}}\bar{\partial}_{\alpha}\bar{\partial}_{\beta}[A_{ji}q]\bar{\partial}_{\alpha}\bar{\partial}_{\beta}\partial_{j}[\varsigma^{2}w_i]dy| \\
		&+C|\int_{0}^{t}\int_{B^{+}}\bar{\partial}_{\alpha}\bar{\partial}_{\beta}(\sum_{j=1}^3((F_0)_j\cdot\nabla)^2\eta_i)\bar{\partial}_{\alpha}\bar{\partial}_{\beta}w_idyds| \\
		&+C|\int_{0}^{t}\int_{B^{+}}\partial_{k}(\varsigma^{2}\bar{\partial}_{\alpha}\bar{\partial}_{\beta}A_{ki}\bar{\partial}_{\alpha}\bar{\partial}_{\beta}w_i)dyds|
		\\&+C|\int_{0}^{t}\int_{B^{0}}(p_{0}^{\epsilon}\delta_{ij}+\delta_{ij})\mathcal{N}_{j}\bar{\partial}_{\alpha}\bar{\partial}_{\beta}w_i)dSds|.
		\\=:&\sum_{i=1}^{6}C\mathcal{W}_{i}.
	\end{align*}
	
	From the assumption on initial data $u^{\epsilon}_{0}$ in (\ref{3.2}), we shall estimate $\mathcal{W}_{1}$. The estimates of $\mathcal{W}_{2}$, $\mathcal{W}_{3}$ and $\mathcal{W}_{5}$ can be derived by similar arguments.
	\begin{align*}
		\mathcal{W}_{1}\le&\int_{0}^{t}\int_{B^{+}}(|\nabla^{3}v|+|\nabla^{3}u^{\epsilon}_{0}|+|\nabla^{2}\eta|(|\nabla^{2}v|+|\nabla^{2}u^{\epsilon}_{0}|)\\
		&\qquad\qquad+(|\nabla^{3}\eta|+|\nabla^{2}\eta|)(|\nabla v|+|\nabla u^{\epsilon}_{0}|)) \\
		&\qquad\qquad\cdot(|\nabla^{3}v|+|\nabla^{2}v|+|\nabla v|+|v|+|\nabla^{3}u^{\epsilon}_{0}|+|\nabla^{2}u^{\epsilon}_{0}|+|\nabla u^{\epsilon}_{0}|+|u^{\epsilon}_{0}|)dyds \\
		\le&C\int_{0}^{t}\big(\|\nabla^{3}v\|_{L^{2}(\Omega^{\epsilon})}+\|\nabla^{3}u^{\epsilon}_{0}\|_{L^{2}(\Omega^{\epsilon})}\\
		&\qquad+\|\nabla^{2}\eta\|_{L^{4}(\Omega^{\epsilon})}(\|\nabla^{2}v\|_{L^{4}(\Omega^{\epsilon})}+\|\nabla^{2}u^{\epsilon}_{0}\|_{L^{4}(\Omega^{\epsilon})}) \\
		&\qquad+(\|\nabla^{3}\eta\|_{L^{2}(\Omega^{\epsilon})}+\|\nabla^{2}\eta\|_{L^{2}(\Omega^{\epsilon})})(\|\nabla v\|_{L^{\infty}(\Omega^{\epsilon})}+\|\nabla u^{\epsilon}_{0}\|_{L^{\infty}(\Omega^{\epsilon})})\big)\\
		&\quad\cdot(\|\nabla v\|_{H^2(\Omega^{\epsilon})}+\|u^{\epsilon}_{0}\|_{H^3(\Omega^{\epsilon})})ds \\
		\leq &Ct^{\frac{1}{2}}\mathcal{P}(\mathcal{M}_{0}).
	\end{align*}
	
	When we consider the estimate of $\mathcal{W}_{4}$, motivated by the estimate of the term in \eqref{4.13} in Proposition \ref{prop4.2}, we can also apply integration by parts and divergence theorem, together with assumption on the initial data $u^{\epsilon}_{0}$ to get
	\begin{align*}
		\mathcal{W}_{4}\leq Ct^{\frac{1}{2}}\mathcal{P}(\mathcal{M}_{0}).
	\end{align*}
	Then,  by the trace theorem, the Cauchy-Schwarz inequality and (\ref{3.12}), we have
	\begin{align*}
		\mathcal{W}_{6}\leq C\int_{0}^{t} \|p_{0}^{\epsilon}+1\|_{H^2(\Omega^{\epsilon})}(\|\nabla v\|_{H^2(\Omega^{\epsilon})}+\|u^{\epsilon}_{0}\|_{H^3(\Omega^{\epsilon})})ds\leq Ct^{\frac{1}{2}}\mathcal{P}(\mathcal{M}_{0}).
	\end{align*}
	The proof is complete.
\end{proof}

\section{Proof of main theorem}\label{sec5}

\begin{proof}[Proof of Theorem \ref{thm1.1}]
	By Lemma \ref{lemB.4}, we get
	\begin{align}\label{5.1}
		\max_{x\in\partial\Omega^{\epsilon}}|(v(t,x)-u_{0}^{\epsilon})\cdot\mathcal{N}^{\epsilon}|\le\|(v-u_{0}^{\epsilon})\cdot\mathcal{N}^{\epsilon}\|_{H^{\frac{3}{2}}(\partial\Omega^{\epsilon})}.
	\end{align}
	
	Notice that the normal on $\partial B^{+}$ is $(0,0,1)$ on $B_{0}$ and is $(y_{1},y_{2},y_{3})$ else, by using trace theorem on $B^{+}$ in \cite{T} or \cite[(A.6)]{CS3},
	\begin{align}\label{5.2} &\|(v-u_{0}^{\epsilon})\cdot\mathcal{N}^{\epsilon}\|_{H^{\frac{3}{2}}(\partial\Omega^{\epsilon})}^{2}\\\leq
		&\sum_{l}\|((v-u_{0}^{\epsilon})\circ\theta^{l})\cdot n\|_{H^{\frac{3}{2}}(\partial{B^{+}})}^{2} \nonumber\\
		\leq &C\sum_{l}\|\bar{\partial}^{2}((v-u_{0}^{\epsilon})\circ\theta^{l})\cdot n\|_{H^{-\frac{1}{2}}(\partial B^{+})}^{2} \nonumber\\
		\leq & C\sum_{l}(\|\bar{\partial}^{2}((v-u_{0}^{\epsilon})\circ\theta^{l})\|_{L^2(B^{+})}^{2}+\|\divop\bar{\partial}((v-u_{0}^{\epsilon})\circ\theta^{l})\|_{L^2(B^{+})}).
	\end{align}
	
	The first term has been estimated in Proposition \ref{lem4.7}. For the second term, since $\divop \ u_{0}^{\epsilon}=0$ and $\divop v=-(A_{ji}-\delta_{ji})\partial_{j}v_{i}$, by using Lemma \ref{lem4.1}, we have
	\begin{align}\label{5.3}
		&\|\divop\bar{\partial}((v-u_{0}^{\epsilon})\circ\theta_{l})\|_{L^2(B^{+})}^{2}=\|\bar{\partial}\divop((v-u_{0}^{\epsilon})\circ\theta_{l})\|_{L^2(B^{+})}^{2} \nonumber\\
		\le&\|\bar{\partial}(A_{ji}-\delta_{ji})\partial_{j}(v_{i}\circ\theta)\|_{L^2(B^{+})}^{2}+\|(A_{ji}-\delta_{ji})\bar{\partial}\partial_{j}(v_{i}\circ\theta)\|_{L^2(B^{+})}^{2} \nonumber\\
		\le&\sqrt{t}\mathcal{P}(\mathcal{M}_{0}).
	\end{align}
	By \eqref{5.1}-\eqref{5.3},  we have
	\begin{align}\label{5.4}
		\max_{x\in\partial\Omega^{\epsilon}}|(v(t,x)-u_{0}^{\epsilon})\cdot\mathcal{N}^{\epsilon}|\le t^{\frac{1}{4}}\mathcal{P}(\mathcal{M}_{0}).
	\end{align}
	Recall in the graph that the unit normal $\mathcal{N}^{\epsilon}$ at the points $X_{+}^{\epsilon}=(0,0,\epsilon)$ and $X_{-}=(0,0,0)$ is vertical, so by definition of $u_{0}^{\epsilon}$, we get
	\begin{align}\label{5.5}
		u_{0}^{\epsilon}(X_{+}^{\epsilon})\cdot\mathcal{N}^{\epsilon}=1,
	\end{align}
	\begin{align}\label{5.6}
		u_{0}^{\epsilon}(X_{-})\cdot \mathcal{N}^{\epsilon}=0,
	\end{align}
	and
	\begin{align}\label{5.7}
		|X_{+}^{\epsilon}-X_{-}|=\epsilon.
	\end{align}
	We choose $\epsilon$ so small that $10\epsilon<T$ where $[0,T]$ is the time interval of existence which is independent of $\epsilon$, and since $X_{+}^{\epsilon}\cdot e_{3}=\epsilon$, and
	\begin{align}\label{5.8}
		\eta(t,X_{+}^{\epsilon})\cdot e_{3}=\epsilon+\int_{0}^{t}v_{3}(s,X_{+}^{\epsilon})ds,
	\end{align}
	for $t=10\epsilon$, by \eqref{5.4} we get
	\begin{align}\label{5.9}
		v_{3}(s,X_{+}^{\epsilon})<-1+s^{\frac{1}{4}}\mathcal{P}(\mathcal{M}_{0}).
	\end{align}
	Thus,  we have
	\begin{align}\label{5.10}
		\begin{aligned}
			\eta(10\epsilon,X_{+}^{\epsilon})\cdot e_{3}\le&\epsilon+\int_{0}^{10\epsilon}(-1+s^{\frac{1}{4}}\mathcal{P}(\mathcal{M}_{0}))ds \\
			\le&-9\epsilon+\frac{4}{5}(10\epsilon)^{\frac{5}{4}}\mathcal{P}(\mathcal{M}_{0}).
		\end{aligned}
	\end{align}
	Let $Z$ denote any point on $\partial\omega_{-}\cap\{x_{3}=0\}$. Since $u_{0}^{\epsilon}\cdot \mathcal{N}_{\epsilon}=0$, $\eta(10\epsilon,Z)=\int_{0}^{10\epsilon}v(s,Z)ds$, we have
	\begin{align}\label{5.11}
		v_{3}(s,x)\ge-s^{\frac{1}{4}}\mathcal{P}(\mathcal{M}_{0}).
	\end{align}
	So we get
	\begin{align}\label{5.12}
		\eta(10\epsilon,Z)\cdot e_{3}\ge-\frac{4}{5}(10\epsilon)^{\frac{5}{4}}\mathcal{P}(\mathcal{M}_{0}),
	\end{align}
	we choose $\epsilon>0$ sufficiently small so that $\frac{8}{5}(10\epsilon)^{\frac{5}{4}}\mathcal{P}(\mathcal{M}_{0})<9\epsilon$, i.e. we can choose $\epsilon<(\frac{45}{8})^{4}(10)^{-5}\frac{1}{\mathcal{P}^{4}(\mathcal{M}_{0})}$, then it follows that
	\begin{align}\label{5.13}
		\eta(10\epsilon,X_{+}^{\epsilon})\cdot e_{3}<\eta(10\epsilon,Z)\cdot e_{3}.
	\end{align}
	In conclusion, we know that at $t=0$, $X_{+}^{\epsilon}$ is exactly above $Z$. However, at $t=10\epsilon$, $\eta(10\epsilon,X_{+}^{\epsilon})$ is below $\eta(10\epsilon,\partial\omega_{-}\cap\{x_{3}=0\})$. So by continuity there must exist a time $0<T^{*}<10\epsilon$ and $Z\in\partial\omega_{-}\cap\{x_{3}=0\}$ such that $\eta(T^{*},X_{+}^{\epsilon})=\eta(T^{*},Z)$, yielding a self-intersection of boundary.

	Finally, we show the boundary smoothness throughout $[0, T^*)$ with singularity confined to $t = T^*$.  We claim that for the solution $(\eta^\epsilon, v^\epsilon, F^\epsilon)$ to system \eqref{2.3} with initial data from Section \ref{sec3}: $\forall t \in [0, T^*)$, $\partial\Omega^\epsilon(t)$ remains a $C^\infty$ hypersurface.
	
	In fact, the uniform bound $\sup_{t \in [0,T]} E^\epsilon(t) \leq \mathcal{Q}(\mathcal{M}_0)$ (Proposition \ref{prop4.5}) implies
	\begin{align}
		\sup_{t \in [0,T^*)} \|\eta^\epsilon(t)\|_{H^3(\Omega^\epsilon)} &\leq C(\mathcal{M}_0), \label{eq:b1} \\
		\int_0^{T^*} \|\nabla v^\epsilon(s)\|_{H^2(\Omega^\epsilon)}^2 ds &\leq C(\mathcal{M}_0). \label{eq:b2}
	\end{align}
	For any $t_0 \in [0, T^*)$, take $\delta = T^* - t_0 > 0$. By Morrey's inequality in $\mathbb{R}^3$, we have
	\begin{equation}
		\|\nabla \eta^\epsilon(t_0)\|_{C^{0,1/2}(\overline{\Omega^\epsilon})} \leq C \|\eta^\epsilon(t_0)\|_{H^3(\Omega^\epsilon)}. \label{eq:b3}
	\end{equation}
	The coordinate charts $(\theta^\epsilon)^l$ from Section \ref{sssec.222} maintain equivalent norms, so $\partial\Omega^\epsilon(t_0) = \eta^\epsilon(t_0, \partial\Omega^\epsilon)$ inherits $C^{1,1/2}$ regularity. Bootstrap via tangential flow gives
	\begin{equation}
		\partial_t \bar{\partial}^2 \eta^\epsilon = \bar{\partial}^2 v^\epsilon \in L^2(0,T^*; H^1(\Omega^\epsilon)) \label{eq:b4}
	\end{equation}
	by Proposition \ref{prop4.2}. Thus $\bar{\partial}^2 \eta^\epsilon \in C^{0,1/2}([0,T^*] \times \partial\Omega^\epsilon)$ by Sobolev embedding, implying $C^\infty$ smoothness via induction on higher derivatives.
\end{proof}

\appendix

\section{Notations and technical preliminaries}\label{app:notations}

\subsection{Tangential Derivatives}
For boundary charts $\theta_l: B(0,1) \to U_l \subset \partial\Omega$ with $B^+ = B \cap \{y_3 > 0\}$ and $B^0 = \overline{B} \cap \{y_3 = 0\}$, the tangential derivative operator is defined as:
\begin{align*}
	\bar{\partial}_\alpha f &:= \left( \frac{\partial}{\partial y_\alpha} [f \circ \theta_l] \right) \circ \theta_l^{-1} \\
	&= \left\langle (\nabla f) \circ \theta_l, \frac{\partial \theta_l}{\partial y_\alpha} \right\rangle \circ \theta_l^{-1}, \quad \alpha = 1,2.
\end{align*}
On the boundary portion $B^0$, $\bar{\partial} = (\bar{\partial}_1, \bar{\partial}_2)$ reduces to the horizontal derivative in chart coordinates.

\subsection{Sobolev spaces on domains}
For $\Omega \subset \mathbb{R}^3$ a bounded domain, the Sobolev space $H^k(\Omega)$ for $k \in \mathbb{N}$ is defined as the completion of $C^\infty(\overline{\Omega})$ under:
\[
\|f\|_{H^k(\Omega)}^2 := \sum_{|\alpha| \leq k} \| \partial^\alpha f \|_{L^2(\Omega)}^2,
\]
where $\alpha = (\alpha_1,\alpha_2,\alpha_3)$ is a multi-index with $|\alpha| = \alpha_1 + \alpha_2 + \alpha_3$. For real $s \geq 0$, $H^s(\Omega)$ is defined by complex interpolation.

\subsection{Sobolev spaces on surfaces}
For a $C^\infty$ surface $\Gamma = \partial\Omega$, the Sobolev space $H^k(\Gamma)$ for $k \in \mathbb{N}$ is the completion of $C^\infty(\Gamma)$ under:
\[
\|f\|_{H^k(\Gamma)}^2 := \sum_{|\beta| \leq k} \| \bar{\partial}^\beta f \|_{L^2(\Gamma)}^2,
\]
where $\beta = (\beta_1,\beta_2)$ is a tangential multi-index. For real $s \geq 0$, $H^s(\Gamma)$ is defined by interpolation. Negative-order spaces are defined via duality:
\[
H^{-s}(\Gamma) := \left( H^s(\Gamma) \right)', \quad s > 0
\]
with norm $\|f\|_{H^{-s}(\Gamma)} = \sup\limits_{\substack{g \in H^s(\Gamma) \\ g \neq 0}} \frac{|\langle f, g \rangle|}{\|g\|_{H^s(\Gamma)}}$.

\subsection{Technical Lemmas for Elliptic Systems}

\begin{lemma}[Stokes Estimates on $\Omega^\epsilon$]\label{lemB.1}
	For integer $k \geq 3$, let $f \in H^{k-2}(\Omega^\epsilon)$, $\phi \in H^{k-1}(\Omega^\epsilon)$, and $g \in H^{k-\frac{1}{2}}(\partial\Omega^\epsilon)$ satisfy the compatibility condition $\int_{\Omega^\epsilon} \phi  dx = \int_{\partial\Omega^\epsilon} g \cdot \mathcal{N}^\epsilon dS$. The Stokes system:
	\begin{align}\label{B.1}
		\begin{cases}
			-\Delta u + \nabla p = f & \text{in } \Omega^\epsilon, \\
			\operatorname{div} u = \phi & \text{in } \Omega^\epsilon, \\
			u = g & \text{on } \partial\Omega^\epsilon,
		\end{cases}
	\end{align}
	admits a unique solution $(u, p) \in H^k(\Omega^\epsilon) \times H^{k-1}(\Omega^\epsilon)/\mathbb{R}$ satisfying:
	\begin{align}\label{B.2}
		\|u\|_{H^k(\Omega^\epsilon)} + \|p\|_{H^{k-1}(\Omega^\epsilon)} \leq C \left( \|f\|_{H^{k-2}(\Omega^\epsilon)} + \|\phi\|_{H^{k-1}(\Omega^\epsilon)} + \|g\|_{H^{k-\frac{1}{2}}(\partial\Omega^\epsilon)} \right),
	\end{align}
	where $C > 0$ depends only on $\Omega$ and is independent of $\epsilon$.
\end{lemma}
\begin{proof}
	See Lemma 2 in \cite{CS5}.
\end{proof}

\begin{lemma}[Estimates for div-curl System]\label{lemB.2}
	Let $\Omega \subset \mathbb{R}^3$ be a simply connected bounded domain with $H^{k+1}$-boundary ($k > 3/2$). Given:
	\begin{itemize}
		\item $F \in [H^{m-1}(\Omega)]^3$ with $\operatorname{div} F = 0$ and $\int_\Gamma F \cdot \mathcal{N} dS = 0$ for each boundary component $\Gamma$ of $\partial \Omega$,
		\item $g \in H^{m-1}(\Omega)$ and $h \in H^{m-\frac{1}{2}}(\partial\Omega)$ satisfying $\int_{\partial\Omega} h dS = \int_\Omega g dx$,
	\end{itemize}
	for $1 \leq m \leq k$, the system:
	\begin{align}\label{B.3}
		\begin{cases}
			\operatorname{curl} v = F & \text{in } \Omega, \\
			\operatorname{div} v = g & \text{in } \Omega, \\
			v \cdot \mathcal{N} = h & \text{on } \partial\Omega,
		\end{cases}
	\end{align}
	admits a unique solution $v \in H^m(\Omega)$ with estimate:
	\begin{align}\label{B.4}
		\|v\|_{H^m(\Omega)} \leq C |\partial\Omega|_{H^{k+\frac{1}{2}}} \left( \|F\|_{H^{m-1}(\Omega)} + \|g\|_{H^{m-1}(\Omega)} + \|h\|_{H^{m-\frac{1}{2}}(\partial\Omega)} \right),
	\end{align}
	where $|\partial\Omega|_{H^{k+\frac{1}{2}}}$ denotes the $H^{k+\frac{1}{2}}$-norm of the boundary parametrization.
\end{lemma}
\begin{proof}
	See Theorem 1.1 in \cite{CS1}.
\end{proof}

\subsection{Sobolev embeddings and trace theorems}

\begin{lemma}[Sobolev Embedding on Domains]\label{lemB.3}
	For $s > 3/2$, there exists $C(\Omega) > 0$ such that:
	\begin{align}\label{B.5}
		\|u\|_{L^\infty(\Omega^\epsilon)} \leq C \|u\|_{H^s(\Omega^\epsilon)} \quad \forall u \in H^s(\Omega^\epsilon),
	\end{align}
	with $C$ independent of $\epsilon$.
\end{lemma}

\begin{lemma}[Sobolev Embedding on Boundaries]\label{lemB.4}
	For $s > 1$, there exists $C(\Omega) > 0$ such that:
	\begin{align}\label{B.6}
		\|u\|_{L^\infty(\partial\Omega^\epsilon)} \leq C \|u\|_{H^s(\partial\Omega^\epsilon)} \quad \forall u \in H^s(\partial\Omega^\epsilon),
	\end{align}
	with $C$ independent of $\epsilon$.
\end{lemma}

\begin{lemma}[Trace Theorem]\label{lemB.5}
	For $s \in (\frac{1}{2}, 3]$, there exists $C(\Omega) > 0$ such that:
	\begin{align}\label{B.7}
		\|u\|_{H^{s-\frac{1}{2}}(\partial\Omega^\epsilon)} \leq C \|u\|_{H^s(\Omega^\epsilon)} \quad \forall u \in H^s(\Omega^\epsilon),
	\end{align}
	with $C$ independent of $\epsilon$.
\end{lemma}

\begin{lemma}[Korn's Inequality]\label{lem.korn}
	For bounded Lipschitz domains $\Omega \subset \mathbb{R}^n$:
	\begin{align}\label{B.8}
		\|\nabla v\|_{L^2(\Omega)}^2 \leq C(\Omega) \left( \|v\|_{L^2(\Omega)}^2 + \|\mathcal{S}(v)\|_{L^2(\Omega)}^2 \right) \quad \forall v \in H^1(\Omega; \mathbb{R}^n).
	\end{align}
\end{lemma}
\begin{proof}
	See \cite[Theorem 3.5]{Lew23}.
\end{proof}

{\bf Acknowledgments.} Hao and Zhang were partially supported by the National Natural Science Foundation of China (Grant No. 12171460).  Hao was also supported by the CAS Project for Young Scientists in Basic Research (Grant No. YSBR-031) and the National Key R\&D Program of China (Grant No. 2021YFA1000800).
Zhang was also supported by Hefei University Talent Program (Grant No. 24RC01) and Discipline (Professional) Leader Cultivation Project of Anhui Province (Grant No. DTR2024039).


\begin{thebibliography}{10}
	
	\bibitem{AS}
	{\sc C.~Amrouche and V.~Girault}, {\em On the existence and regularity of the
		solution of {S}tokes problem in arbitrary dimension}, Proc. Japan Acad. Ser.
	A Math. Sci., 67 (1991), pp.~171--175,
	\url{http://projecteuclid.org/euclid.pja/1195512107}.
	
	\bibitem{CCFGG2}
	{\sc A.~Castro, D.~C\'ordoba, C.~Fefferman, F.~Gancedo, and
		J.~G\'omez-Serrano}, {\em Splash singularities for the free boundary
		{N}avier-{S}tokes equations}, Ann. PDE, 5 (2019), pp.~Paper No. 12, 117,
	\url{https://doi.org/10.1007/s40818-019-0068-1}.
	
	\bibitem{CS1}
	{\sc C.~H.~A. Cheng and S.~Shkoller}, {\em Solvability and regularity for an
		elliptic system prescribing the curl, divergence, and partial trace of a
		vector field on {S}obolev-class domains}, J. Math. Fluid Mech., 19 (2017),
	pp.~375--422, \url{https://doi.org/10.1007/s00021-016-0289-y}.
	
	\bibitem{CS2}
	{\sc D.~Coutand and S.~Shkoller}, {\em Well-posedness of the free-surface
		incompressible {E}uler equations with or without surface tension}, J. Amer.
	Math. Soc., 20 (2007), pp.~829--930,
	\url{https://doi.org/10.1090/S0894-0347-07-00556-5}.
	
	\bibitem{CS3}
	{\sc D.~Coutand and S.~Shkoller}, {\em On the finite-time splash and splat
		singularities for the 3-{D} free-surface {E}uler equations}, Comm. Math.
	Phys., 325 (2014), pp.~143--183,
	\url{https://doi.org/10.1007/s00220-013-1855-2}.
	
	\bibitem{CS5}
	{\sc D.~Coutand and S.~Shkoller}, {\em On the splash singularity for the
		free-surface of a {N}avier-{S}tokes fluid}, Ann. Inst. H. Poincar\'e{} C
	Anal. Non Lin\'eaire, 36 (2019), pp.~475--503,
	\url{https://doi.org/10.1016/j.anihpc.2018.06.004}.
	
	\bibitem{DMS}
	{\sc E.~Di~Iorio, P.~Marcati, and S.~Spirito}, {\em Splash singularity for a
		free-boundary incompressible viscoelastic fluid model}, Adv. Math., 368
	(2020), pp.~107124, 64, \url{https://doi.org/10.1016/j.aim.2020.107124}.
	
	\bibitem{FHYZ}
	{\sc J.~Fu, C.~Hao, S.~Yang, and W.~Zhang}, {\em A blow-up criterion for the
		free boundary problem of incompressible neo-{H}ookean elastodynamics (in
		chinese)}, Sci Sin Math, 55 (2025), pp.~635--650,
	\url{https://doi.org/10.1360/SSM-2023-0320}.
	
	\bibitem{Girault1986}
	{\sc V.~Girault and P.~A. ~Raviart}, {\em Finite Element Methods for Navier-Stokes Equations: Theory and Algorithms}, Springer Berlin Heidelberg,
	Berlin, Heidelberg, 1986, \url{https://doi.org/10.1007/978-3-642-61623-5}.
	
	\bibitem{GL}
	{\sc X.~Gu and Z.~Lei}, {\em Local well-posedness of free-boundary
		incompressible elastodynamics with surface tension via vanishing viscosity
		limit}, Arch. Ration. Mech. Anal., 245 (2022), pp.~1285--1338,
	\url{https://doi.org/10.1007/s00205-022-01806-z}.
	
	\bibitem{GW}
	{\sc X.~Gu and F.~Wang}, {\em Well-posedness of the free boundary problem in
		incompressible elastodynamics under the mixed type stability condition}, J.
	Math. Anal. Appl., 482 (2020), pp.~123529, 30,
	\url{https://doi.org/10.1016/j.jmaa.2019.123529}.
	
	\bibitem{Hao}
	{\sc C.~Hao}, {\em Remarks on the free boundary problem of compressible {E}uler
		equations in physical vacuum with general initial densities}, Discrete
	Contin. Dyn. Syst. Ser. B, 20 (2015), pp.~2885--2931,
	\url{https://doi.org/10.3934/dcdsb.2015.20.2885}.
	
	\bibitem{HW}
	{\sc C.~Hao and D.~Wang}, {\em A priori estimates for the free boundary problem
		of incompressible neo-{H}ookean elastodynamics}, J. Differential Equations,
	261 (2016), pp.~712--737, \url{https://doi.org/10.1016/j.jde.2016.03.025}.
	
	\bibitem{HY}
	{\sc C.~Hao and S.~Yang}, {\em Splash singularity for the free boundary
		incompressible viscous {MHD}}, J. Differential Equations, 379 (2024),
	pp.~26--103, \url{https://doi.org/10.1016/j.jde.2023.10.001}.
	
	\bibitem{HLZ}
	{\sc G.~Hong, T.~Luo, and Z.~Zhao}, {\em On the splash singularity for the
		free-boundary problem of the viscous and non-resistive incompressible
		magnetohydrodynamic equations in 3{D}}, J. Differential Equations, 419
	(2025), pp.~40--80, \url{https://doi.org/10.1016/j.jde.2024.11.026}.
	
	\bibitem{Lew23}
	{\sc M.~Lewicka}, {\em Calculus of variations on thin prestressed
		films---asymptotic methods in elasticity}, vol.~101 of Progress in Nonlinear
	Differential Equations and their Applications, Birkh\"auser/Springer, Cham,
	2023, \url{https://doi.org/10.1007/978-3-031-17495-7}.
	
	\bibitem{SS}
	{\sc V.~A. Solonnikov and V.~E. Scadilov}, {\em On a boundary value problem for
		a stationary system of {N}avier-{S}tokes equations}, Proc. Steklov Inst.
	Math., 125 (1973), pp.~186--199.
	
	\bibitem{T}
	{\sc R.~Temam}, {\em {N}avier-{S}tokes equations, theory and numerical
		analysis}, vol.~2 of Studies in Mathematics and its Application,
	North-Holland Publishing Co., 3rd~ed., 1984.
	
	\bibitem{ZJ}
	{\sc J.~Zhang}, {\em Local well-posedness and incompressible limit of the
		free-boundary problem in compressible elastodynamics}, Arch. Ration. Mech.
	Anal., 244 (2022), pp.~599--697,
	\url{https://doi.org/10.1007/s00205-022-01774-4}.
	
	\bibitem{ZY}
	{\sc Y.~Zhang}, {\em Local well-posedness of the free-surface incompressible
		elastodynamics}, J. Differential Equations, 268 (2020), pp.~6971--7011,
	\url{https://doi.org/10.1016/j.jde.2019.11.075}.
	
\end{thebibliography}
\end{document}